%% file: ims-template.tex
\documentclass[noinfoline,aos,MSNbibl,nameyear]{imsart}

\RequirePackage[breaklinks,colorlinks,citecolor=blue,urlcolor=blue]{hyperref}

\startlocaldefs

\usepackage{amsmath}
\usepackage{amsthm}
\usepackage{amsfonts}
\usepackage{graphicx}
\usepackage{amssymb}
\usepackage{enumerate}
\usepackage{enumitem}

\newtheorem{theorem}{Theorem}
\newtheorem{remark}{Remark}

\newtheorem{definition}{Definition}
\newtheorem{cor}{Corollary}
\newtheorem{lemma}{Lemma}
\newtheorem{assump}{Assumption}
\newtheorem{condition}{Condition}

\newcommand{\vertiii}[1]{{\left\vert\kern-0.25ex\left\vert\kern-0.25ex\left\vert #1
    \right\vert\kern-0.25ex\right\vert\kern-0.25ex\right\vert}}

\RequirePackage[authoryear]{natbib} 

\endlocaldefs
  
\begin{document}


\input{main}

\begin{supplement}[id=suppA]
  \stitle{Supplement to ``Semi-parametric efficiency bounds for high-dimensional models''} \slink[doi]{10.1214/00-AOASXXXXSUPP}
  \sdatatype{.pdf} 
	\sfilename{aos1286\_supp.pdf}
  \sdescription{The supplementary material contains proofs.}
\end{supplement}

\bibliography{gminf}

\printaddresses

\input{proof}

\end{document}

%% file: main.tex
\begin{frontmatter}

\title{Semi-parametric efficiency bounds for high-dimensional models
}
\runtitle{Efficiency bounds for high-dimensional models}

\author{\fnms{Jana} \snm{Jankov\'a}\corref{Jana Jankov\'a}\ead[label=e1]{jankova@stat.math.ethz.ch}
}
\and
\author{\fnms{Sara} \snm{van de Geer}\corref{Sara van de Geer}\ead[label=e2]{geer@stat.math.ethz.ch}}

 \address{ Seminar f\"ur Statistik\\ 
ETH Z\"urich \\R\"amistrasse 101\\ 8092 Z\"urich \\ Switzerland \\ \printead{e1}
\\
\printead{e2}
}
 \affiliation{ 
ETH Z\"urich}

%
\runauthor{J. Jankov\'a \and S. van de Geer}

\begin{abstract}
Asymptotic lower bounds for estimation play a fundamental role in assessing the quality of statistical procedures.
In this paper we propose a framework for obtaining semi-parametric efficiency bounds for sparse high-dimensional models, where the dimension of the parameter is larger than the sample size. 
We adopt a semi-parametric point of view: we concentrate on one dimensional functions of a high-dimensional parameter.
We follow two different approaches to reach the lower bounds: asymptotic Cram\'er-Rao  bounds and Le Cam's type of analysis. 
Both these approaches allow us to define a class of asymptotically unbiased or ``regular''  estimators for which a lower bound is derived.
Consequently, we show that certain estimators obtained by de-spar\-si\-fying (or de-biasing) an $\ell_1$-penalized M-estimator are asymptotically unbiased and achieve the lower bound on the variance: thus in this sense they are asymptotically efficient. 
The paper discusses in detail the linear regression model and the Gaussian graphical model.  

\end{abstract}



\end{frontmatter}

\section{Introduction}

\indent
Following the development of numerous me\-thods for high-dimensional estimation, more recently
the need for statistical inference has emerged. A number of papers have since studied the problem and proposed constructions of estimators which are asymptotically normally distributed and hence lead to inference.
These results naturally give rise to
the question of their optimality. 
This motivates us to study the question whether 
we can establish asymptotic efficiency bounds in high-dimensional models and whether we can construct an estimator achieving these bounds.
\par
To introduce the setting, suppose that we observe a sample $X^{(1)},\dots,X^{(n)}$ which is distributed according to a probability distribution $P_\beta$ that depends on an unknown high-dimensional parameter $\beta\in\mathcal B\subset \mathbb R^p$. The dimension $p$ of the parameter can be much larger than the sample size $n$. A major structural assumption we consider in this paper is sparsity in the high-dimensional parameter. 
In these sparse high-dimensional settings, a common  approach to estimation is based on regularized M-estimators, where the regularization is in terms of the $\ell_1$-penalty. This approach has been studied extensively, and under several settings, it produces near-oracle estimators of $\beta$ under certain sparsity conditions (and some further conditions). 
However, the oracle properties of the regularized estimators come at a price: the regularization introduces bias by shrinking the estimated coefficients towards zero. Hence, the regularized approach does not easily yield estimators which are asymptotically normally distributed.  This makes it difficult to establish results for statistical inference.
\par
Several streams of work have emerged that studied ``post-regularization inference'', which focused on construction of methodology for inference, with some preliminary use of regularized estimators. This was mostly considered for estimation of low-dimensional parameters of the high-dimensional vector.  
One stream of work concentrates on ``de-sparsifying'' or ``de-biasing'' procedures, which were studied for the linear model (\cite{zhang}, \cite{vdgeer13}, \cite{jm14a}, \cite{jm14a}, \cite{jm14b}, \cite{jm15}), for
ge\-ne\-ra\-li\-zed lin\-ear models (\cite{vdgeer13})  and some special cases of non-linear models, such as undirected graphical models (\cite{jvdgeer14}, \cite{jvdgeer15}).
This approach uses the $\ell_1$-regularized M-estimator as an initial estimator and implements  a bias correction step which may be interpreted as one iteration using the Newton-Raphson method. 
Another stream of work studies the use of orthogonalizing conditions  to define a new post-regularization estimator; this approach was considered for general models under high-level conditions in \cite{vch1}. 
Further examples of high-dimen\-sio\-nal inference include  the works \cite{zhou}, \cite{sparse.cca} or data splitting methods (\cite{data.split}). 
The work in essence shows an important result: an asymptotically normal estimator for low-dimen\-sio\-nal parameters {can} be constructed in several of the common models.
\par

Further key questions that were studied concern optimality properties of these de-sparsified estimators. In particular, 
what are lower bounds on the rate of convergence in the supremum norm? These questions have been  investigated for the linear regression with random design (\cite{cai.guo})
and for Gaussian graphical models (\cite{zhou}) and other special cases of non-linear models (\cite{sparse.cca}). 
The results in these settings reveal several important findings, which we discuss for the linear regression and graphical models. The minimax rates for estimation of single elements (of the vector of regression coefficients or the precision matrix) are shown to satisfy
\begin{equation}\label{minirate}
\inf_{T}\sup_{\beta\in\mathcal B}\mathbb E_\beta |T(X^{(1)},\dots,X^{(n)}) - \beta_i|\geq C(1/\sqrt{n}+s\log p/n),
\end{equation}
for some constant $C>0,$ where $\beta_i\in\mathbb R$ is a single regression coefficient or a single entry in a precision matrix and the \textit{unknown} sparsity $s$ is the number of non-zero entries in the regression vector or, in the case of Gaussian graphical models, in rows of a precision matrix. The infimum in \eqref{minirate} is taken over all estimators $T$. The statement \eqref{minirate} further requires some mild regularity conditions (see \cite{cai.guo}, \cite{zhou}). 
Naturally, \eqref{minirate} implies that the parametric rate is optimal: it cannot be improved in order.
 On the other hand, if there is insufficient sparsity, in particular when the sparsity $s$ satisfies $s\gg n/\log p $, the minimax lower bounds diverge. This is no surprise as the oracle inequalities for certain M-estimators have only been shown under the condition $s=o(n/\log p).$ In the intermediate sparsity regime when $\sqrt{n} /\log p \leq s<
n/\log p,$ the parametric rate cannot be achieved. 
\par
As for the upper bounds, the parametric rate $1/\sqrt{n}$ can be achieved for estimation of single entries. This basically follows directly from the asymptotic normality of the de-sparsified estimators, if sparsity of $\beta$ is of small order $\sqrt{n}/\log p.$  This sparsity condition is stronger than the condition necessary for oracle inequalities ($s=o(n/\log p)$).
However, as we discuss in Section \ref{subsec:nec}, the sparsity condition $s=o(\sqrt{n}/\log p)$ is essentially necessary for asymptotically normal estimation. 
To summarize the findings, the analysis of the minimax rates revealed that under sufficient sparsity of small order $\sqrt{n}/\log p,$
the parametric rate of order $1/\sqrt{n}$ is optimal, and the de-sparsified estimator achieves it (in the above mentioned cases).
\par
In this paper, we attempt to answer further questions that arise concerning the  optimality of  asymptotically normal estimators in high-dimensional settings. The analysis on minimax rates does not address an important question. 
The derived lower bound \eqref{minirate}  does not reveal any explicit lower bounds on the (asymptotic) variance. The question of efficiency in the spirit of the famous Cram\'er-Rao result thus  remains open in the high-dimensional setting.
This motivates us to pose the following questions.
 Can we establish lower bounds on the variance, similar to the Cram\'er-Rao bounds in the (semi-)parametric setting, also in the high-dimensional setting? And if yes, can we construct an estimator that achieves these bounds? 
We give an affirmative answer to these questions.

\section{Our contributions}
\label{sec:contrib}
\par
Asymptotic efficiency of estimators was thoroughly studied in the traditional settings; we refer the reader to the books \cite{vdv}, \cite{bickelbook} and the references therein.
These results are however developed for \textit{fixed} 
 models which do not change with $n$, and hence they cannot be applied to high-dimensional settings where the dimension of the parameter may grow with the sample size.
\par
In this paper we develop a framework for establishing asymptotic efficiency of estimators in high-dimensional models changing with $n$. 
We concentrate on two approaches towards deriving the lower efficiency bounds:  asymptotic Cram\'er-Rao bounds and Le Cam's approach.
\par
Firstly, we develop an asymptotic version of a semi-parametric  Cram\'er-Rao lower bound 
for sparse high-dimensional linear and graphical  models.
 To this end, we propose a strong asymptotic unbiasedness assumption. Loosely speaking, this unbiasedness assumption measures the rate at which  the bias vanishes in shrinking neighbourhoods of the true distribution of ``size'' $1/\sqrt{n}$.
We consider the linear model and the Gaussian graphical model and for each of them, we establish lower bounds on the variance of any asymptotically unbiased estimator. The proposed framework might be applicable to other high-dimensional models in a similar spirit.

\par
Consequently, for linear regression and Gaussian graphical models, we show that the de-sparsified estimator is an asymptotically unbiased estimator and is asymptotically efficient, i.e. it reaches the derived lower bound.
Thus, compared to previous results, which only showed asymptotic normality or minimaxity (up to order in $n$) of the de-sparsified estimator, we show that it is in terms of variance the best among all asymptotically unbiased estimators: thus in this sense asymptotically efficient.
\par
In the second approach, we extend some of the classical results of Le Cam on local asymptotic normality to the high-dimensional setting.  
The  result underlies a likelihood expansion analysis and involves a careful adjustment of Le Cam's arguments to the high-dimensional setting. The result obtained gives us the limiting distribution of an asymptotically linear estimator under a small perturbation of the parameter. 
We next show for the linear model that the de-sparsified estimator is regular: it converges locally uniformly to the limiting normal distribution with zero mean, and among all regular estimators it has the smallest asymptotic variance. 
\par
The two approaches above are strongly related, but one does not clearly dominate the other. A more detailed comparison is discussed in Section \ref{subsec:conclusion}. 
\par
As a by-product of our analysis, we establish new oracle results for the Lasso. Typical analysis considers oracle inequalities for the prediction error and the $\ell_1$-error which hold with high-probability. We strengthen these oracle inequalities by showing that they also hold for the mean $\ell_1$-error and  for higher orders of this error.
These oracle inequalities are needed to claim strong asymptotic unbiasedness of the de-sparsified estimators.

\section{Relation to prior work}
As pointed out in Section \ref{sec:contrib}, the traditional results as in, for instance, \cite{vdv} or \cite{bickelbook}, are not directly applicable to the high-dimensional setting. We extend the traditional approach to semi-parametric efficiency to the context of high-dimensional models which requires
adjustment of the  arguments to a model changing with $n$ and the sparsity of the model is required to keep remainders in approximate expansions under control. 
Our main results show that the lower bounds for high-dimensional models are analogous to those for parametric models, however, a new message for high-dimensional models is that to obtain the parametric  lower bound, we require that the ``worst possible sub-direction'' is sparse.
Without this condition, we are unable to claim asymptotic efficiency of the de-sparsified Lasso estimator.

\par
Regarding the upper bounds, to construct asymptotically efficient estimators, our work follows the methodology from the works \cite{vdgeer13} and \cite{jvdgeer15}, where de-sparsified Lasso estimators are proposed for the linear regression and for undirected graphical models. We borrow these constructions with some small adjustments. However, the upper bounds derived for the de-sparsified estimators in the mentioned papers are not sufficient for the present analysis: we need to show a stronger oracle bound which holds in expectation. Moreover, we  extend the results for estimation of single entries as considered in \cite{vdgeer13} and \cite{jvdgeer15} to linear functionals.
\par
Asymptotic efficiency of estimators in high-dimensional settings changing with $n$ was first considered in the paper \cite{vdgeer13}. The paper provides a formulation of asymptotic efficiency of entries of the de-biased lasso. The approach is based on embedding the high-dimensional model into a fixed (i.e. not changing with $n$) infinite-dimensional model, for which semi-parametric efficiency bounds are available (see \cite{vdv}). However, such an embedding requires a very special model structure. In the present paper, we do not use an embedding but instead directly develop the theory for models changing with $n.$

%

\section{Organization of the paper}
\par
The particular sections of the paper are divided as follows. In Section \ref{subsec:strong} we state preliminary results on oracle inequalities for the mean $\ell_1$-error of the  Lasso estimator. 
In Section \ref{subsec:au} we propose a strong asymptotic unbiasedness assumption. Section \ref{subsec:lr.all} gives lower and upper bounds on the variance of asymptotically unbiased estimators in the linear model, considering random design in Section \ref{subsec:lr.random} and fixed design in Section \ref{subsec:lr.fixed}.
In Section \ref{subsec:ggm.all} we derive lower and upper bounds on the variance of asymptotically unbiased estimators in Gaussian graphical models. 
 Section \ref{sec:lecam} contains an extension of Le Cam's lemma to the high-dimensional setting, which is applicable to general non-linear models. Section \ref{subsec:conclusion} summarizes the results, conclusions and some open questions. Finally, the proofs are contained in the supplemental article \cite{sup}.

\section{Notation}
For a vector $x=(x_1,\dots,x_p)\in\mathbb R^p$ we denote its $\ell_p$ norm by $\|x\|_p:= (\sum_{i=1}^p x_i^p)^{1/p}$ for $p\geq 1$. We further let $\|x\|_\infty:=\max_{i=1,\dots,p}|x_i|$ and $\|x\|_0 = |\{i:i\in\{1,\dots,p\},x_i\not =0\}|.$  For a vector $x\in\mathbb R^n$ we denote $\|x\|^2_n:=\|x\|_2^2/n$ (with some abuse of notation).
By $e_i$ we denote a $p$-dimensional vector of zeros with  a one at position $i$. For a matrix $A\in\mathbb R^{m\times n}$, we denote its $(i,j)$-th entry by $A_{ij},i=1,\dots,m,j=1,\dots,n.$
Further, we let $\|A\|_\infty:=\max_{i=1,\dots,m,j=1,\dots,n} |A_{ij}|$,
 $\vertiii{A}_1$ $ := \max_{i=1,\dots,m} \sum_{j=1}^n|A_{ij}|$ and we let $\|A\|_F$ denote the Frobenius norm of $A.$
We denote its $j$-th column by $A_j$.
By $\Lambda_{\min}(A)$ and $\Lambda_{\max}(A)$ we denote the minimum and maximum eigenvalue of a symmetric matrix $A$, respectively. 
We use $\textrm{tr}(A)$ to denote the trace of the matrix $A$.
We recall here that for symmetric matrices $A,B\in \mathbb R^{p\times p}$ it holds that
$\textrm{vec}(A)^T\textrm{vec}(B) = \textrm{tr}(AB),$
where $\textrm{vec}(A)$ is the vectorized version of a matrix $A$ obtained by stacking columns of $A$ on each other. 
 \par
For real sequences $f_n,g_n$, we write $f_n=\mathcal O(g_n)$ or $f_n\lesssim g_n$  if $|f_n| \leq C |g_n|$ for some $C>0$ independent of $n$ for all $n.$ 
We write 
$f_n\asymp g_n$ if both $f_n = \mathcal O(g_n)$ and $1/f_n =\mathcal O(1/g_n)$ hold. 
Finally, $f_n=o(g_n)$ if $\lim_{n\rightarrow \infty} f_n/g_n =0.$   
For a sequence of random variables $X_n$, we write $X_n=\mathcal O_P(f_n)$ if $X_n/f_n$ is bounded in probability. We write $X_n=o_P(1)$ if $X_n$ converges to zero in probability.
We use $\rightsquigarrow$ to denote the convergence in distribution. 
By $1_T$ we denote the indicator function of the set $T.$ The identity matrix is denoted by $I$.


\section{Asymptotic unbiasedness}
\label{subsec:au}
This section defines the concept of stro\-ng asymptotic unbiasedness that will be needed for the linear and graphical model.
We turn to the linear model in the next section.
Consider 
a probability distribution $P_{{\beta}}$ on some observation space $\mathcal X,$
where the parameter $\beta$ lies is a $p$-dimensional parameter space $\mathcal B \subset \mathbb R^p.$  
We consider the parameter set 
\begin{equation}\label{parsp}
 \mathcal B(d_n) :=\{ \beta \in \mathcal B: \|\beta\|_0 \leq d_n , \|\beta\|_2 \leq C\},
\end{equation}
where $C>0$ is some universal constant and $d_n$ is a  \textit{known} sequence that will be specified later. 
We further define an $\ell_2$-neighbourhood of a point ${\beta}\in\mathcal B(d_n)$ as follows
\begin{equation}\label{neigh}
 B({\beta},\varepsilon) := \{\tilde{\beta}\in\mathcal B(d_n): \|\tilde{\beta}-{\beta}\|_2\leq \varepsilon\}.
\end{equation} 
We remark that all the parameter vectors  appearing in this paper are \textit{sequences} depending on $n$. In  general we omit the index $n$, except for situations where omitting the index could lead to confusion.

Let $g:\mathcal B\rightarrow \mathbb R$ and let the parameter of interest be $g(\beta)$.
Our  goal is to derive an asymptotic lower bound for the variance of an estimator $T_n$ of $g(\beta)$, 
which is in some sense asymptotically unbiased. 
To this end, we define \textit{strong asymptotic unbiasedness} as follows.  

\begin{definition}\label{au}
Let $m_n$ be a sequence such that $n=o(m_n)$.
We say that $T_n$ is a strongly asymptotically unbiased estimator of $g(\beta)$ at $\beta_0$ 
(in a neighbourhood of size $c$) with a rate $m_n$ if it holds that $\emph{var}_{\beta_0}(T_n)=\mathcal O(1/n)$ 
and for every $\beta\in B\left(\beta_0,\frac{c}{\sqrt{m_n}}\right)$ it holds
$$\lim_{n\rightarrow \infty}\sqrt{m_n} (\mathbb E_{\beta} T_n - g(\beta)) = 0.$$
\end{definition}

The motivation for Definition \ref{au} comes from the asymptotic unbiasedness assumption for semi-parametric models, which is assumed to hold in a small  neighborhood of $\beta_0.$ 
Definition \ref{au} implies that the mean squared error of the considered estimator must be of order $1/n$.

\section{Strong oracle inequalities for the Lasso}
\label{subsec:strong}
We present new results on oracle inequalities for the Lasso estimator in linear regression which will be needed in subsequent sections, but can also be of independent interest.
Typical high-dimensional analysis derives oracle inequalities for the Lasso which hold with high probability (see \cite{hds} for an overview of such results). 
The paper \cite{bellec.tsybakov} derives bounds on the expectation of the prediction error.
Here we derive oracle inequalities for the $\ell_1$-estimation error that hold in expectation.
\par
Consider the linear model 
\begin{equation} \label{LR0}
Y = X\beta_0 + \epsilon,
\end{equation}
where $X$ is the $n\times p$ design matrix with independent rows $X^{(i)},i=1,\dots,n$, $Y$ is the $n\times 1$ vector of observations and $\epsilon=(\epsilon_1,\dots,\epsilon_n)^T \in\mathbb R^n$ is the (unobservable) error. The error satisfies $\mathbb E\epsilon=0$ and its components $\epsilon_i$ are independent for $i=1,\dots,n$. 
Moreover, the error $\epsilon$ and the design matrix $X$ are independent. 
 We further denote the Gram matrix by $\hat\Sigma:=X^T X/n$.
The vector $\beta_0 =(\beta_1^0,\dots,\beta_p^0)\in \mathbb R^p$ is unknown.
The unknown number of non-zero entries of $\beta_0$ is denoted by $s:=\|\beta_0\|_0$ and is called the sparsity of $\beta_0.$

\par
The Lasso estimator with a tuning parameter $\lambda>0$ is defined as follows:
\begin{equation}\label{lasso}
\hat\beta := \textrm{arg}\min_{\beta\in\mathbb R^p} \|Y-X\beta\|_n^2+ 2\lambda \|\beta\|_1.
\end{equation}
The known results on oracle inequalities for the Lasso \eqref{lasso} give high-proba\-bi\-li\-ty bounds for the prediction error and the $\ell_1$-error (or under some conditions, for the $\ell_q$-error for $1\leq q \leq 2$). In particular, for the tuning parameter  $\lambda\asymp \sqrt{\log p/n}$ and under  further conditions that may be found in \cite{hds}, it holds
$$\|X(\hat\beta - \beta_0)\|_n^2 + \lambda \|\hat\beta-\beta_0\|_1 =\mathcal O_P(s\lambda^2).$$
\cite{bellec.tsybakov} show analogous results for the expected prediction error $\mathbb E\|X(\hat\beta-\beta_0)\|_n$ for the case of fixed design.
 We show such results may be obtained for the expected $\ell_1$-error, under almost identical conditions.
In particular, Theorem \ref{strong2} presented below implies that the mean $\ell_1$-error, $\mathbb E_{\beta_0}\|\hat\beta-\beta_0\|_1$, is up to a logarithmic factor of the same order as the oracle error $\mathbb E_{\beta_0}\|\beta_{ora}-\beta_0\|_1=\mathcal O(s/\sqrt{n}),$ where $\beta_{ora}$ is the oracle maximum likelihood estimator (i.e. a maximum likelihood estimator applied with the knowledge of true non-zero entries of $\beta_0$). Theorem \ref{strong2} actually shows a more general result since it considers also higher-order errors, namely the $k$-th order error $\mathbb E_{\beta_0}\|\hat\beta-\beta_0\|_1^k$ for any fixed $k\in\{1,2,\dots\}$.

We consider the situation when the errors $\epsilon_i$ are independent and sub-Gaussian (with a universal constant)  and the design $X$ has independent sub-Gaussian rows (with a universal constant). 
To this end, we recall a sub-Gaussianity assumption on random variables and vectors (see Section 14 in \cite{hds}).
\begin{definition}\label{sgdef}
We say that a random vector $Z\in\mathbb R^m$ 
has sub-Gaussian entries with constants 
$K,K_2>0$ if 
$$\mathbb Ee^{Z_j^2/K^2}\leq K_2, \;\;\;\;j=1,\dots,m.$$
 We say that a random vector $Z\in\mathbb R^m$ is sub-Gaussian with constants 
$K,K_2>0$ if for all $\alpha\in \mathbb R^m$ such that 
$\|\alpha\|_2=1$ it holds that
$$\mathbb E e^{(\alpha^T Z)^2/K^2} \leq K_2.$$
\end{definition}
In our further analysis, we typically require that the sub-Gaussianity condition as in Definition \ref{sgdef} is satisfied with universal constants $K,K_2>0.$
A prime example of a sub-Gaussian random vector with a universal constant is a Gaussian random vector with zero mean and covariance matrix $\Sigma_0$ that satisfies $\Lambda_{\max}(\Sigma_0) = \mathcal O(1)$. 
We  formulate the conditions on the error and the design in the following. 
\begin{enumerate}[label=(A\arabic*),start=1]
\item\label{model.sg}
Assume the linear model \eqref{LR0}, where the errors $\epsilon_i$ are independent sub-Gaussian random variables with universal constants and with $\mathbb E\epsilon_i=0$. 
\item\label{design.entry}
Assume 
 that $X$ is a random $n\times p$ matrix independent of $\epsilon$ with independent rows $X^{(i)},i=1,\dots,n,$ with mean zero and with  sub-Gaussian entries with universal constants.
We let $\Sigma_0:=\mathbb E\hat\Sigma$ and suppose that $1/\Lambda_{\min}(\Sigma_0)=\mathcal O(1)$. 
\end{enumerate}

\begin{enumerate}[label=(A\arabic**),start=2]
\item\label{design}
Assume 
 that $X$ is a random $n\times p$ matrix independent of $\epsilon$ with independent sub-Gaussian rows $X^{(i)},i=1,\dots,n,$ with universal constants, with mean zero.
We let $\Sigma_0:=\mathbb E\hat\Sigma$ and suppose that $1/\Lambda_{\min}(\Sigma_0)=\mathcal O(1)$. 
\end{enumerate}

Under conditions \ref{design.entry} or \ref{design} we denote the inverse covariance matrix by $\Theta_0 := \Sigma_0^{-1}$
and by $\Theta_j^0$ we denote its $j$-th column ($j=1,\dots,p$).

\begin{theorem}\label{strong2}
Suppose that conditions \ref{model.sg}, \ref{design.entry} are satisfied.
Suppose that $\|\beta_0\|_2=\mathcal O(1)$, $s\sqrt{\log p/n}=o(1)$ and
let $k\in\{1,2,\dots\}$ be fixed. 
Consider  the Lasso estimator $\hat\beta$  defined in \eqref{lasso} with a tuning parameter $\lambda\geq c\tau \sqrt{\log p/n},$ where $c>0$ is a sufficiently large universal constant and $\tau>1$ satisfies $\tau^2 > 2k \log((\sqrt{s}\lambda^2)^{-1})/\log p$.
Then there exists a universal constant $C_1$ such that
$$(\mathbb E_{\beta_0}\|\hat\beta-\beta_0\|_1^k)^{1/k} \leq C_1 s\lambda.$$
\end{theorem}

\noindent
Taking $k=1$, under the conditions of Theorem \ref{strong2} we obtain
$$\mathbb E_{\beta_0}\|\hat\beta-\beta_0\|_1 \leq C_1 s\lambda.$$
\par
Theorem \ref{strong2} can also be easily extended to fixed design, under a compatibility condition (see Section \ref{sec:proof.oi}) on the Gram matrix $\hat\Sigma$, which substitutes the condition $\Lambda_{\min}(\Sigma_0)\geq L >0$, and under the condition $\|\hat\Sigma\|_\infty = \mathcal O(1)$. 
\par
We comment on the conditions \ref{model.sg}, \ref{design.entry} and $\|\beta_0\|_2=\mathcal O(1)$, $\|\beta_0\|_0=
o(\sqrt{n/\log p})$ assumed in Theorem \ref{strong2}. 
Condition $\|\beta_0\|_0 =o(\sqrt{n/\log p})$ together with conditions \ref{model.sg}, \ref{design.entry} was used to apply the high-probability oracle results for Lasso as in \cite{hds} to the case of random design.
Condition $\|\beta_0\|_2 = \mathcal O(1)$ can be justified under an assumption on the boundedness of  the ``signal-to-noise ratio''. The ``signal-to-noise ratio'' is defined as the ratio of the variance of the signal (observations) and  the variance of the noise, 
i.e. 
$\sum_{i=1}^n \textrm{var}_{\beta_0}(Y_i)/ \sum_{i=1}^n \textrm{var}(\epsilon_i ) = 1 + \beta_0^T \Sigma_0 \beta_0 /\sigma_\epsilon^2,$
where $\sigma_\epsilon^2 := \frac{1}{n}\sum_{i=1}^n \text{var}(\epsilon_i)$. 
Hence, under upper-boundedness of $1/\Lambda_{\min}(\Sigma_0)$, the signal-to-noise ratio is up to a constant lower-bounded by $\|\beta_0\|_2^2/\sigma_\epsilon^2$. If we assume that the signal-to-noise ratio remains bounded and the variance of the noise $\sigma_\epsilon^2$ is bounded (as implied by condition \ref{model.sg}), then the $\ell_2$-norm of $\beta_0 $ must also remain bounded. 
\par
Finally, the condition $\tau^2 > 2k \log((s\lambda^2)^{-1})/\log p$
 only guarantees that we choose sufficiently large regularization parameter $\lambda\geq c\tau \sqrt{\log p/n}$ by choosing $\tau$ large enough compared to the order $k$ of the error that we want to control. If $p\geq n$ and $\lambda=c\tau\sqrt{\log p/n}$, the condition reduces to $\tau\geq C\sqrt{k}$ for some constant $C>0.$
Then clearly, 
this condition means that the higher order of error we want to control, the stronger regularization must be chosen. 

\section{The de-sparsified Lasso}
\label{subsec:lr.all}
\subsection{Methodology}
As an initial estimator, we consider the Lasso estimator \eqref{lasso}.
The Lasso estimator is well-understood in terms of prediction and estimation error bounds, and was shown minimax optimal in terms of the prediction error and $\ell_1$-error. 
However, due to  the inclusion of the $\ell_1$-penalty, the estimator is biased and its limiting distribution can accumulate a positive mass at zero (\cite{knight2000}). 
In view of statistical inference, a de-sparsified or de-biased version of the Lasso was then considered (see \cite{zhang}, \cite{vdgeer13}, \cite{jm14a}, \cite{jm14a}, \cite{jm14b}, \cite{jm15}), which was shown to be asymptotically normal for estimation of $\beta_j^0$. 
\par
To construct the de-biased estimator, we further need to construct a surrogate inverse of $\hat\Sigma$, or in other words we need to construct an estimator of the inverse covariance matrix $\Theta_0=\Sigma_0^{-1}$.
We define $\hat\Theta_j$ as an estimate of the column $\Theta^0_j$ obtained by solving the following program, that will be referred to as \textit{nodewise regression} (see \cite{vdgeer13}).
Recall that $X$ is the design matrix with rows $X^{(i)},i=1,\dots,n.$
The columns of the design matrix $X$ will be denoted by $X_j,j=1,\dots,p,$ and  by $X_{-j}$ we denote the $n\times (p-1)$ matrix obtained by removing the $j$-th column from  $X.$
For $j=1,\dots,p$, we let
\begin{equation}\label{nc}
\hat\gamma_j:= \textrm{arg}\min_{\gamma\in\mathbb R^{p-1}} \|X_j-X_{-j}\gamma\|_n^2+ 2\lambda_j \|\gamma\|_1,
\end{equation}
$$\hat\tau_j^2 := \|X_j-X_{-j}\hat\gamma_j\|_n^2 + \lambda_j\|\hat\gamma_j\|_1,$$
and we denote the $j$-th column of the nodewise Lasso estimator by
\begin{equation}
\label{node.lr}
\hat\Theta_{ j} := (-\hat\gamma_{j,1},\dots,-\hat\gamma_{j,j-1},1,-\hat\gamma_{j,j+1},\dots,-\hat\gamma_{j,p})^T/\hat\tau_j^2,
\end{equation}
where $\lambda_j \asymp\sqrt{\log p/n}$ for $j=1,\dots,p$, uniformly in $j$. We denote the nodewise Lasso estimator by  $\hat\Theta := (\hat\Theta_1,\dots,\hat\Theta_p).$
The necessary Karush-Kuhn-Tucker conditions corresponding to the nodewise regression (obtained by replacing derivatives by sub-differentials) imply the condition $\|\hat\Sigma\hat\Theta_j-e_j \|_\infty=\mathcal O_P(\lambda_j/\hat\tau_j^2)$ (see \cite{vdgeer13}), which will be needed later. 
We now define the de-sparsified Lasso introduced in \cite{vdgeer13},
\begin{equation}
\label{bj}
\hat  b := \hat\beta + \hat\Theta^T  X^T(Y-X\hat \beta) /n,
\end{equation}
and we let $\hat b_j$ denote its $j$-th entry.
The motivation for the definition \eqref{bj} comes from updating the initial Lasso estimator $\hat\beta$ by removing the bias due to the $\ell_1$-penalty. 
We briefly summarize the main results on $\hat b$ as derived in \cite{vdgeer13}. The estimator $\hat b_j$ can be shown to be asymptotically linear with a remainder term of small order $1/\sqrt{n}$, in particular, under the conditions \ref{model.sg}, \ref{design} and
 $$s=o(\sqrt{n}/\log p),\;\;\;
\max_{j=1,\dots,p }s_j=o(\sqrt{n}/\log p)$$
 it holds
$$\hat b -\beta_0= \Theta_0 X^T \epsilon /n + \Delta,$$
where $s_j:=\|\Theta_j^0\|_0$ and $\|\Delta\|_\infty= o_P(1/\sqrt{n})$.
 Thus, after normalization by $\sqrt{n}$ and by the (estimated) standard deviation, asymptotic normality of entries of $\hat b$ with zero mean and unit variance follows by the central limit theorem. We now investigate the question of ``regularity'' and asymptotic efficiency of this estimator. 
\par
We first show  that the de-sparsified estimator $\hat b_j$
satisfies the {strong asymptotic unbiasedness} condition from Definition \ref{au} in Section \ref{subsec:au}.
We then show that $\hat b_j$ achieves the lower bound on the variance
of any strongly asymptotically unbiased estimator. 
Thus in this sense the de-sparsified estimator is  asymptotically efficient.  
In Section \ref{subsec:lr.random} we investigate the case of a random Gaussian design matrix and in Section \ref{subsec:lr.fixed} the case of a fixed design matrix.


\subsection{Strong asymptotic unbiasedness of the de-sparsified Lasso}
\label{subsec:saulasso}
\par

We consider estimation of linear functionals $g(\beta)=\xi^T \beta$, where $\xi \in \mathbb R^p$ is a known vector. 
We define an estimator of $g(\beta)=\xi^T\beta$ as a linear combination $\xi$ of the de-sparsified estimator $\hat b.$ This yields 
\begin{eqnarray}
\label{xib}
\hat b_\xi := \xi^T \hat b=
 \xi^T \hat \beta +  \xi^T\hat\Theta X^T(Y-X\hat\beta)/n.
\end{eqnarray}
Then we have the following lemma, which shows strong asymptotic unbiasedness of $\hat b_\xi$ for estimation of $\xi^T\beta.$

\begin{lemma} \label{SAUbxi}
Suppose that conditions 
\ref{model.sg}, \ref{design} are satisfied, $\beta_0\in \mathcal B(d_n)$ where
$d_n = o\left( {\sqrt{n}}/{\log p} \right)$, 
$\max_j s_j \leq d_n$,
$\|\xi\|_1 = \mathcal O(1)$ and $\|\Sigma_0\|_\infty = \mathcal O(1)$.
Let $\hat b_\xi$ be the estimator defined in \eqref{xib} with tuning parameters of the Lasso and nodewise regression $\lambda\asymp\lambda_j\asymp \sqrt{\log p/n}$ uniformly in $j=1,\dots,p$.
Then $\hat b_\xi$ is a strongly asymptotically unbiased estimator of $\xi^T\beta$ at $\beta_0.$
\end{lemma}

\par
\subsection{Main results for random design}
\label{subsec:lr.random}
We derive lower bounds for the variance of a strongly asymptotically unbiased estimator. 
We consider the following conditions on the error distribution and the design matrix $X$.

\begin{enumerate}[label=(B\arabic*)]
\item\label{model}
Assume the linear model \eqref{LR0} with $\epsilon\sim \mathcal N(0,I)$.
\end{enumerate}

\begin{enumerate}[label=(B\arabic*),start=2]
\item\label{design.normal}
 Assume 
 that $X$ is a random $n\times p$ matrix independent of $\epsilon$ with independent rows $X^{(i)}\sim \mathcal N(0,\Sigma_0)$ for $i=1,\dots,n.$
Suppose that the inverse covariance matrix $\Theta_0 := \Sigma_0^{-1}$ exists, $1/\Lambda_{\min}(\Sigma_0)=\mathcal O(1)$ and $\|\Sigma_0\|_\infty = \mathcal O(1).$
\end{enumerate}

\begin{theorem}\label{crlb.lr} 
Suppose that conditions \ref{model}, \ref{design.normal}  are satisfied.
Suppose that $T_n$ is a strongly asymptotically unbiased estimator of $g(\beta)$ at 
$\beta_0\in\mathcal B(d_n)$ with a rate $m_n$. 
Let $h\in\mathbb R^p$ satisfy $h^T\Sigma_0 h =1$ and $\beta_0 + h/\sqrt{m_n}\in B\left(\beta_0,\frac{c}{\sqrt{m_n}}\right)$ 
for a sufficiently large universal constant $c$.
Assume moreover that for some $\dot g(\beta_0)\in\mathbb R^p$ it holds
\begin{equation}\label{g.nice}
\sqrt{m_n} \left(  g(\beta_0+h/\sqrt{m_n}) - g(\beta_0)  \right) = h^T \dot g(\beta_0) + o(1).
\end{equation}
Then
$$n\emph{var}_{\beta_0}(T_n) \geq [h^T \dot g(\beta_0)]^2 - o(1).$$
\end{theorem}

Theorem \ref{crlb.lr} yields a lower bound $[h^T \dot g(\beta_0)]^2 - o(1)$ on the variance of an estimator which is a strongly asymptotically unbiased estimator in a direction $h$, such that $\beta_0+h/\sqrt{m_n}$ remains within the model.
By maximizing $[h^T \dot g(\beta_0)]^2$ over all feasible $h$, we obtain the following corollary.

\begin{cor}\label{cor.wst}
If $\beta_0 + \Theta_0\dot g(\beta_0) /\sqrt{ \dot g(\beta_0)^T\Theta_0 \dot g(\beta_0) m_n} \in B\left(\beta_0,\frac{c}{\sqrt{m_n}}\right)$,
then the lower bound from Theorem \ref{crlb.lr} is maximized at the value 
$$h_0:= \Theta_0 \dot g(\beta_0)/\sqrt{\dot g(\beta_0)^T \Theta_0 \dot g(\beta_0)},$$
and under the conditions of Theorem \ref{crlb.lr}, we get 
and under the conditions of Theorem \ref{crlb.lr}, we get 
$$n\emph{var}_{\beta_0}(T_n) \geq \dot g(\beta_0)^T\Theta_0 \dot g(\beta_0) - o(1).$$
\end{cor}

\begin{definition}\label{wst}
Let $g$ be differentiable at $\beta_0$ with derivative $\dot g(\beta_0).$ We call
$$c_0:= \Theta_0 \dot g(\beta_0)/{\dot g(\beta_0)^T \Theta_0 \dot g(\beta_0)}$$
the worst possible sub-direction for estimating $g(\beta_0).$
\end{definition}
 The motivation for the terminology \textit{worst possible sub-direction} in Definition \ref{wst} is given by Corollary \ref{cor.wst}.
The normalization by $\dot g(\beta_0)^T \Theta_0 \dot g(\beta_0)$ is arbitrary but natural from a projetion theory point of view.
\par
As a special case, consider estimation of $g(\beta) = \beta_j$ for some fixed value of $j\in\{1,\dots,p\}.$ Then $\dot g(\beta)=e_j,$
the $j$-th unit vector in $\mathbb R^p.$
Clearly, $\Theta_0\dot g(\beta_0)=\Theta_0e_j = \Theta^0_j$ and $g(\beta_0)^T\Theta_0\dot g(\beta_0)$ $ = e_j^T\Theta_0e_j=\Theta^0_{jj},$ where $\Theta^0_j$ is the $j$-th column of $\Theta_0$
and $\Theta_{jj}^0$ is its $j$-th diagonal element. 
It follows that 
$c_j^0=\Theta_{j}^0/\Theta^0_{jj}$ is the worst possible sub-direction for estimating $\beta_j.$
If $\beta_0 + \Theta_j^0 /\sqrt{ \Theta_{11}^0 m_n} \in B\left(\beta_0,\frac{c}{\sqrt{m_n}}\right)$, 
then
 Corollary \ref{cor.wst} implies the lower bound
$$\textrm{var}_{\beta_0}(T_n) \geq \Theta_{jj}^0/n+o(1/n).$$
\par

\begin{remark}
To establish the lower bound, it is crucial that the worst possible sub-direction lies within the model. 
For illustration, consider the situation with the parameter of interest being $g(\beta) = \beta_1$.
When $\Theta_1^0$ is not sufficiently sparse, 
we are not allowed to take the global maximizer $h=\Theta_1^0/\sqrt{\Theta_{11}^0}$ in the maximum and the lower bound might thus become smaller. 
In that case, the lower bound is given via a sparse approximation of the (non-sparse) precision matrix. For a set $M\subset \{1,\dots,p\}$ and a vector $v\in \mathbb R^p$, we denote $v_M$ as a $p$-dimensional vector with entries not in $M$ set to zero. Then we may write
\begin{eqnarray*}
&&
\max_{\beta_ 0 + h/\sqrt{m_n} \in B(\beta_0, c/\sqrt{m_n})} \frac{h^Te_1}{h^T \Sigma_0 h} 
\\
&&\geq
\max_{
\substack{M\subset \{1,\dots,p\}:
\\ 
|M| = d_n-\|\beta_0\|_0
}}
\;\;
\max_{
\substack{
h \in \mathbb R^p: 
\|h_M\|_2\leq c, \\  \|\beta^0_M + h_M/\sqrt{m_n}\|_2 \leq C}
}  
\frac{ [h_M^T e_1]^2 }{ h_M^T \Sigma_0 h_M}.
\end{eqnarray*}
But if $h:= (\Sigma^0_{M,M})^{-1}e_1$ satisfies 
$\|h\|_2\leq c$ and $\|\beta^0_M + h/\sqrt{m_n}\|_2 \leq C$,
then the lower bound is 
$$\max_{M\subset \{1,\dots,p\}: |M| = d_n-\|\beta_0\|_0} (\Sigma^0_{M,M})^{-1}_{11} - o(1),
$$
where $\Sigma^0_{M,M}$ is the reduction of $\Sigma_0$ obtained by keeping only columns and rows belonging to the set $M.$
If $\Theta_1^0$ has sparsity $d_n-\|\beta_0\|_0$, then this lower bound coincides with $(\Sigma_{0})^{-1}_{11}-o(1)$ as before.
If $\Theta_1^0$ is not sufficiently sparse, then the lower bound is given via a sparse approximation of the precision matrix.
Finally, as will be seen in the following sections, without assuming the sparsity condition on the worst possible sub-direction, we would not be able to conclude asymptotic efficiency of the de-sparsified Lasso estimator.
\end{remark}

Finally we show that the de-sparsified estimator $\hat b_j$ achieves the lower bound on the variance. 
Thus the de-sparsified estimator is strongly asymptotically unbiased and has the smallest variance among all strongly asymptotically unbiased estimators. We assume Gaussianity of the error and the design matrix, as the lower bounds have only been derived for this case.
\begin{theorem}\label{bxi}
Suppose that conditions 
\ref{model}, \ref{design.normal} are satisfied, $\beta_0\in\mathcal B(d_n)$ with $d_n = o\left( {\sqrt{n}}/{\log p} \right)$ and $\max_j s_j \leq d_n$.
Assume that $\|\xi\|_1 =\mathcal O(1)$. 
Let $\hat b_\xi$ be the estimator defined in \eqref{xib} with tuning parameters of the Lasso and nodewise regression $\lambda\asymp\lambda_j
\asymp \sqrt{\log p /n}$, uniformly in $j=1,\dots,p$.
Then $\hat b_\xi $ is a strongly asymptotically unbiased estimator of $\xi^T\beta$ at $\beta_0$. Let $T$ be any strongly asymptotically unbiased estimator of $\xi^T \beta$ at $\beta_0$  and assume that $\beta_0 + \Theta_0\xi/(\xi^T \Theta_0\xi{n})^{1/2}\in B_{}(\beta_0,c/\sqrt{n})$. Then it holds 
$$\emph{var}_{\beta_0}( T ) \geq  \frac{\xi^T \Theta_0 \xi+ o(1)}{{n}},\quad\quad \emph{var}_{\beta_0}(\hat b_\xi)=\frac{\xi^T \Theta_0 \xi+ o(1)}{{n}} .$$
\end{theorem}
To obtain the result of Theorem \ref{bxi}, we assumed that
$\beta_0 + \Theta_0\xi/(\xi^T \Theta_0\xi{n})^{1/2}$ $\in$  $B_{}(\beta_0,c/\sqrt{n})$, which guarantees that the worst possible sub-direction stays within the model.
Further we assumed that the sparsity in $\beta_0$ satisfies $s=o(\sqrt{n}/\log p)$ and that the sparsity in the rows of $\Theta_0$ is of small order $\sqrt{n}/\log p$.
Thus, to be able to claim asymptotic efficiency of the de-sparsified Lasso, we not only require sparsity in $\beta_0$, but also sufficient sparsity in the precision matrix.
Note that  the sparsity condition on $\beta_0$ is almost a necessary condition as discussed in Section \ref{subsec:nec} below.

\subsection{Main results for fixed design}
\label{subsec:lr.fixed}

In this section, we assume that the design matrix $X$ is fixed (non-random). Recall that $\hat\Sigma=X^T X/n$ is the Gram matrix.
The following theorem is an analogy of Theorem \ref{crlb.lr} for fixed design.

\begin{theorem}\label{fixed}
Let $X$ be a fixed $n\times p$ matrix and suppose that condition \ref{model} is satisfied.
Let $h\in\mathbb R^p$ be such that $h^T \hat\Sigma h =\mathcal O(1)$ and 
$\beta_0+h/\sqrt{m_n} \in B_{}(\beta_0,c/\sqrt{m_n})$. Suppose that $T_n$
is a strongly  asymptotically unbiased estimator of $g(\beta)$ at $\beta_0$ in the direction $h$ with rate $m_n.$
Assume moreover that for some $\dot g(\beta_0)\in\mathbb R^p$ it holds that
\begin{equation}\label{g.nice2}
\sqrt{m_n} \left(  g(\beta_0+h/\sqrt{m_n}) - g(\beta_0)  \right) = h^T \dot g(\beta_0) + o(1).
\end{equation}
Then
$$n\emph{var}_{\beta_0}(T_n) \geq [h^T \dot g(\beta_0)]^2 - o(1).$$
\end{theorem}

For fixed design, the matrix $\hat\Sigma$ is not invertible, and thus we cannot use the reasoning as
in Section \ref{subsec:lr.random}.
We can however try to remedy this by proposing an approximate worst possible sub-direction.
To this end, we may use an estimator $\hat\Theta$, which acts as a surrogate inverse of $\hat\Sigma$ in a certain sense.
Such an estimate can be obtained in the same way as for the random design, using the nodewise regression \eqref{node.lr}.
The necessary Karush-Kuhn-Tucker conditions of the nodewise regression (obtained by replacing derivatives by sub-differentials) again imply the condition $\|\hat\Sigma\hat\Theta_j-e_j \|_\infty=\mathcal O_P( \lambda_j/\hat\tau_j^2)$. 
The de-sparsified estimator can then be defined in the same way as for the random design, as in equation \eqref{bj}.

We consider estimation of $g(\beta_0):=\beta_j^0$, although one could further consider estimation of linear functionals, similarly as for the random design.
Strong asymptotic unbiasedness of $\hat b_j$ for estimation of $\beta_j$ then follows similarly as in Lemma \ref{SAUbxi} (with $g(\beta)=\beta_j$) for all  $\beta\in\mathcal B(d_n)$, under $d_n = o\left( {\sqrt{n}}/{\log p} \right)$, if the compatibility condition is satisfied for 
$\hat\Sigma$ with a universal constant and $\|\hat\Sigma\|_\infty = \mathcal O(1)$. For the definition of the compatibility condition, see Definition \ref{cc.def} in Section \ref{sec:proof.oi} of the supplemental article \cite{sup}.
 We formulate the asymptotic efficiency of $\hat b_j$ for  $g(\beta):=\beta_j$ in the following theorem.
\begin{theorem}
\label{LR2}
Assume that condition \ref{model} is satisfied and $\beta_0\in\mathcal B(d_n)$ with $d_n = o\left( {\sqrt{n}}/{\log p} \right)$. 
Let $j\in\{1,\dots,p\}$ and let $\hat\Theta_{j}$ be obtained using the nodewise regression as in \eqref{node.lr} with $\lambda_j\asymp\sqrt{\log p/n}$. 
Suppose that 
$\beta_0+\hat\Theta_j/({\hat\Theta_{jj}n})^{1/2} \in B_{}(\beta_0,c/\sqrt{n})$, $\|\hat\Theta_j\|_2 =\mathcal O(1),$ 
 the compatibility condition is satisfied for $\hat\Sigma$ with a universal constant and $\|\hat\Sigma\|_\infty = \mathcal O(1)$.
Then $\hat b_j $ defined in \eqref{bj} using $\hat\Theta_{j}$ and with $\lambda\asymp\sqrt{\log p/n}$ is a strongly asymptotically unbiased
estimator of $\beta_j$ at $\beta_0$ and for any 
strongly asymptotically unbiased estimator $T$ of $\beta_j$ at $\beta_0$
it holds 
$$\emph{var}_{\beta_0}(T) \geq  \frac{\hat\Theta_{jj}+ o(1)}{{n}} 
,
\quad
\emph{var}_{\beta_0}(\hat b_j)=\frac{\hat\Theta_{jj}+ o(1)}{{n}} 
.$$
\end{theorem}

\noindent
The condition $\beta_0+\hat\Theta_j/\sqrt{\hat\Theta_{jj}n} \in \mathcal B(d_n)$ implies that $\|\hat\Theta_j\|_0 =\mathcal O(d_n)$.
To this end, we refer to Lemma \ref{sparse} in Section \ref{subsec:oi} of the supplemental article \cite{sup}, which shows that sparsity in $\hat\Theta_j$ constructed using nodewise regression is guaranteed under random design. The condition $\|\hat\Theta_j\|_2=\mathcal O(1)$ replaces the eigenvalue condition we needed in the case of random design.

\subsection{Le Cam's bounds}
\label{subsec:ex.lr}
In this section, we provide an alternative approach, which makes another choice in the formulation of asymptotic efficiency.  This approach is based on Le Cam's arguments (see e.g. \cite{vdv}) rather than the Cram\'er-Rao bounds, and it allows us to show that the convergence of the de-sparsified estimator to the limiting normal distribution with smallest possible variance is locally uniform in the underlying unknown parameter, and the asymptotic variance of the de-sparsified estimator is smallest among the class of asymptotically linear estimators.
Furthermore, the result identifies the asymptotic bias of asymptotically linear estimators. 
A detailed comparison of the two approaches for deriving the lower bounds is deferred to Section \ref{subsec:conclusion}.
\par
We consider the setting from Section \ref{subsec:lr.random}, where the design matrix $X$ is random with the parameter of interest  being $g(\beta)=\beta_j$.

\begin{theorem}\label{lecam.bj}
Assume that conditions \ref{model}, \ref{design.normal} are satisfied, $\beta_0\in\mathcal B(d_n)$ with  $d_n=o(\sqrt{n}/\log p)$,
 $\|\Theta_j^0\|_0\leq d_n$ and $\Lambda_{\max}(\Sigma_0)=\mathcal O(1)$.
Assume that $\hat b_j$ is defined in \eqref{bj} with tuning parameters $\lambda \asymp \lambda_j \asymp \sqrt{\log p/n}.$ 
Then for every $\tilde\beta_n\in B_{}\left(\beta_0,\frac{c}{\sqrt{n}}\right)$ it holds
$$\frac{
\sqrt{n}(\hat b_j - \tilde\beta_n) }{
 (\Theta_{jj}^0)^{1/2}
} \stackrel{\tilde\beta_n}{\rightsquigarrow} \mathcal N( 0, 1 ).$$ 
Let $T_n$ be an asymptotically linear estimator with an influence function $l_{\beta_0}$:
\begin{equation}
\label{unbiased0}
T_n - \beta_j^0 = \frac{1}{n} \sum_{i=1}^n l_{\beta_0} (X^{(i)}, Y^{(i)}) + o_{P_{\beta_0}}(n^{-1/2}),
\end{equation}
where $\mathbb E l_{\beta_0} (X^{(i)}, Y^{(i)}) =0$ and $\emph{var}(l_{\beta_0} (X^{(i)}, Y^{(i)})) =: V_{\beta_0} < \infty.$
Assume  that for all $h\in\mathbb R^p$ and $i=1,\dots,n$ it holds
\begin{equation}
\label{lr.nobias}
\mathbb E l_{\beta_0} (X^{(i)}, Y^{(i)}) \epsilon_i h^T X^{(i)} - h_j = o(1).
\end{equation}
Then
$$V_{\beta_0} \geq \Theta_{jj}^0+o(1).$$

\end{theorem}

\subsection{Discussion of the conditions}
\label{subsec:nec}
We briefly discuss the conditions assumed to obtain the above results. To establish asymptotic efficiency of the de-sparsified estimator, we  considered conditions analogous to the conditions assumed in \cite{vdgeer13}. These include a sparsity condition on the parameter $\beta_0$ of order $o(\sqrt{n}/\log p)$,  conditions on the covariance matrix  $\Lambda_{\min}(\Sigma_0)=\mathcal O(1),
\|\Sigma_0\|_\infty =\mathcal O(1)$, sparsity of the precision matrix and a Gaussianity assumption on the rows on the precision matrix. 
Unlike in \cite{vdgeer13}, we assume Gaussianity of the design matrix and the error; this condition was needed to derive the lower bounds. 
In addition to the conditions from \cite{vdgeer13}, we also assume boundedness of $\ell_2$-norm of $\beta_0$, which follows if the signal to noise ratio is bounded as argued in Section \ref{subsec:strong}.
Condition \eqref{lr.nobias} from Theorem \ref{lecam.bj} is a variant of asymptotic unbiasedness which is known to be satisfied in many traditional settings (see e.g. \cite{vdv}). The condition is discussed in more detail in Section \ref{sec:lecam} below.
\par 
Our analysis requires the sparsity condition $s=o(\sqrt{n}/\log p).$ 
This condition is essentially  necessary in the linear regression setting for construction of an asymptotically normal estimator, as argued in the following.
First observe that if the (slightly weaker) condition $s=\mathcal O(\sqrt{n}/\log p)$ is not satisfied, then there cannot exist an estimator $T_n$ of $\beta_j\in\mathbb R$ and a sequence $\sigma_n=\mathcal O(1)$ such that 
\begin{equation}\label{an}
\sqrt{n}(T_n-\beta^0_j)/\sigma_n\rightsquigarrow \mathcal N(0,1).
\end{equation}
Suppose that there exists an estimator $T_n$ that satisfies \eqref{an}. Then necessarily $\sqrt{n}(T_n-\beta_j^0)/\sigma_n=\mathcal O_P(1)$.
By similar reasoning as in \cite{zhou}
, we have under the conditions assumed the minimax rates for $\mathbb E|T_n - \beta^0_j|$ of order $\frac{1}{\sqrt{n}} + \frac{s\log p}{n}.$
But then 
necessarily $s\log p/{n}=\mathcal O(1/\sqrt{n}),$ which gives $s=\mathcal O(\sqrt{n}/\log p)$. This is only slightly weaker than the condition we require, $s=o(\sqrt{n}/\log p)$.
\par
Furthermore, for simplicity of presentation, we assumed that the variance of the noise is fixed at $\sigma_\epsilon=1$.
In general, we can include the parameter $\sigma_\epsilon$ as an unknown parameter in the model, and by orthogonality of the score corresponding to this parameter and the score corresponding to $\beta$, we can easily extend the arguments. The noise variance will then appear in both lower and upper bounds.

\section{Gaussian graphical models}
\label{subsec:ggm.all}
In this part, we consider efficient estimation of edge weights in undirected Gaussian graphical models. 
Gaussian graphical models have become a popular tool for representing dependencies within large sets of variables and have found application in areas such as neuroscience, biology and climate data analysis. In particular, Gaussian graphical models  
encode conditional dependencies between variables (nodes in the graph) by including an edge between two variables if and only if they are not independent given all the other variables. This corresponds to the problem of estimation of the precision matrix of a multivariate normal distribution, which we now introduce.
\begin{enumerate}[label=(C\arabic*),start=1]
\item\label{ggm.x}
Assume that the $n\times p$ matrix $X$ has independent rows $X^{(i)}$, $i=1,\dots,n$ which are $\mathcal N_p(0, \Sigma_0)-$distributed.
\end{enumerate}
Denote the precision matrix by $\Theta_0:=\Sigma_0^{-1}$, where the inverse of $\Sigma_0$ is assumed to exist. The matrix  $\Theta_0 \in \mathbb R^{p\times p}$ is unknown, but we  assume bounds on its row-sparsity (column-sparsity) $s_j:=\|\Theta_j^0\|_0$, where $\Theta_j^0$ is the $j$-th column of the precision matrix.
\par


\subsection{Methodology}

There have been several methods proposed for estimation of the precision matrix in the high-dimensional setting when $p\gg n$ (see \cite{glasso}, \cite{buhlmann}). These methods are based on regularization techniques and lead to estimators that are biased. De-biasing was then studied similarly as in the linear regression, and it was shown that de-biasing leads to estimators which are asymptotically normal. 
For our further analysis, we consider the de-sparsified nodewise Lasso  estimator proposed in \cite{jvdgeer15}. 
We show that this estimator is strongly asymptotically unbiased and reaches the lower bound on the variance derived in the previous section.

To introduce  the methodology, 
consider again the nodewise Lasso estimator $\hat\Theta=(\hat\Theta_1,\dots,\hat\Theta_p)$ defined in \eqref{node.lr}.
Define the de-sparsified nodewise Lasso (see \cite{jvdgeer15})
\begin{equation}\label{dsnw}
\hat  T := \hat \Theta + \hat\Theta^T  - \hat\Theta \hat\Sigma \hat\Theta.
\end{equation}
Furthermore, we write $\hat  T_{ij} := \hat \Theta_{ij} + \hat\Theta_{ji}  - \hat\Theta_i^T \hat\Sigma \hat\Theta_j$ for 
$i,j=1,\dots,p.$
The method and its asymptotic properties were studied in \cite{jvdgeer15}. 
The estimator $\hat \Theta_j$ can be shown to be asymptotically linear with a remainder term of small order $1/\sqrt{n}$, in particular, under condition \ref{ggm.x} and under $\max_{j=1,\dots,p}s_j=o(\sqrt{n}/\log p)$ it holds 
$$\hat T -\Theta_0= -\Theta_0^T (\hat\Sigma-\Sigma_0) \Theta_0 + \Delta,$$
where $\|\Delta \|_\infty=o_P(1/\sqrt{n}).$
 Thus, after normalization by $\sqrt{n}$ and by the (estimated) standard deviation, it follows that it is asymptotically standard normal and minimax optimal (see \cite{zhou}, \cite{jvdgeer15}). 
We investigate 
the question of ``regularity'' 
and asymptotic efficiency of the proposed estimator.

\subsection{Strong asymptotic unbiasedness of the de-sparsified nodewise Lasso}
\label{subsec:ggm.sau}
Suppose that the parameter $\Theta$ ranges over a parameter space $T\subset \mathbb R^{p\times p}$. 
We then define the parameter set
\begin{eqnarray*}
\mathcal G(d_1,\dots,d_p)&:=&\{\Theta\in T:\Theta = \Theta^T, \|\Theta_j\|_0 \leq C_1d_j, j=1,\dots,p,\\
&&\quad\quad\quad \quad\quad\;\; 1/\Lambda_{\min}(\Theta)\leq C_2,
\Lambda_{\max}(\Theta)\leq C_3\},
\end{eqnarray*}
for some universal constants $C_1,C_2,C_3>0.$
We also need to readjust the definition of a neighbourhood from \eqref{neigh}; hence in this section we let   
$$B(\Theta,\epsilon):=\{\tilde\Theta \in \mathcal G(d_1,\dots,d_p): \|\tilde\Theta-\Theta\|_F\leq \epsilon\}.$$
The following lemma shows that  $\hat T_{ij}$ is strongly asymptotically unbiased for estimation of $\Theta_{ij}^0$.

\noindent
\begin{lemma} \label{SAU} Let $i,j\in\{1,\dots,p\}$, assume that condition \ref{ggm.x} is satisfied and  $\Theta_{0}\in
\mathcal G(d_1,\dots,d_p)$ with $\max(d_i,d_j) = o\left( {\sqrt{n}}/{\log p} \right)$. 
Let $\hat T_{ij}$ be defined in \eqref{dsnw}, where $\hat\Theta_i,\hat\Theta_j$ are the $i$-th and $j$-th columns of the nodewise Lasso estimator with tuning parameters $\lambda_i\asymp\lambda_j \asymp \sqrt{\log p /n}$. 
Then $\hat T_{ij}$ is a strongly asymptotically unbiased estimator for $\Theta_{ij}^0$.
\end{lemma}


\subsection{Main results}
\label{subsec:ggm.ub}

We first derive an asymptotic lower bound for the variance of $T_n$ when $T_n$ is strongly asymptotically unbiased.
We restrict our attention to estimation of linear functionals of the precision matrix
 $\Theta_0$, $h(\Theta_0) = \textrm{tr}(\Psi \Theta_0),$ where $\Psi\in\mathbb R^{p\times p}$ is a known matrix.
We shall consider the case when $\Psi$
is of rank one, say $\Psi=\xi_1\xi_2^T$ for some vectors $\xi_1,\xi_2\in\mathbb R^p$. This corresponds to estimation of $g(\Theta_0) = \xi_1^T \Theta_0 \xi_2$, where $\xi_1,\xi_2\in\mathbb R^p$ are known vectors. 
\par
Contrary to previous sections, the high-dimensional parameter is a matrix, therefore instead of a vector direction $h$ we shall write the capital letter $H$ to denote a matrix direction in $\mathbb R^{p\times p}.$

\begin{theorem}\label{ggm.lc}
Assume condition \ref{ggm.x},  assume that
$\Theta_{0}\in\mathcal G(d_1,\dots,d_p)$ where
$\max_{j=1,\dots,p} d_j = o\left( {\sqrt{n}}/{\log p} \right)$ 
and $\Theta_0+ H/\sqrt{n}\in B_{}(\Theta_0,c/\sqrt{n})$.
Suppose that $T_n$ is a strongly asymptotically unbiased estimator of $g(\Theta) = \xi_1^T\Theta\xi_2$ at $\Theta_0\in\mathcal G(d_1,\dots,d_p)$
in the direction $H := \Theta_0 (\xi_1\xi_2^T+\xi_2\xi_1^T)\Theta_0/\sigma_{},$ where 
$$\sigma_{}^2:=
{\xi_1^T\Theta_0\xi_1 \xi_2^T\Theta_0\xi_2 +(\xi_1^T\Theta_0\xi_2)^2}
.$$ 
Then it holds 
$$\emph{var}_{\Theta_0}(T_n) \geq \frac{\sigma^2 - o(1)}{n}.$$
\end{theorem}

\noindent

As a corollary, consider estimation of $g(\Theta_0) = \Theta^0_{ij}$ for some fixed $(i,j)\in \{1,\dots,p\}^2.$
Then the worst sub-direction is given by $H:=(\Theta_i^0(\Theta_j^0)^T + \Theta_j^0(\Theta_i^0)^T )/\sigma_{}$
where $\sigma^2 := (\Theta^0_{ij})^2+\Theta^0_{ii}\Theta^0_{jj}$ and the corresponding lower bound is 
$((\Theta^0_{ij})^2+\Theta^0_{ii}\Theta^0_{jj})/n +o(1/n).$

%

\noindent
We now show that the de-sparsified estimator $\hat T_{ij}$ reaches the lower bound on the variance for the parameter 
of interest $g(\Theta_0)=\Theta_{ij}^0$.

\begin{theorem}\label{ggm.main}
Suppose that condition \ref{ggm.x} holds, $\Theta_{0}\in\mathcal G(d_1,\dots,d_p)$ where
$\max(d_i,d_j) = o\left( {\sqrt{n}}/{\log p} \right)$. 
Suppose that $\Theta_0+ H/\sqrt{n}\in B_{}(\Theta_0,c/\sqrt{n})$ for $H:= (\Theta^0_i(\Theta^0_j)^T + \Theta^0_j(\Theta^0_i)^T)/\sigma_{}.$
Let $\hat T_{ij}$ be defined in \eqref{dsnw}, where $\hat\Theta_i,\hat\Theta_j$ are the $i$-th and $j$-th columns of the nodewise Lasso estimator with tuning parameters $\lambda_i\asymp\lambda_j \asymp \sqrt{\log p /n}$. 
Then $\hat T_{ij}$ is a strongly asymptotically unbiased estimator of $\Theta_{ij}$ at $\Theta_0$
 and for any strongly asymptotically unbiased estimator $T$ of $\Theta_{ij}$ at $\Theta_0$ it holds 
$$\emph{var}_{\Theta_0}( T) \geq \frac{\Theta^0_{ii}\Theta^0_{jj}+(\Theta^0_{ij})^2+ o(1)}{{n}}, \quad  \emph{var}_{\Theta_0}(\hat T_{ij}) =\frac{\Theta^0_{ii}\Theta^0_{jj}+(\Theta^0_{ij})^2+ o(1)}{{n}}.$$
\end{theorem}

The condition $\Theta_0+ H/\sqrt{n}\in B_{}(\Theta_0,c/\sqrt{n})$ for $H= (\Theta^0_i(\Theta^0_j)^T + \Theta^0_j(\Theta^0_i)^T)/\sigma_{}$ ensures that perturbation of $\Theta_0$ along the worst possible sub-direction $H$ lies within the model. 
This also implies that $\|H_k\|_0 \leq 2C_1 d_k, k=1,\dots,p$, which in turn implies that 
necessarily $\|\Theta_{i}^0\|_0 = \mathcal O(d_k), \|\Theta_{j}^0\|_0=\mathcal O(d_k)$ for $k=1,\dots,p$. 
Note that we only require sparsity in the $i$-th and $j$-th column of the precision matrix.
Furthermore, we must have $\|H\|_F \leq c.$ 
This is satisfied under the eigenvalue conditions 
noting that $\|H\|_F^2=\text{tr}(H^TH)$ and $\|\Theta_k^0\|_2 = \mathcal O(1)$ for $k=i,j.$ 

\subsection{Discussion of the conditions}
We comment on the conditions used to obtain the above results. The conditions under which we show asymptotic efficiency only include eigenvalue conditions on the true precision matrix, sparsity conditions on columns/rows of the precision matrix and Gaussianity of the observations $X^{(i)},i=1,\dots,n$. These conditions are almost identical to conditions in \cite{vdgeer13} and \cite{jvdgeer15}, with the exception of Gaussianity which was used for deriving the lower bounds. In particular, the   condition on row sparsity required is  the same as for the linear model: $s=o(\sqrt{n}/\log p).$ In view of the results on minimax rates for estimation of elements of precision matrices (which are derived in \cite{zhou}), the condition $s=o(\sqrt{n}/\log p)$ is necessary for asymptotically normal estimation, which follows by similar reasoning as for the linear regression.

\section{Le Cam's bounds for general models} 
\label{sec:lecam}

In this section, we provide an extension to general non-linear models and a general parameter of interest. This is achieved via adjustment of Le Cam's arguments on asymptotic efficiency to the high-dimensional setting.
Let $X^{(1)},\dots,X^{(n)}$  be i.i.d. with distribution $P_{\beta_{n,0}}:{\beta_{n,0}}\in \mathcal B$ where $\mathcal B$ is an open convex subset of $\mathbb R^p.$
We consider the parameter set 
$$ \mathcal B(d_n):=\{ \beta \in \mathcal B: \|\beta\|_0 \leq C_1 d_n , \|\beta\|_2 \leq C_2\},$$
where $C_1,C_2=\mathcal O(1)$ and $d_n$ is a known sequence that will be specified later.
Suppose that the parameter of interest is $g(\beta_{n,0})$ for some function $g:\mathcal B \rightarrow \mathbb R$.
Assume that for an estimator $T_n$ of $g({\beta_{n,0}}),$ we can show asymptotic linearity: there exists a real-valued function 
$l_{{\beta_{n,0}}}$ on $\mathcal X$ (an influence function) and some sequence ${\beta_{n,0}}$ such that 
$$T_n-g({\beta_{n,0}}) = \frac{1}{n}\sum_{i=1}^n l_{{\beta_{n,0}}}(X^{(i)}) + o_{P_{\beta_{n,0}}}(n^{-1/2}),$$
where ${ P_{\beta_{n,0}} } l_{{\beta_{n,0}}} = 0$ and the variance $V_{{\beta_{n,0}}} :={ P_{\beta_{n,0}} }l_{{\beta_{n,0}}}^2<\infty$.
Under the conditions of the central limit theorem, the asymptotic linearity implies that 
\begin{equation}\label{pc}
\sqrt{n}(T_n-g({\beta_{n,0}}))/V_{{\beta_{n,0}}}^{1/2} \stackrel{{\beta_{n,0}}}{\rightsquigarrow }\mathcal N(0,1).
\end{equation} 
For asymptotically linear estimators, we thus have the ``asymptotic variance'' $V_{{\beta_{n,0}}} ={ P_{\beta_{n,0}} }l_{{\beta_{n,0}}}^2.$
We shall need some conditions on the differentiability of $g$ and the score function. Furthermore, we shall need a  Lindeberg's condition related to the influence and score function. 
Assume that $P_\beta$ is dominated by some $\sigma$-finite measure for all $\beta$ in the parameter space and denote the corresponding probability densities by $p_\beta.$ 
We denote the log-likelihood by $\ell_\beta(x) := \log p_\beta(x)$ and the score function by $s_\beta(x):= \frac{\partial \ell_\beta(x)}{\partial \beta}$ for all $x\in\mathcal X$.

\begin{enumerate}[label=(D\arabic*),start=1]
\item \label{difg}
\textit{(Differentiability of $g$)}
Suppose that for 
a given  $\tilde\beta_n\in B_{}({\beta_{n,0}},\frac{c}{\sqrt{n}})$ it holds
$$\sqrt{n}(g(\tilde\beta_n) -g({\beta_{n,0}}))= h^T\dot g({\beta_{n,0}})+o(1),$$
where $h = \sqrt{n}(\tilde\beta_n-\beta_{n,0}).$
\item\label{score}
\textit{(Differentiability of the score)}
Suppose that the score function $\beta \mapsto s_\beta$ is twice differentiable and the 
second derivative satisfies $\|\ddot s_\beta\|_\infty\leq L$ for some universal constant $L>0$ and for all $\beta\in\mathcal B(d_n).$
Let $I_{{\beta_{n,0}}}:={ P_{\beta_{n,0}} } s_{{\beta_{n,0}}} s_{{\beta_{n,0}}}^T$  and assume that  $\Lambda_{\max} (I_{{\beta_{n,0}}} )= \mathcal O(1),$  $1/\Lambda_{\min} (I_{{\beta_{n,0}}} )= \mathcal O(1)$ and
\begin{equation}\label{fi}
\|\frac{1}{n}\sum_{i=1}^n  \dot s_{{\beta_{n,0}}} + I_{{\beta_{n,0}}}   \|_\infty = \mathcal O_P(\lambda),
\end{equation}
for some $\lambda>0.$ Suppose that $d_n=o(\max\{1/\lambda,n^{1/3}\}).$ 
\item\label{lindeberg}
\textit{(Lindeberg's condition)} Denote $f_{\beta_{n,0}}(x):=l_{{\beta_{n,0}}}(x) + h^Ts_{{\beta_{n,0}}}(x)$ for $x\in\mathbb R^p.$
Suppose that for all $\epsilon>0$ 
\begin{equation}
\lim_{n\rightarrow \infty} { P_{\beta_{n,0}} } f_{\beta_{n,0}}^2
\mathbf  1_{|f_{\beta_{n,0}} |> \epsilon \sqrt{n}} =0,
\end{equation}
and assume that $V_{{\beta_{n,0}}}:={ P_{\beta_{n,0}} }l_{{\beta_{n,0}}}^2 =\mathcal O(1)$ and $ 1/V_{{\beta_{n,0}}}=\mathcal O(1)$.
\end{enumerate}
Condition \ref{difg} is a differentiability condition on $g$; an analogous condition is assumed in the first approach through Cram\'er-Rao bounds. Condition \ref{score} is a differentiability condition on the score, which is used to obtain a Taylor expansion of the likelihood. Furthermore, the condition \eqref{fi} guarantees that $-\frac{1}{n}\sum_{i=1}^n  \dot s_{{\beta_{n,0}}}(X^{(i)})$ is a good estimator of the Fisher information in supremum norm. This can be verified e.g. for linear regression with $\lambda\asymp \sqrt{\log p/n}.$
Condition \ref{score} further assumes the sparsity  $d_n=o(\max\{1/\lambda,n^{1/3}\})$, which guarantees that the likelihood ratio expansion approximately holds.
Finally, condition \ref{lindeberg} is a Lindeberg's condition which is needed to conclude asymptotic normality of certain quantities, since in Theorem \ref{lecam} below we do not require any distributional assumption. This condition can be verified for particular models.

\begin{theorem}%
\label{lecam}
Let $g:\mathcal B \rightarrow \mathbb R$ and suppose that for some fixed sequence  ${\beta_{n,0}}\in\mathcal B(d_n)$ it holds
\begin{equation} \label{al}
T_n-g({\beta_{n,0}}) = \frac{1}{n}\sum_{i=1}^n l_{{\beta_{n,0}}}(X^{(i)}) + o_{P_{{\beta_{n,0}}}}(n^{-1/2}),
\end{equation}
where ${ P_{\beta_{n,0}} } l_{{\beta_{n,0}}} =0$.
For some fixed constant $c>0$,  let $\tilde\beta_n\in B_{}({\beta_{n,0}},\frac{c}{\sqrt{n}})$ and denote 
$h:=\sqrt{n}(\tilde\beta_n-{\beta_{n,0}})$. Suppose that conditions \ref{difg}, \ref{score} and \ref{lindeberg} are satisfied.
Then 
it holds 
$$\frac{
\sqrt{n}(T_n - g({\beta_{n,0}}+\frac{h}{\sqrt{n}})) - ({ P_{\beta_{n,0}} }(l_{{\beta_{n,0}}} h^T s_{{\beta_{n,0}}}) -h^T\dot g({\beta_{n,0}}))}{
 V_{{\beta_{n,0}}}^{1/2}
} \stackrel{{{\beta_{n,0}} + \frac{h}{\sqrt{n}}}}{\rightsquigarrow} \mathcal N( 0,1).$$
\end{theorem}
The result of Theorem \ref{lecam} contains a bias term  ${ P_{\beta_{n,0}} }(l_{{\beta_{n,0}}} h^T s_{{\beta_{n,0}}}) -h^T\dot g({{\beta_{n,0}}})$ which depends on $h$.
Now consider that the bias term in the result of the theorem above vanishes, i.e. that the following condition on the score function $s_{{\beta_{n,0}}}$ and the function $l_{{\beta_{n,0}}}$ is satisfied: 
for every $h\in\mathbb R^p$ 
it holds that
\begin{equation}\label{nobias}
{ P_{\beta_{n,0}} }(l_{{\beta_{n,0}}} h^T s_{{\beta_{n,0}}}) -h^T\dot g({{\beta_{n,0}}})=o(1).
\end{equation}
The condition \eqref{nobias} is a variant of asymptotic unbiasedness which is known to be satisfied in many traditional settings.
If  condition \eqref{nobias} is satisfied, then
the Cauchy-Schwarz inequality implies
\begin{eqnarray*}
(h^T \dot g({{\beta_{n,0}}}))^2 
\leq  V_{{\beta_{n,0}}} h^T I_{{\beta_{n,0}}} h+o(V_{\beta_{n,0}}^{1/2} (h^TI_{\beta_{n,0}}h)^{1/2}).
\end{eqnarray*}
Hence this implies a lower bound on the asymptotic variance $V_{{\beta_{n,0}}}$ of an asymptotically linear estimator as follows
\begin{equation}\label{ub}
V_{{\beta_{n,0}}} \geq 
(h^T \dot g({{\beta_{n,0}}}))^2 / h^T I_{{\beta_{n,0}}} h +o(V_{\beta_{n,0}}^{1/2}/(h^TI_{\beta_{n,0}}h)^{1/2}).
\end{equation}
Assuming that the inverse of $I_{{\beta_{n,0}}}$ exists, 
the right-hand side of \eqref{ub} is maximized at $h=I_{{\beta_{n,0}}}^{-1}\dot g({{\beta_{n,0}}}),$
provided that ${\beta_{n,0}}+h/\sqrt{n} \in B({\beta_{n,0}},c/\sqrt{n}).$
 Hence we obtain the following lower bound on the asymptotic variance
$$V_{{\beta_{n,0}}} \geq \dot g({{\beta_{n,0}}}) ^T I_{{\beta_{n,0}}}^{-1} \dot g({{\beta_{n,0}}})
+o(V_{\beta_{n,0}}^{1/2} (\dot g({{\beta_{n,0}}}) ^T I_{{\beta_{n,0}}}^{-1} \dot g({{\beta_{n,0}}}))^{1/2}).$$
We summarize this simple claim in the lemma below. 
\begin{lemma}\label{lb2}
Let $T_n$ satisfy \eqref{al} with $V_{\beta_{n,0}}=\mathcal O(1), 1/\Lambda_{\min}(I_{\beta_{n,0}})=\mathcal O(1)$ 
and for every $h\in\mathbb R^p$ it holds that
\begin{equation}\label{nob}
{ P_{\beta_{n,0}} }(l_{{\beta_{n,0}}} h^T s_{{\beta_{n,0}}}) -h^T\dot g({{\beta_{n,0}}})=o(1),
\end{equation}
then if ${\beta_{n,0}}+I_{{\beta_{n,0}}}^{-1}\dot g({{\beta_{n,0}}})/\sqrt{n} \in B({\beta_{n,0}},c/\sqrt{n}),$ it holds that 
$$V_{{\beta_{n,0}}} \geq \dot g({{\beta_{n,0}}}) ^T I_{{\beta_{n,0}}}^{-1} \dot g({{\beta_{n,0}}})+o(1).$$
\end{lemma}
Theorem \ref{lecam} in conjunction with Lemma \ref{lb2} gives the result summarized in Corollary \ref{lecamcor2} below.

\begin{cor}\label{lecamcor2}
Suppose that conditions of Theorem \ref{lecam} and the condition \eqref{nob} are satisfied and that ${\beta_{n,0}}+I_{{\beta_{n,0}}}^{-1}\dot g({{\beta_{n,0}}})/\sqrt{n} \in B({\beta_{n,0}},c/\sqrt{n})$. Then
\begin{equation}\label{reg}
\sqrt{n}(T_n-g({\beta_{n,0}} + h/\sqrt{n}))/V_{{\beta_{n,0}}}^{1/2} \stackrel{{\beta_{n,0}}+h/\sqrt{n}}{\rightsquigarrow} \mathcal N(0,1),
\end{equation}
where
$$V_{{\beta_{n,0}}} \geq \dot g({{\beta_{n,0}}}) ^T I_{{\beta_{n,0}}}^{-1} \dot g({{\beta_{n,0}}})+o(1).$$
\end{cor}
The corollary implies that asymptotic efficiency is attained by an estimator which is asymptotically linear with an influence function $l_{\beta}=\dot g(\beta)^T I_{\beta}^{-1}s_{\beta}$, provided that it satisfies the condition \eqref{nob}.
\par
We have already shown how these results can be applied to the linear regression setting in Section \ref{subsec:ex.lr}.
We remark that the result of Theorem \ref{lecam} is not directly applicable to Gaussian graphical models, where the unknown parameter has overall sparsity $ps,$ where $s=o(\sqrt{n}/\log p).$

\begin{remark}
The sparsity condition $d_n=o(n^{1/3})$ arises when considering Taylor expansion of the log-likelihood for \textit{general} models. Hence, when there is some special structure in the log-likelihood function, weaker sparsity conditions might be possible. For instance, for linear regression setting, the Hessian of the log-likelihood does not depend on the unknown parameter $\beta_0$, hence in that case by inspection of the likelihood expansion in the proof of Theorem \ref{lecam}, we see that the condition $d_n=o(\sqrt{n/\log p})$ is sufficient. 
\end{remark}



\section{Conclusions}
\label{subsec:conclusion}

In this paper we have proposed a framework for studying asymptotic efficiency in high-dimensional models. 
We adopted a semi-parametric point of view: we concentrated on one dimensional functions of a high-dimensional parameter for which the lower bounds were derived.
The semi-parametric efficiency bounds we obtained correspond to the efficiency bounds for parametric models. However, the treatment for high-dimensional models  required more elaborate analysis due to the models changing with $n$ and assumed sparsity of the model. 
\par 
We further considered construction of estimators attaining the lo\-wer bo\-unds. 
We showed that indeed construction of asymptotically efficient estimator is possible: a de-sparsified estimator in linear regression and Gaussian graphical models is asymptotically efficient for estimation of certain simple functionals. 
 Our analysis identified the theoretical conditions on the parameter sparsity and further conditions on the model under which asymptotic efficiency may be shown. %

\par 
\vskip0.2cm
\textit{Comparison of the two approaches}. The analysis was done in two ways: in the spirit of asymptotic Cram\'er-Rao bounds and 
Le Cam's bounds (\cite{vdv}). These are strongly related: both define a restricted set of estimators which are in some sense asymptotically unbiased and claim lower bounds for any estimator in this class.
\par 
However, the two lines of work are not directly comparable as they are different results under different assumptions.
Le Cam's bounds give a lower bound on \textit{asymptotic variance}, while the Cram\'er-Rao bounds give a bound on the \textit{variance} of an estimator.  
We formulated Le Cam's approach for a general sparse model, while the Cram\'er-Rao bounds were only considered for the linear regression and Gaussian graphical models. 
Apart from this, the main results arising from the two approaches also present some differences in the assumptions. 
For the Le Cam's-type results, we assumed a stronger sparsity condition of order $d_n=o({n}^{1/3}/\log p)$ 
because of the Taylor expansion of the likelihood. 
However, for the linear regression setting, the sparsity condition can be improved to $d_n=o(\sqrt{n}/\log p)$, which is the same as in the Cram\'er-Rao bounds. For Gaussian graphical models, Le Cam's approach as formulated in this paper cannot be directly used, unlike the approach through the Cram\'er-Rao bounds.

\par\vskip0.2cm
\textit{Extensions}. 
 Our results on upper bounds are presented for the case when the parameter of interest is a single entry of the high-dimensional parameter or a linear combination with e.g. bounded $\ell_1$-norm. 
It is interesting to note some relations to literature on minimax rates. One  question is whether asymptotic efficiency can be attained e.g. for estimation of linear functionals in linear regression when the linear combination $\xi$ is sparse. Our results needed that $\|\xi\|_1$ remains bounded. 
Some recent works  on high-dimensional models further consider estimation of more complicated, non-sparse functionals (in linear regression \cite{cai.guo}, for Gaussian sequence models \cite{tsy}). These results are however of a different nature. 
Consider for instance estimation of $\sum_{i=1}^p \beta_i$ in high-dimensional linear regression. In this case, the parametric rate cannot be achieved (\cite{cai.guo}) and thus it remains unclear what can be said about ``asymptotic efficiency''. 
\par
Furthermore, we have treated the case of a one-dimensional parameter of interest, though the analysis might be extended to settings when the parameter of interest is higher-dimensional (of a fixed dimension).
Finally, our analysis considered particular examples of de-sparsified estimators, however, other estimators which are in some sense equivalent to these de-sparsified estimators are applicable.

%% file: proof.tex
\newpage

\setcounter{page}{1}
\begin{frontmatter}
\setattribute{journal}{name}{}
\title{Supplement to ``Semi-parametric efficiency bounds for high-dimensional models''}
\runtitle{Efficiency bounds for high-dimensional models}
\author{\fnms{Jana} \snm{Jankov\'a}\corref{Jana Jankov\'a}\ead[label=e1]{jankova@stat.math.ethz.ch}
}
\and
\author{\fnms{Sara} \snm{van de Geer}\corref{Sara van de Geer}\ead[label=e2]{geer@stat.math.ethz.ch}}
\begin{abstract}
This supplement contains the proofs. Section \ref{sec:prelim} summarizes some preliminary material on concentration of measure.
Section \ref{sec:proof.oi} contains the proofs of Section \ref{subsec:strong}: Strong oracle inequalities for the Lasso.
Section \ref{subsec:oi} contains strong oracle inequalities for the nodewise Lasso.
Section \ref{sec:lrproof}
contains proofs for Section \ref{subsec:lr.all}. 
Section \ref{sec:ggm.lbp} contains proofs for Section \ref{subsec:ggm.all}: Gaussian graphical models.
In Section \ref{sec:lcp} we give the proofs for Section \ref{sec:lecam}: Le Cam's bounds for general models.
Proofs for Section \ref{subsec:oi} are deferred to Appendix \ref{sec:proof.ois} and some technical lemmas are deferred to Appendices \ref{sec:A} and \ref{sec:B}.
\end{abstract}

\end{frontmatter}

\section{Concentration inequalities for sub-exponential random variables} 
\label{sec:prelim}

In  this preliminary section, we recall some results on concentration results for sub-exponential random variables (for the definition of a sub-exponential random variable, see Section 14.2.1 in \cite{hds}).
\noindent
Lemma \ref{con} below is a version of Lemma 14.13 in \cite{hds}.
\begin{lemma}\label{con}
Let $Z_1,\dots,Z_n$ be independent random variables with values in some (measurable) space $\mathcal Z$ and $\gamma_1,\dots,\gamma_p$ be real-valued functions on $\mathcal Z$ satisfying, for $j=1,\dots,p$,
$$\mathbb E\gamma_j(Z_i)=0,\;\;\; \mathbb Ee^{|\gamma_{j}(Z_i)|/K} \leq M_1, \;\;\;\forall i=1,\dots,n, $$
where $ M_1>0$ is a universal constant and $K>0$.\\
Then there exists a universal constant $M_2$ such that  for all $t>0$ we have with probability 
at least $1-e^{-nt}$ that
$$\max_{j=1,\dots,p} \left\lvert\frac{1}{n}\sum_{i=1}^n\gamma_j(Z_i)\right\lvert
\leq M_2 K t + \sqrt{2t} + \sqrt{\frac{2\log (2p)}{n}} + \frac{M_2 K\log (2p)}{n}.$$
%
\end{lemma}

\noindent
The following lemma is a version of Corollary 14.1 in \cite{hds}.
\begin{lemma}\label{econ}
Assume the conditions of Lemma \ref{con}.
Then for all $m=1,2,\dots$ it holds
\begin{eqnarray*}
&& \mathbb E \left(\max_{j=1,\dots,p}  |\frac{1}{n}\sum_{i=1}^n \gamma_j(Z_i)|^m  \right)  
\\
&&
\leq \;\;\;\left(  \sqrt{\frac{2\log (2p + e^{m-1}  -p )}{n}} + \frac{M_2 K \log (2p + e^{m-1} -p )}{n} \right)^m,
\end{eqnarray*}
where $M_2>0$ is a universal constant.
\end{lemma}

\section{Proofs for Section \ref{subsec:strong}: Strong oracle inequalities for the Lasso}
\label{sec:proof.oi}
In this section we prove the oracle inequality for the Lasso as stated in Theorem \ref{strong2}. 
We need the following preliminary Lemmas \ref{cc}, \ref{ee} and \ref{oi}. 
Lemma \ref{cc} below gives sufficient conditions under which  the compatibility condition is satisfied.  
Lemma \ref{ee} is a concentration result for sub-Gaussian random variables as in Section 14 in \cite{hds}. 
Lemma \ref{oi} is a version of Theorem 6.1 in \cite{hds}. 
Recall that we denote $\hat\Sigma:= X^TX/n$ and $\Sigma_0:=\mathbb EX^TX/n.$
We recall the definition of the compatibility condition (see Section 6.13 in \cite{hds}).
Let $S:=\{i:\beta_i\not = 0\}$ and let $s=|S|.$  
We denote by $\beta_S$ the vector obtained from the vector $\beta\in\mathbb R^p$ by replacing entries corresponding to the indices in $S$ by zeros. 

\begin{definition}\label{cc.def}
We say that a matrix $\Sigma_0$ satisfies the compatibility condition with a  constant $\phi$
if 
$$\phi:=\min\left\{\frac{s\beta^T \Sigma_0 \beta}{\|\beta_S\|_1^2}: \|\beta_{S^c}\|_1 \leq 3\|\beta_S\|_1 \right\}>0.$$
\end{definition}

\begin{lemma}[Corollary 6.8 in \cite{hds}]
\label{cc}
 Suppose  that  $\Lambda_{\min}({\Sigma_0}) \geq L $ for a universal constant $L>0.$ 
Then $\Sigma_0$ satisfies the compatibility condition with the constant $L.$
Further suppose that $s\lambda=o(1)$. Then on the set $\|\hat\Sigma-\Sigma_0\|_\infty \leq \lambda$, 
for all $n$ sufficiently large, $\hat\Sigma$ satisfies the compatibility condition with the constant
$L/2$. 
\end{lemma}

\begin{lemma}\label{ee}
Suppose that $\epsilon_i,i=1,\dots,n$ are sub-Gaussian random variables with a universal constant $K_1$ 
and that $X^{(i)},i=1,\dots,n$ are independent random vectors with sub-Gaussian entries,
with a universal constant $K_2.$ Suppose that $\epsilon_i$ and $X^{(i)}$ are independent for $i=1,\dots,n$ and $\log p /n=o(1).$
Then  there exists a constant $c_1$ such that for all $\tau>1$ 
$$P\left(\|\epsilon^T X\|_\infty/n
\geq  c_1\tau\sqrt{\frac{\log (2p)}{n}} \right) \leq (2p)^{-\tau^2}.$$

\end{lemma}

\begin{proof}
We apply Lemma \ref{con} with
$\gamma_j(Z_i) = \epsilon_i X_{ij}$ for $i=1,\dots,n$ and $j=1,\dots,p.$ 
Then 
$\mathbb E\epsilon_i X_{ij}=\mathbb EX_{ij}\mathbb E(\epsilon_i|X_{ij})=0$, where we used independence of $\epsilon_i$ and $X^{(i)}.$
By sub-Gaussianity of $\epsilon_i$ and $X_{ij}$ and by the Cauchy-Schwarz inequality it follows
\begin{eqnarray*}
\mathbb Ee^{|\epsilon_i X_{ij}|/\max\{K_1,K_2\}}
& \leq &
 \mathbb Ee^{|\epsilon_i|^2/(2\max\{K_1,K_2\}^2) + |X_{ij}|^2/(2\max\{K_1,K_2\}^2)} 
\\
&\leq & (\mathbb Ee^{|\epsilon_i|^2/\max\{K_1,K_2\}^2})^{1/2}
(\mathbb Ee^{|X_{ij}|^2/\max\{K_1,K_2\}^2})^{1/2}
\\ 
&=&\mathcal O(1).
\end{eqnarray*}
Consequently, by Lemma \ref{con} and since $\log p /n=o(1),$ there exists a constant $c_1$ such that for all $\tau>1$ 
$$P\left(\max_{j=1,\dots,p} \left\lvert\frac{1}{n}\sum_{i=1}^n\epsilon_i X_{ij}\right\lvert
\geq  c_1\tau\sqrt{\frac{\log (2p)}{n}} \right) \leq (2p)^{-\tau^2}.$$

\end{proof}

\noindent
Finally, we give an oracle inequality for the Lasso. The proof may be found in \cite{hds}.
\begin{lemma}[a version of Theorem 6.1 in \cite{hds}]\label{oi}
Consider the Lasso estimator $\hat\beta$  defined in \eqref{lasso} with a tuning parameter $\lambda\geq 2\lambda_0.$ Suppose that
$s\lambda=o(1)$ and that $\Lambda_{\min}(\Sigma_0) \geq L$  for some universal constant $L>0.$
Then 
 on the set 
$$\mathcal T:=\{\|\epsilon^T X\|_\infty/n \leq \lambda_0, \|\hat\Sigma-\Sigma_0\|_\infty \leq \lambda_0\}$$
it holds 
$$\|\hat\beta-\beta_0\|_1 \leq 16\lambda{s}/{L}.$$

\end{lemma}

\noindent 
We are now ready to prove Theorem \ref{strong2}.

\begin{proof}[Proof of Theorem \ref{strong2}]
First we summarize the oracle inequality for the Lasso which holds with high probability. 
Let $\mathcal T_1:=\{\|\epsilon^T X\|_\infty/n \leq \tau c_1\sqrt{\log p/n}\}$ for some $\tau>1$ and for some suitable constant $c_1>0.$
By Lemma \ref{ee}, for the complementary set $\mathcal T_1^c$  it holds
that $P_{\beta_0}(\mathcal T_1^c)\leq (2p)^{-\tau^2 }$.\\
Let $\mathcal T_2 := \{\|\hat\Sigma-\Sigma_0\|_\infty \leq \tau c_1\sqrt{\log p/n}\}.$
Then by Lemma \ref{con}, taking $\gamma_{i,j}(X^{(k)})$ $ = e_i^T (X^{(k)} (X^{(k)})^T-\Sigma_0)e_j$ for $k=1,\dots,n$ and $i,j=1,\dots,p$, it follows that 
$P_{\beta_0}(\mathcal T_2^c) \leq (2p)^{-\tau^2}$.
Denote $\mathcal T:=\mathcal T_1 \cap \mathcal T_2;$
then $P_{\beta_0}(\mathcal T^c) \leq 2(2p)^{-\tau^2}.$
By Lemma \ref{oi}, when $\lambda \geq 2\lambda_0:=2\tau c_1\sqrt{\log p/n}$, on the set $\mathcal T$
it holds that
$\|\hat\beta-\beta_0\|_1\leq 16\lambda {s}/{L}.$
\\\\
We now proceed to show that the oracle inequality for the Lasso holds also in expectation.
The definition of $\hat \beta$ gives
\begin{eqnarray*}
\|Y - X\hat\beta\|_n^2 + \lambda \|\hat\beta\|_1
\leq 
\|\epsilon\|_n^2 + \lambda\|\beta_0\|_1.
\end{eqnarray*}
Consequently, 
$$\|\hat \beta \|_1 \leq \|\epsilon\|_n^2/\lambda + \|\beta_0\|_1.$$
Then, and by the triangle inequality 
\begin{eqnarray*}
\|\hat \beta -\beta_0\|_1 \leq  \|\hat \beta\|_1 + \|\beta_0\|_1 \leq \|\epsilon\|_n^2/\lambda + 2\|\beta_0\|_1,
\end{eqnarray*}
and thus for any $k\in\{1,2,\dots\}$
\begin{eqnarray*}
\mathbb E_{\beta_0} \|\hat \beta -\beta_0\|_1^k
\leq 
\mathbb E_{\beta_0} (\|\epsilon\|_n^2/\lambda + 2\|\beta_0\|_1)^k. 
\end{eqnarray*}
Then by the inequality $|x+y|^k \leq 2^{k-1}(|x|^k + |y|^k), k\geq 1$, it follows 
\begin{eqnarray*}
\mathbb E_{\beta_0} (\|\epsilon\|_n^2/\lambda + 2\|\beta_0\|_1)^k
&\leq&
\mathbb E_{\beta_0}2^{k-1}\left( (\|\epsilon\|_n^2/\lambda)^k + (2\|\beta_0\|_1)^k\right)
\end{eqnarray*}
By assumption, the random variables $\epsilon_1,\dots,\epsilon_n$ are sub-Gaussian with universal constants. 
This also implies that $\frac{1}{n}\sum_{i=1}^n \text{var}({\epsilon_i}^2) = \mathcal O(1).$
Hence by Lemma \ref{econ} applied with $\gamma_j(Z_i) := \epsilon_i-\text{var}(\epsilon_i)$ ($j=1,i=1,\dots,n$) we have  
\begin{eqnarray*}
\mathbb E_{\beta_0}  (\|\epsilon\|_n^2)^k  
= 
\mathcal O((\frac{1}{n}\sum_{i=1}^n \text{var}({\epsilon_i}^2))^k)=\mathcal O(1).
\end{eqnarray*}
Next observe that by assumption we have $\|\beta_0\|_2=\mathcal O(1)$ and hence
$$\|\beta_0\|_1^k \leq (\sqrt{s}\|\beta_0\|_2)^k \leq \mathcal O(s^{k/2}).$$
We can thus conclude that
 \begin{eqnarray*}
\mathbb E_{\beta_0} (\|\epsilon\|_n^2/\lambda + 2\|\beta_0\|_1)^k
&\leq &
2^{k-1}\left( \mathbb E_{\beta_0}(\|\epsilon\|_n^2/\lambda)^k + \mathbb E_{\beta_0}(2\|\beta_0\|_1)^k\right)
\\
&\leq &
\mathcal O(s^{k/2}\lambda^{-k})
.
\end{eqnarray*}
 Hence we obtain a rough bound
\begin{eqnarray*}
(\mathbb E_{\beta_0} \|\hat \beta -\beta_0\|_1^k)^{1/k}
=\mathcal O(s^{1/2}\lambda^{-1})
.
\end{eqnarray*}
On the set $\mathcal T$ we have the oracle bound
$\|\hat \beta-\beta_0\|_1 =\mathcal O(s\lambda) $ and thus on the set $\mathcal T,$ $\|\hat \beta-\beta_0\|_1^k =\mathcal O(s^k\lambda^k)$. Otherwise (so also on the set $\mathcal T^c$) we have the rough bound
$\mathbb E_{\beta_0} \|\hat \beta -\beta_0\|_1^k
=\mathcal O(s^{k/2}\lambda^{-k}).$
Denote by $1_A$ the indicator function of a set $A.$
 Then it follows using the bounds holding on $\mathcal T$ and $\mathcal T^c$ and by the Cauchy-Schwarz inequality 
\begin{eqnarray*}
\mathbb E_{\beta_0}\|\hat \beta-\beta_0\|_1^k 
&=& \mathbb E_{\beta_0}\|\hat\beta-\beta_0\|_1^k 1_{\mathcal T } + \mathbb E_{\beta_0}\|\hat\beta-\beta_0\|_1^k 1_{\mathcal T^c }
\\
&\leq &
\mathcal O(s^k\lambda^k) 
+
\sqrt{\mathbb E_{\beta_0}\|\hat\beta-\beta_0\|_1^{2k}} \sqrt{\mathbb E_{\beta_0}1_{\mathcal T^c }} \\
&=&  \mathcal O(s^{k}\lambda^k) + \mathcal O(\sqrt{s^{k}\lambda^{-2k}}) \;\sqrt{P_{\beta_0}(\mathcal T^c)} \\
& \leq &   \mathcal O(s^k\lambda^k) + \mathcal O(s^{k/2}\lambda^{-k}) \sqrt{2}(2p)^{-\tau^2 /2} \\
&=& \mathcal O(s^k\lambda^k)
,
\end{eqnarray*}
where we used the assumption $\tau^2\geq \frac{2k\log (\sqrt{s}\lambda^2)}{\log p}$ which implies
 $$\mathcal O(s^{k/2}\lambda^{-k}) p^{-\tau^2 /2}=\mathcal O(s^k\lambda^k).$$
 Hence we conclude that there exists a constant $C_1$ such that
\begin{eqnarray}\label{expoi1}
(\mathbb E_{\beta_0} \|\hat \beta -\beta_0\|_1^k)^{1/k}
=C_1s^{}\lambda
.
\end{eqnarray}
\end{proof}

\section{Strong oracle inequalities for the nodewise Lasso}
\label{subsec:oi}
In this section, we derive a strong oracle inequality for the nodewise regression estimator $\hat\Theta$ of the inverse covariance matrix $\Theta_0$ defined in \eqref{node.lr}.
These results will be needed to show strong asymptotic unbiasedness and upper bounds on the variance of the de-sparsified Lasso. 
The proofs may be found in Appendix \ref{sec:proof.ois}. 
We require a sparsity condition on the inverse covariance matrix of the covariates. To this end, denote the sparsity of the $j$-th row/column of the matrix $\Theta_0$ by $s_j,$ 
i.e. 
$$s_j = \|\Theta_j^0\|_0.$$
First we remark that the paper \cite{vdgeer13} shows that under conditions \ref{model.sg}, \ref{design} and $s_j=o(n/\log p)$ it holds 
$\|\hat\Theta_j - \Theta_j^0\|_1=\mathcal O_P(s_j\lambda_j).$ 
We aim to show a stronger claim, $\mathbb E\|\hat\Theta_j - \Theta_j^0\|_1=\mathcal O(s_j\lambda_j).$ 
This is a more difficult task than for the linear regression, since one has to make sure that the estimate of one over the noise level, 
$1/\hat\tau_j^2$, does not blow up in expectation. 
Moreover, for further results we shall need not only the result for each $j=1,\dots,p$ but actually for the maximum over $j=1,\dots,p,$ that is, an oracle bound for 
$\mathbb E\max_{j=1,\dots,p}\|\hat\Theta_j - \Theta_j^0\|_1^k,$ where $k\in\{1,2,\dots\}.$
\\
We introduce further notation; we let 
$$\gamma_j^0:=\textrm{arg}\min_{\gamma\in\mathbb R^{p-1}} \mathbb E \|X_j-X_{-j}\gamma\|_n^2$$
and $$\tau_j^2 :=  \mathbb E \|X_j-X_{-j}\gamma_0\|_n^2.$$
The following lemma is similar to  Theorem \ref{strong2} for the mean $\ell_1$-error of the Lasso, however, we consider the $p$ Lasso estimators obtained from the nodewise regression and derive an upper bound on the maximum mean $\ell_1$-error, where the maximum is taken over the $p$ estimators.
\begin{lemma}\label{aux.gm1}
Assume that condition \ref{design} is satisfied, let $k\in\{1,2,\dots\}$ be fixed and assume that $\max_{j=1,\dots,p}{s_j\sqrt{\log p/n}} =o(1),j=1,\dots,p$.
Let $\hat\gamma_j$ be defined as in \eqref{nc} with tuning parameters $\lambda_j= c \sqrt{\log p/n}, j=1,\dots,p,$  for some sufficiently large constant $c>0$. Then it holds that
 \begin{eqnarray*}
[\mathbb E_{} \max_{j=1,\dots,p}\|\hat \gamma_j-\gamma^0_j\|_1^k]^{1/k}
= \mathcal O(\max_{j=1,\dots,p}{s_j}\lambda_j) 
.
\end{eqnarray*}

\end{lemma}

\noindent
The following lemma shows that the noise estimator $\hat\tau_j^2$ 
is a near-oracle estimator of $\tau_j^2=1/\Theta_{jj}^0$, and $1/\hat\tau_j^2$ is a near-oracle estimator of $1/\tau_j^2.$

\begin{lemma}\label{tau}
Assume that condition \ref{design} is satisfied, $\|\Sigma_0\|_\infty = \mathcal O(1)$, ${\max_{j=1,\dots,p} s_j\sqrt{\log p/n}} =o(1)$ and let $k\in\{1,2,\dots\}$ be fixed. 
Let $\hat\gamma_j,j=1,\dots,p$ be defined as in \eqref{nc} with tuning parameters $\lambda_j=c \sqrt{\log p/n},j=1,\dots,p,$ for some sufficiently large constant $c>0$. Then 
the following statements hold
\begin{enumerate}[label=\arabic*)]
\item \label{taudiff}
$[{\mathbb E \max_{j=1,\dots,p} |{\hat\tau_j^2} - {\tau_j^2}|^k}]^{1/k} = \mathcal O(\max_{j=1,\dots,p}\sqrt{s_j}\lambda_j),$
\item
\label{tau.d}
$[\mathbb E \max_{j=1,\dots,p}|\frac{1}{\hat\tau_j^2} -\frac{1}{\tau_j^2}|^k ]^{1/k}= \mathcal O(\max_{j=1,\dots,p}\sqrt{s_j}\lambda_j).$
\end{enumerate}
\end{lemma}
\noindent
Combination of the results in Lemmas \ref{aux.gm1} and \ref{tau} gives the following result for the mean $\ell_1$-error of the nodewise regression estimator $\hat\Theta$ defined in \eqref{node.lr}.

\begin{lemma}\label{the}
Assume that condition \ref{design} is satisfied, $\|\Sigma_0\|_\infty = \mathcal O(1)$, 
${\max_{j=1,\dots,p} s_j\sqrt{\log p/n}} =o(1)$ and let $k\in\{1,2,\dots\}$ be fixed.
Then for $\hat\Theta_j,j=1,\dots,p$ defined in \eqref{node.lr} with tuning parameters $\lambda_j =c\sqrt{\log p/n},j=1,\dots,p,$ 
for some sufficiently large constant $c>0$, it holds 
$$[\mathbb E_{}\max_{j=1,\dots,p}\|\hat\Theta_j - \Theta^0_j\|_1^k]^{1/k} = \mathcal O(\max_{j=1,\dots,p}s_j\lambda_j).$$
\end{lemma}

\noindent
We note that the statements of Lemmas \ref{aux.gm1}, \ref{tau}, \ref{the} would also hold  for single estimators for some fixed $j$ (without taking the maximum over $j$). Then the bounds would only depend on the tuning parameter for that particular estimator and on the sparsity $s_j.$
We also give the following lemma for on the sparsity in the nodewise Lasso estimator. It shows that under certain conditions, the estimator $\hat\gamma_j$ has sparsity of order $s_j$ with high probability.
\begin{lemma}\label{sparse}
Assume 
that condition \ref{design} is satisfied, ${\max_j s_j\sqrt{\log p/n}} =o(1)$, $\Lambda_{\max}(\Sigma_0)$ $=\mathcal O(1)$.
Let $\hat\gamma_j$ be defined as in \eqref{nc} with a tuning parameter $\lambda_j= c \sqrt{\log p/n}$ for some sufficiently large constant $c>0$. Then it holds that
$$\|\hat \gamma_j\|_0 =\mathcal O_P( s_j).$$
\end{lemma}

\section{Proofs for Section \ref{subsec:saulasso}: The de-sparsified Lasso}
\label{sec:lrproof}

\subsection{Proofs for Section \ref{subsec:saulasso}: Strong asymptotic unbiasedness of the de-sparsified Lasso}
\label{sec:saulassop}

\begin{proof}[Proof of Lemma \ref{SAUbxi}]
For the de-sparsified estimator $\hat b_\xi$ we have by simple algebra the equality
$$\hat b_\xi - \xi^T\beta_0 =
  \xi^T\hat\Theta X^T\epsilon / n
+
 \xi^T (\hat \Sigma\hat\Theta - I)^T (\hat \beta-\beta_0).
$$
Consider any $\beta_0\in\mathcal B(d_n)$. First note that 
$$\mathbb E_{\beta_0} \xi^T \hat\Theta^T X^T\epsilon / n =\mathbb E\mathbb E_{\beta_0}( \xi^T \hat\Theta^T X^T\epsilon / n|X)
=
\mathbb E\xi^T \hat\Theta^T X^T\mathbb E_{\beta_0}( \epsilon |X)/n=0.$$
We then have  by the definition of $\hat b_\xi $, by the H\"older's inequality and the Cauchy-Schwarz inequality
\begin{eqnarray*}
\mathbb E _{\beta_0}(\hat b_\xi - \xi^T\beta_0) &=&
 \underbrace{\mathbb E_{\beta_0} \xi^T\hat\Theta X^T\epsilon / n}_{=0} 
+
\mathbb E_{\beta_0} \xi^T (\hat \Sigma\hat\Theta - I)^T (\hat \beta-\beta_0)\\
&\leq &
\mathbb E_{\beta_0} \|\xi\|_1\|\hat \Sigma\hat\Theta - I\|_\infty \|\hat \beta-\beta_0\|_1 
\\
&\leq &
\|\xi\|_1 (\mathbb E_{\beta_0} \|\hat \Sigma\hat\Theta - I\|_\infty^2)^{1/2} (\mathbb E_{\beta_0} \|\hat \beta-\beta_0\|_1^2)^{1/2}. 
\end{eqnarray*}
By the Karush-Kuhn-Tucker conditions for $\hat\gamma_j,j=1,\dots,p$ (see \cite{vdgeer13}), we have 
$$\|\hat \Sigma\hat\Theta - I\|_\infty \leq \max_{j=1,\dots,p}\lambda_j /\hat\tau_j^2.$$ 
By Lemma \ref{tau} from Section \ref{subsec:oi} it follows that 
$$(\mathbb E\|\hat \Sigma\hat\Theta - I\|_\infty^2)^{1/2} \leq 
\max_{j=1,\dots,p}\lambda_j(\mathbb E_{\beta_0}\max_{j=1,\dots,p}1 /(\hat\tau_j^2)^2)^{1/2} =\mathcal O(\max_{j=1,\dots,p}\lambda_j ).$$
Next we apply Theorem \ref{strong2}. Conditions $\|\beta_0\|_0\leq d_n,$ $\|\beta_0\|_2=\mathcal O(1)$, \ref{model.sg}, \ref{design} and sparsity $d_n=o(\sqrt{n}/\log p)$ imply that conditions of Theorem \ref{strong2} are satisfied.
Hence
$(\mathbb E_{\beta_0} \|\hat \beta-\beta_0\|_1^2)^{1/2}=\mathcal O(s\lambda). $
Hence, and since $\|\xi\|_1=\mathcal O(1)$, and using the last display we obtain that
$$
\mathbb E _{\beta_0}(\hat b_\xi - \xi^T\beta_0) = \mathcal O( s\lambda \max_{j=1,\dots,p} \lambda_j) = o(1/\sqrt{n}),
$$
where we used the sparsity condition $s\leq d_n=o(\sqrt{n}/\log p).$
Thus we have shown  $\sqrt{n}(\mathbb E _{\beta_0}(\hat b_\xi - \xi^T\beta_0) ) = o(1)$. 
But then there exists $\delta_n\rightarrow 0$ such that
$\sqrt{n/\delta_n}(\mathbb E _{\beta_0}(\hat b_\xi - \xi^T\beta_0) ) = o(1)$ (take e.g.
 $\delta_n:= \sqrt{\sqrt{n}(\mathbb E _{\beta_0}(\hat b_\xi - \xi^T\beta_0) )}$) and hence the 
estimator $\hat b_{\xi}$ is strongly asymptotically unbiased with a rate $m_n := n/\delta_n$.


\end{proof}

\subsection{Proofs for Section \ref{subsec:lr.random}: Main results for random design}
\label{sec:lr.randomp}

Before proving the statement of Theorem \ref{crlb.lr}, we need auxiliary Lemmas \ref{aux1}, \ref{aux2} and \ref{a1}.
Throughout this section, we denote by $\phi$
the probability density function of a standard normal random
variable.
\begin{lemma}\label{aux1}
Let $Z\sim \mathcal N(0,1).$
Then for all $t\in\mathbb R$
$$\mathbb E\left[ e^{tZ-t^2/2} - 1 - tZ \right]^2 = e^{t^2}-1 -t^2.$$
Moreover, for $2t^2<1$ we have 
$$\mathbb E e^{t^2Z^2} = \frac{1}{\sqrt{1-2t^2}}.$$
\end{lemma}

\begin{proof}[Proof of Lemma \ref{aux1}]
By direct calculation
$$\mathbb E\left[ e^{tZ-t^2/2} \right]^2 = \mathbb E e^{2tZ-t^2} = e^{t^2},$$
$$\mathbb Ee^{tZ-t^2/2}=1$$
and
$$\mathbb EZe^{tZ-t^2/2}=t\mathbb Ee^{tZ-t^2/2}=t.$$
The first result of the lemma follows immediately.
The second result is also easily found by standard calculations:
$$\mathbb E e^{t^2Z^2} =\int e^{t^2z^2}\phi(z)dz =\int \phi(z\sqrt{1-2t^2})dz=\frac{1}{\sqrt{1-2t^2}}.$$
\end{proof}

\begin{lemma}\label{aux2}
Suppose that $u\in\mathbb R^p$ satisfies $2u^T \Sigma_0 u < 1.$
Let $Z=(X,Y)$, where $Y=X\beta+\epsilon$, $\epsilon\sim \mathcal N(0,I)$ independent of $X$, and $X\sim\mathcal N(0,\Sigma_0)$. Denote the corresponding probability density of $Z$ by $p_\beta$ and let $s_{\beta_0}(Z) := X^T\epsilon.$
Then it holds 
$$\mathbb E_{\beta_0}\left(\frac{p_{\beta_0+u}(Z) - p_{\beta_0}(Z)}{p_{\beta_0}(Z)}-s_{\beta_0}(Z)^T u\right)^2 = 
(1-2u^T \Sigma_0 u)^{-n/2} -1 -nu^T \Sigma_0 u
.$$
\end{lemma}

\begin{proof}[Proof of Lemma \ref{aux2}]
Denote the density of $Y$ given $X$ by $p_{\beta_0}(\cdot|X)$, i.e. for $y=(y_1,\dots,y_n)$
$$p_{\beta_0}(y|X):=\prod_{i=1}^n \phi(y_i-(X^{(i)})^T\beta_0)=\frac{1}{(2\pi)^{n/2}}e^{-(y-X\beta_0)^T (y-X\beta_0)/2},$$
where $\phi$ is the standard normal density. \\ 
Given $X,$ the random variable $\epsilon^T Xu$ is $\mathcal N(0,nu^T \hat\Sigma u)$-distributed.
It follows therefore from the first result of Lemma \ref{aux1} with $t^2=nu^T\hat\Sigma u$ that
\begin{eqnarray*}
&& \mathbb E_{\beta_0}\left[\left(\frac{p_{\beta_0+u}(Y-Xu|X) - p_{\beta_0}(Y|X)}{p_{\beta_0}(Y|X)}-s_{\beta_0}(Z)^T u\right)^2 \lvert X
\right]\\
&&= 
\mathbb E e^{nu^T \hat\Sigma u} -1 -nu^T \hat\Sigma u
.\end{eqnarray*}
Since $(X^{(i)})^T u\sim \mathcal N(0,u^T\Sigma_0 u)$ for $i=1,\dots,n$, we have by the second result of Lemma \ref{aux1}
$$\mathbb E_{\beta_0}e^{((X^{(i)})^T u)^2} = \frac{1}{\sqrt{1-2u^T \Sigma_0 u}}.$$
Hence
$$\mathbb E_{\beta_0} [e^{nu^T \hat\Sigma u} -1 -nu^T \hat\Sigma u]
=  (1-2u^T \Sigma_0 u)^{-n/2} -1 -nu^T \Sigma_0 u
,$$
from which  the result follows.
\end{proof}

\begin{lemma}\label{a1}
Suppose that $nu^T \Sigma_0 u=o(1).$ Then
$$(1-2u^T \Sigma_0 u)^{-n/2} -1 -nu^T \Sigma_0 u=o(nu^T\Sigma_0u).
$$
\end{lemma}
\begin{proof}
Since $nu^T \Sigma_0 u = o(1),$ we can use the following Taylor expansions of $\log$ and $\exp$ 
\begin{eqnarray*}
(1-2u^T \Sigma_0 u)^{-n/2} &=&
 e^{-n \log (1-2u^T\Sigma_0u)/2} = e^{nh^T\Sigma_0 u +o(nh^T\Sigma_0u)}\\
&=&1+nh^T\Sigma_0 h +o(nh^T\Sigma_0h).
\end{eqnarray*}
Hence
$$(1-2u^T \Sigma_0 u)^{-n/2} -1 -nu^T \Sigma_0 u=o(n u^T\Sigma_0 u).
$$

\end{proof}


\begin{proof}[Proof of Theorem \ref{crlb.lr}]
By assumption \eqref{g.nice} on the differentiability of $g$ and by strong asymptotic unbiasedness of $T_n$ at $\beta_0$ and 
$\beta_0+h/\sqrt{m_n}$, it follows 
\begin{eqnarray*}
h^T \dot g(\beta_0) 
&=& 
\sqrt{m_n} \left(  g(\beta_0+h/\sqrt{m_n}) - g(\beta_0)  \right) + o(1)
\\
&=&
\sqrt{m_n} \left( \mathbb E_{\beta_0 + h/\sqrt{m_n}} T_n(Z)  - \mathbb E_{\beta_0} T_n(Z)\right) +o(1),
\end{eqnarray*}
where $Z:=(X,Y)$. We denote probability density corresponding to $Z$ by $p_\beta$, i.e.
$$p_{\beta_0}(X,Y):=\frac{1}{(2\pi)^{n/2}}e^{-(Y-X\beta_0)^T (Y-X\beta_0)/2}.$$
Let $s_{\beta_0}(Z) := X^T\epsilon.$
We may rewrite the expressions to obtain
\begin{eqnarray*}
&&\sqrt{m_n} \left( \mathbb E_{\beta_0 + h/\sqrt{m_n}} T_n(Z)  - \mathbb E_{\beta_0} T_n(Z)\right) +o(1) 
\\
&&\quad\quad=
\sqrt{m_n}\int T_n(z) (p_{\beta_0 + h/\sqrt{m_n}}(z) - p_{\beta_0}(z))dz\\
& &\quad\quad=
\mathbb E_{\beta_0} T_n(Z) \frac{p_{\beta_0+{h}/{\sqrt{m_n}}}(Z) - p_{\beta_0}(Z)}{p_{\beta_0}(Z)/\sqrt{m_n}}
\\
&&\quad\quad=
\mathbb E_{\beta_0} T_n(Z) \left(\frac{p_{\beta_0+{h}/{\sqrt{m_n}}}(Z) - p_{\beta_0}(Z)}{p_{\beta_0}(Z)/\sqrt{m_n}}
-
{s_{\beta_0}(Z)^T h}
\right) \\
&& \quad\quad\quad+\; 
\mathbb E_{\beta_0} T_n(Z) {s_{\beta_0}(Z)^T h}
\\
&&\quad\quad=
\mathbb E_{\beta_0} (T_n(Z)-g({\beta_0})) \times
\\
&& \quad\quad\quad   \left(\frac{p_{\beta_0+h/\sqrt{m_n}}(Z) - p_{\beta_0}(Z)}{p_{\beta_0}(Z)/\sqrt{m_n}}-
{s_{\beta_0}(Z)^T h}\right) \\
&&\quad\quad\quad+\;
\mathbb E_{\beta_0} T_n(Z) {s_{\beta_0}(Z)^T h},
\end{eqnarray*}
where we used that $\mathbb E_{\beta_0}s_{\beta_0}(Z)=0$, $\int p_{\beta_0}(z)dz =1$ and $\int p_{\beta_0+h/\sqrt{m_n}}(z)dz =1$.
Since the variance of $T_n$ is $\mathcal O(1/n)$ by the definition of strong asymptotic unbiasedness, then
\begin{eqnarray*}
\mathbb E_{\beta_0} (T_n(Z)-g({\beta_0}))^2 &= &\textrm{var}(T_n(Z)) +
\left[ \mathbb E_{\beta_0} (T_n(Z)-g({\beta_0}))\right]^2 \\
&=& \mathcal O(1/n)+o(1/n)=\mathcal O(1/n).
\end{eqnarray*}

\noindent
By Lemmas \ref{aux2} and \ref{a1} with $u:= h/\sqrt{m_n}$ and since $h^T \Sigma_0h=1$,
\begin{eqnarray*}
\mathbb E_{\beta_0}\left(\frac{p_{\beta_0+h/\sqrt{m_n}}(Z) - p_{\beta_0}(Z)}{p_{\beta_0}(Z)}-s_{\beta_0}(Z)^T h/\sqrt{m_n}\right)^2 
&=& o(n/m_nh^T\Sigma_0 h)
\\
&=&o(n/m_n).
\end{eqnarray*}
Hence, multiplying by $m_n$
$$\mathbb E_{\beta_0}\left(\frac{p_{\beta_0+h/\sqrt{m_n}}(Z) - p_{\beta_0}(Z)}{p_{\beta_0}(Z)/\sqrt{m_n}}-s_{\beta_0}(Z)^T h\right)^2 
= o(n).$$ 
Consequently, and by the Cauchy-Schwarz inequality, we have the upper bound
\begin{eqnarray*}
&& \left\lvert\mathbb E_{\beta_0} (T_n(Z)-g({\beta_0})) \left(\frac{p_{\beta_0+h/\sqrt{m_n}}(Z) - p_{\beta_0}(Z)}{p_{\beta_0}(Z)/\sqrt{m_n}}-s_{\beta_0}(Z)^T h\right)\right\lvert
\\
&\leq &
 \sqrt{\mathbb E_{\beta_0} (T_n(Z)-g({\beta_0}))^2} 
\sqrt{\mathbb E \left(\frac{p_{\beta_0+h/\sqrt{m_n}}(Z) - p_{\beta_0}(Z)}{p_{\beta_0}(Z)/\sqrt{m_n}}-s_{\beta_0}(Z)^T h\right)^2}
\\
&=&\mathcal O\left(1/\sqrt{n}\right) o(\sqrt{n})=o(1).
\end{eqnarray*}
Next observe that by the triangle inequality we have 
$$|h^T \dot g(\beta_0)| - |\textrm{cov}_{\beta_0}(T_n, \epsilon^T X h)|
\leq 
  |h^T \dot g(\beta_0)-\textrm{cov}_{\beta_0}(T_n, \epsilon^T X h)|=o(1 ),$$
	and hence and by the Cauchy-Schwarz inequality it follows
\begin{eqnarray*}
|h^T \dot g(\beta_0) |
 & \leq &
|\textrm{cov}_{\beta_0}(T_n, \epsilon^T X h)|
+
o(1)
\\
&\leq & \sqrt{\textrm{var}_{\beta_0}(T_n)} \sqrt{\textrm{var}_{\beta_0}(\epsilon^T Xh)} + o(1).
\\
&\leq &  \sqrt{\textrm{var}_{\beta_0}(T_n)} \sqrt{n}+ o(1),
\end{eqnarray*}
where we used that $\textrm{var}_{\beta_0}(X^T \epsilon)=n\Sigma_0$ and  $h^T \Sigma_0 h = 1$.
By the strong asymptotic unbiasedness assumption on $T_n$, we have $\textrm{var}_{\beta_0}(T_n)=\mathcal O(1/n)$ and thus taking squares of both sides of the last inequality we obtain
\begin{eqnarray*}
|h^T \dot g(\beta_0) |^2
\leq  {n} {\textrm{var}_{\beta_0}(T_n)} + 2\sqrt{n} \sqrt{\textrm{var}_{\beta_0}(T_n)}o(1) + o(1) \leq {n} {\textrm{var}_{\beta_0}(T_n)} + o(1).
\end{eqnarray*}

\end{proof}



\begin{proof}[Proof of Theorem \ref{bxi}]
To obtain the lower bound on the variance, we apply Theorem \ref{crlb.lr} with $h:=\Theta_0 \xi/\sqrt{\xi^T \Theta_0 \xi}$ (note that for $g(\beta)=\xi^T\beta$, the condition on $g$ is satisfied).  
Then by direct calculation, we see that the condition $h^T \Sigma_0h=1$ is satisfied for this choice of $h$. 
Moreover, since $\beta_0+h/\sqrt{n}\in B_{}(\beta_0,c/\sqrt{n})$ by assumption, then it implies that also 
$\beta_0+h/\sqrt{m_n}\in B_{}(\beta_0,c/\sqrt{m_n})$.
Thus
by Theorem \ref{crlb.lr} follows the lower bound
$\textrm{var}_{\beta_0}(T) \geq \frac{\xi^T \Theta_0 \xi +o(1)}{n},$ for any strongly asymptotically unbiased estimator $T$ of $\xi^T\beta_0$ at $\beta_0$ with the rate $m_n.$
\\
Next we turn to proving the upper bound. 
By assumptions of this theorem, the conditions of Lemma \ref{SAUbxi} are also satisfied and thus Lemma \ref{SAUbxi} implies that $\hat b_\xi$ is strongly asymptotically unbiased at $\beta_0$. 
It remains to calculate the variance of $\hat b_{\xi}.$
Consider the following decomposition 
$$\hat b_{\xi} - \xi^T \beta_0 =
 \xi^T \Theta_0 X^T\epsilon / n 
+ 
\xi^T(\hat\Theta-\Theta_0)^T X^T\epsilon / n 
+
\xi^T (\hat \Sigma\hat\Theta - I)^T (\hat \beta-\beta_0).
$$
Then one can show using 
the Cauchy-Schwarz inequality that
\begin{eqnarray}\label{bxi.deco}
 \textrm{var}_{\beta_0} (\hat b_\xi ) &=&
 \underbrace{\textrm{var}_{\beta_0} (\xi^T \Theta_0  X^T\epsilon / n)}_{r_1} 
+
\underbrace{\textrm{var}_{\beta_0} ( \xi^T(\hat\Theta-\Theta_0)^T X^T\epsilon / n ) }_{r_2}
\\\nonumber
&&+\;
\underbrace{\textrm{var}_{\beta_0} ( \xi^T(\hat \Sigma \hat\Theta - I)^T (\hat \beta-\beta_0) )}_{r_3} 
\\\nonumber
&& +\; \mathcal O(r_1^{1/2} r_2^{1/2}+ r_1^{1/2} r_3^{1/2}+ r_2^{1/2} r_3^{1/2}).
\end{eqnarray}
First note that as in the proof of Lemma \ref{SAUbxi}, we have $ \mathbb E \xi^T\Theta_0 X^T\epsilon / n = 0$ 
and hence
\begin{eqnarray*}
r_1=\textrm{var}_{\beta_0} (\xi^T \Theta_0 X^T\epsilon / n) 
&=& 
\mathbb E_{\beta_0} ( \mathbb E[(\xi^T \Theta_0 X^T\epsilon / n)^2|X] )
\\
&=&
\mathbb E_{\beta_0} (\Theta_0\xi)^T X^TX/n \Theta_0 \xi) = \xi^T \Theta_0\xi/n.
\end{eqnarray*}
For a random variable $U,$ we have $\textrm{var}(U) \leq \mathbb EU^2$ and hence by H\"older's inequality and the Cauchy-Schwarz inequality
\begin{eqnarray*}
r_2&=&
\textrm{var}_{\beta_0} ( \xi^T(\hat\Theta-\Theta_0)^T X^T\epsilon / n ) \\
&\leq & \mathbb E_{\beta_0}( \xi^T(\hat\Theta-\Theta_0)^T X^T\epsilon / n )^2 
\\
&\leq &
\|\xi\|_1^2 \mathbb E_{\beta_0} \vertiii{\hat\Theta-\Theta_0}_1^2 \|X^T\epsilon / n\|_\infty^2
\\
&\leq &
\|\xi\|_1^2 \left(\mathbb E_{\beta_0} \vertiii{\hat\Theta-\Theta_0}_1^4\right)^{1/2}\left(\mathbb E_{\beta_0} \|X^T\epsilon / n\|_\infty^4\right)^{1/2}
\\
&=&\mathcal O(\max_{j=1,\dots,p}\lambda_j ^2s_j^2\log p/n)=o(1/n),
\end{eqnarray*}
where we used the  result of Lemma \ref{the} and applied Lemma \ref{con}.
For the remainder $r_3$ we have
\begin{eqnarray*}
r_3 &\leq & 
 \mathbb E_{\beta_0} ( \xi^T(\hat \Sigma \hat\Theta - I)^T (\hat \beta-\beta_0) )^2 
\\
&\leq & \|\xi\|_1^2
\mathbb E _{\beta_0}\|\hat \Sigma \hat\Theta - I\|_\infty^2 \|\hat \beta-\beta_0\|_1^2
\\
&\leq & \|\xi\|_1^2
\left(\mathbb E_{\beta_0} \|\hat \Sigma \hat\Theta - I\|_\infty^4\right)^{1/2} \left(\mathbb E_{\beta_0}\|\hat \beta-\beta_0\|_1^4\right)^{1/2}
\\
&=&
\mathcal O(\max_{j=1,\dots,p}\lambda_j ^2 s^2\lambda^2)=o(1/n),
\end{eqnarray*}
where we used Lemma \ref{tau} and Theorem \ref{strong2}.
\\
Thus using the above calculations and using \eqref{bxi.deco} we conclude that
$$\textrm{var}_{\beta_0}(\hat b_\xi) = {\xi^T \Theta_{0}\xi/n + o(1/n)}.$$

\end{proof}

\subsection{Proofs for Section \ref{subsec:lr.fixed}: Main results for fixed design}
\label{sec:lr.fixedp}

\begin{proof}[Proof of Theorem \ref{fixed}]
The proof follows the same lines as the proof of Theorem \ref{crlb.lr}. The only difference is that we need to check the condition 
$$\mathbb E_{\beta_0}\left(\frac{p_{\beta_0+h/\sqrt{m_n}}(Z) - p_{\beta_0}(Z)}{p_{\beta_0}(Z)/\sqrt{m_n}}-s_{\beta_0}(Z)^T h\right)^2 = o(n),$$
for fixed design, where $p_{\beta_0},Z$ and $s_{\beta_0}$ are defined identically as in the proof of Theorem \ref{crlb.lr}. 
We denote $u:=h/\sqrt{m_n}.$ 
Analogously as in the proof of Lemma \ref{aux2} in Section \ref{sec:lr.randomp}, we obtain
\begin{eqnarray*}
\mathbb E_{\beta_0}\left(
e^{-\epsilon^T X u + \frac{1}{2}u^TX^TXu} -1
 - \epsilon^TX u 
\right)^2
&=&
e^{u^TX^TXu} - 1 - u^TX^TXu
\\
&=&
o(u^TX^TXu) =o(nu^T\hat\Sigma u)
.
\end{eqnarray*}
Then by the assumption $h^T \hat\Sigma h=\mathcal O(1)$ we obtain
$$o(nu^T\hat\Sigma u)  = o( nh^T\hat\Sigma h/m_n)=o( h^T \hat\Sigma h n/m_n) = o(n/m_n).$$
Hence plugging in $u=h/\sqrt{m_n}$ and multiplying by $m_n$ we obtain
$$\mathbb E_{\beta_0}\left(\frac{p_{\beta_0+h/\sqrt{m_n}}(Z) - p_{\beta_0}(Z)}{p_{\beta_0}(Z)/\sqrt{m_n}}-s_{\beta_0}(Z)^T h\right)^2 = o(n).$$
\end{proof}

\begin{proof}[Proof of Theorem \ref{LR2}]
The lower bound follows by Theorem \ref{fixed} (note that $g(\beta)=\beta_j$ and thus the condition on $g$ is satisfied) applied with
$h:=\hat\Theta_j/\sqrt{\hat\Theta_{jj}}$. We only need to check that 
$h^T\hat\Sigma h =\mathcal O(1)$.
By the assumption $\beta_0+\hat\Theta_j/\sqrt{\hat\Theta_{jj}n}\in\mathcal B(d_n)$ we obtain that $s_j:=\|\hat\Theta_j\|_0 \leq d_n$.
Then
\begin{eqnarray}
\nonumber
h^T \hat \Sigma h&=&
\hat\Theta_j^T \hat\Sigma \hat\Theta_j /\hat\Theta_{jj} 
\\\nonumber
&\leq & \|\hat\Theta_j^T\|_1 \| \hat\Sigma \hat\Theta_j - e_j\|_\infty  /\hat\Theta_{jj}+ 
\hat\Theta_j^Te_j/\hat\Theta_{jj}
\\\nonumber
&\leq &\mathcal O\left(\sqrt{s_j} \|\hat\Sigma \hat\Theta_j \hat\tau_j^2 - \hat\tau_j^2\|_\infty\right) + 1,
\\\label{var.fixed}
&\leq &\mathcal O\left(\sqrt{s_j} \lambda_j/\hat\tau_j^2\right) + 1 =o(1)+1,
\end{eqnarray}
where we used the KKT condition for $\hat\gamma_j$, $1/\hat\tau_j^2 = \hat\Theta_{jj} \leq \|\hat\Theta_j\|_2=\mathcal O(1)$
and $s_j \leq d_n=o(\sqrt{n}/\log p).$
This yields the lower bound $\hat\Theta_{jj}+o(1).$ 

Next we turn to proving the upper bound.
Strong asymptotic unbiasedness of $\hat b_j$ follows similarly as in Lemma \ref{SAUbxi} under the assumptions  $\beta_0\in\mathcal B(d_n)$, $d_n = o\left({\sqrt{n}}/{\log p} \right)$, if $\hat\Sigma$ satisfies the compatibility condition with a universal constant and $\beta_0+\hat\Theta_j/\sqrt{\hat\Theta_{jj}n}\in B(\beta_0, c/\sqrt{n})$.\\
For the variance of 
$$\hat b_j- \beta_j^0 =
 (\hat\Theta_j)^T X^T\epsilon / n 
+
(\hat \Sigma\hat\Theta_j - e_j)^T (\hat \beta-\beta_0),
$$
 we get (using that $|\hat\Theta_j^T \hat\Sigma \hat\Theta_j/\hat\Theta_{jj}-1|=o(1)$ as derived in \eqref{var.fixed})
\begin{eqnarray*}
\textrm{var}_{\beta_0}(\hat b_j) 
&=& 
\hat\Theta_j^T  \hat\Sigma \hat\Theta_j/n + 
\mathcal O( \mathbb E\|\hat\Sigma\hat\Theta_j-e_j\|_\infty^2 \|\hat\beta-\beta_0\|_1^2)
\\
&=& 
\hat\Theta_j^T  \hat\Sigma \hat\Theta_j/n + 
\mathcal O( \|\hat\Sigma\hat\Theta_j-e_j\|_\infty^2 \mathbb E_{\beta_0}\|\hat\beta-\beta_0\|_1^2)
\\
&=& \hat\Theta_j^T  \hat\Sigma \hat\Theta_j/n+ \mathcal O(\lambda_j^2/(\hat\tau_j^2)^2s^2\lambda^2)
\\
&=&\hat\Theta_{jj}/n + o(1/n),
\end{eqnarray*}
where we used $1/\hat\tau_j^2 = \hat\Theta_{jj} \leq \|\hat\Theta_j\|_2=\mathcal O(1).$
\end{proof}

\subsection{Proofs for Section \ref{subsec:ex.lr}: Le Cam's bounds}
\label{LRlecam}

\begin{proof}[Proof of Theorem \ref{lecam.bj}]
We apply Theorem \ref{lecam} from Section \ref{sec:lecam}.
In this setting, we have asymptotic linearity of the de-sparsified Lasso (see \cite{vdgeer13}) with
the influence function $l_{\beta_0}(X_i,Y_i)=(\Theta^0_j)^T X_i \epsilon_i$, where $\Theta_j^0$ is the $j$-th column of the precision matrix. 
We first show the bias condition \eqref{nob} is satisfied with the influence function $l_\beta$. 
By direct calculation, for any $h\in\mathbb R^p$ we have
\begin{eqnarray*}
P_{\beta_0}(l_{\beta_0} h^T s_{\beta_0})-h^T e_j 
&=&
(\Theta_j^0)^T\mathbb E X_1 \epsilon_1^2  X_1^T h -h_j 
\\
&=&
(\Theta_j^0)^T \mathbb EX_1X_1^T\mathbb E( \epsilon_1^2|X_1) h-h_j=0.
\end{eqnarray*}
Therefore in this case the bias condition holds.
Hence we can conclude by Theorem \ref{lecam} that 
the de-sparsified estimator 
$$\hat b_j = \hat \beta_j + \hat\Theta_j^T X^T(Y-X\hat\beta) /n $$
satisfies for every $\tilde\beta_n\in B_{}(\beta_0,\frac{c}{\sqrt{n}})$ 
$$\frac{
\sqrt{n}(\hat b_j - \tilde\beta_n) }{
 (\Theta_{jj}^0)^{1/2}
} \stackrel{{\tilde\beta_n}}{\rightsquigarrow} \mathcal N( 0, 1 ).$$ 
In addition, no asymptotically linear estimator satisfying the condition \eqref{nob} (here condition \eqref{lr.nobias}) can have smaller asymptotic variance than $\Theta_{jj}^0$ as follows by the lower bound on the asymptotic variance in Lemma \ref{lb2} in Section \ref{sec:lecam}, i.e.
$$
V_{\beta_0} \geq  \dot g({\beta_0})^T I_{\beta_0}^{-1}\dot g({\beta_0}) +o(1) = \Theta_{jj}^0+o(1).
$$ 

\end{proof}

\section{Proofs for Section \ref{subsec:ggm.all}: Gaussian graphical models}
\label{sec:ggm.lbp}

\subsection{Proofs for Section \ref{subsec:ggm.sau}: Strong asymptotic unbiasedness of the de-sparsified nodewise Lasso}
\label{sec:ggm.saup}

\begin{proof}[Proof of Lemma \ref{SAU}]
By the Karush-Kuhn-Tucker conditions cor\-res\-pon\-ding
 to the node\-wise Lasso estimator, we have $\|\hat\Sigma\hat\Theta_i-e_i\|_\infty = \mathcal O(\lambda_i/\hat\tau_i^2).$
Hence, and applying a version of Lemma \ref{the} without the maximum over $j=1,\dots,p$, we obtain
\begin{eqnarray*}
\mathbb E _{\Theta_0}(\hat T_{ij} - \Theta^0_{ij}) &=&
 \underbrace{\mathbb E_{\Theta_0} (\Theta^0_i)^T(\hat\Sigma - \Sigma_0)\Theta^0_j}_{=0} 
+
\mathbb E_{\Theta_0}(\hat\Theta_i-\Theta^0_i)^T (\hat\Sigma\Theta^0_j- e_j)
\\
&& +\;\; 
\mathbb E_{\Theta_0}(\hat \Sigma\hat\Theta_i - e_i)^T (\hat \Theta_j-\Theta^0_j)
\\
&\leq &
(\mathbb E_{\Theta_0} \| \hat\Theta_i-\Theta^0_i \|_1^2 )^{1/2} (\mathbb E_{\Theta_0}\| \hat\Sigma\Theta^0_j-e_j\|_\infty^2)^{1/2}
\\
&& +\;\;
(\mathbb E_{\Theta_0} \|\hat \Sigma\hat\Theta_i - e_i\|_\infty^2)^{1/2} (\mathbb E_{\Theta_0}\| \hat\Theta_j-\Theta^0_j\|_1^2)^{1/2} 
\\
& \leq &
 \mathcal O( s_i\lambda_i \lambda_j(\mathbb E_{\Theta_0}1/(\hat\tau_j^2)^2)^{1/2} )  +
\mathcal O( s_j\lambda_i\lambda_j(\mathbb E_{\Theta_0}1/(\hat\tau_i^2)^2)^{1/2})
\\
&= & o(1/\sqrt{n}). 
\end{eqnarray*}

\end{proof}

\subsection{Proofs for Section \ref{subsec:ggm.ub}: Main results}
\label{sec:ggm.mainp}

In this section we will give the proof of Theorems \ref{ggm.lc} and \ref{ggm.main}. In the proof of Theorem \ref{ggm.lc}, we use Lemma \ref{ggm.error.lc} which is stated and proved in Appendix \ref{sec:A}.

\begin{proof}[Proof of Theorem \ref{ggm.lc}]
The proof is similar to the proof of Theorem \ref{crlb.lr}.
By strong asymptotic unbiasedness of $T_n$ at $\Theta_0$ in the direction $H$, it follows 
\begin{eqnarray*}
\xi_1^TH\xi_2 
&=& 
\sqrt{m_n} \left(  \xi_1^T(\Theta_0+H/\sqrt{m_n})\xi_2 - \xi_1^T\Theta_0\xi_2 \right) 
\\
&=&
\sqrt{m_n} \left( \mathbb E_{\Theta_0 + H/\sqrt{m_n}} T_n(X)  - \mathbb E_{\Theta_0} T_n(X)\right) +o(1) 
\end{eqnarray*}
Let $s_{\Theta_0}(X) :=-n( \hat\Sigma-\Sigma_0)/2.$
Denoting the  probability density corresponding to $X$ by $p_\Theta$, we may further rewrite the expressions to obtain
\begin{eqnarray*}
&&\sqrt{m_n} \left( \mathbb E_{\Theta_0 + H/\sqrt{m_n}} T_n(X)  - \mathbb E_{\Theta_0} T_n(X)\right) 
\\
&&\quad\quad=
\sqrt{m_n}\int T_n(x) (p_{\Theta_0 + H/\sqrt{m_n}}(x) - p_{\Theta_0}(x))dz\\
& &\quad\quad=
\sqrt{m_n}\mathbb E_{\Theta_0} T_n(X) \frac{p_{\Theta_0+{H}/{\sqrt{m_n}}}(X) - p_{\Theta_0}(X)}{p_{\Theta_0}(X)}
\\
&&\quad\quad=
\mathbb E_{\Theta_0} T_n(X) \left(\frac{p_{\Theta_0+{H}/{\sqrt{m_n}}}(X) - p_{\Theta_0}(X)}{p_{\Theta_0}(X)/\sqrt{m_n}}
-
{\textrm{vec}(H)^T \textrm{vec}(s_{\Theta_0}(X))}
\right) \\
&& \quad\quad\quad+\; 
\mathbb E_{\Theta_0} T_n(X) {\textrm{vec}(H)^T \textrm{vec}(s_{\Theta_0}(X))}
\\
&&\quad\quad=
\mathbb E_{\Theta_0} (T_n(X)-\xi_1^T{\Theta_0}\xi_2) \times
\\
&& \quad\quad\quad  \times \left(\frac{p_{\Theta_0+H/\sqrt{m_n}}(X) - p_{\Theta_0}(X)}{p_{\Theta_0}(X)/\sqrt{m_n}}-
{\textrm{vec}(H)^T \textrm{vec}(s_{\Theta_0}(X))}\right) \\
&&\quad\quad\quad+\;
\mathbb E_{\Theta_0} T_n(X) {\textrm{vec}(H)^T \textrm{vec}(s_{\Theta_0}(X))},
\end{eqnarray*}
where in the last equality we used  that $\mathbb E_{\Theta_0}s_{\Theta_0}(X)=0$, $\int p_{\Theta_0}(x)dx =1$ and $\int p_{\Theta_0+H/\sqrt{m_n}}(x)dx=1$.
Since the variance of $T_n$ is $\mathcal O(1/n)$ by the definition of strong asymptotic unbiasedness, then
\begin{eqnarray*}
\mathbb E_{\Theta_0} (T_n(X)-\xi_1^T {\Theta_0}\xi_2))^2 &= &\textrm{var}_{\Theta_0}(T_n(X)) +
\left[ \mathbb E_{\Theta_0} (T_n(X)-\xi_1^T{\Theta_0}\xi_2)\right]^2 \\
&=& \mathcal O(1/n)+o(1/n)=\mathcal O(1/n).
\end{eqnarray*}
We need to use Lemma \ref{ggm.error.lc} in Appendix \ref{sec:A} to conclude that
the remainder is small.
Lemma \ref{ggm.error.lc} implies
\begin{eqnarray*}
\mathbb E \left( \frac{p_{\Theta_0 + H/\sqrt{m_n}}(X)}{p_{\Theta_0}(X)} - 1 - \textrm{vec}(H)^T \textrm{vec}(s_{\Theta_0}(X))/\sqrt{m_n} \right)^2
=o(\delta_n).
\end{eqnarray*}
Consequently, and by the Cauchy-Schwarz inequality, we have the upper bound
$$\left\lvert\mathbb E_{\Theta_0} (T_n(X)-\xi_1^T{\Theta_0}\xi_2) \left(\frac{p_{\Theta_0+H/\sqrt{m_n}}(X) - p_{\Theta_0}(X)}{p_{\Theta_0}(X)/\sqrt{m_n}}-
\textrm{vec}(H)^T \textrm{vec}(s_{\Theta_0}(X))\right)\right\lvert
$$
\begin{eqnarray*}
&&
\leq 
 \sqrt{\mathbb E_{\Theta_0} (T_n(X)-\xi_1^T{\Theta_0}\xi_2)^2} \times 
\\
&&
\times
\sqrt{\mathbb E_{\Theta_0} \left(\frac{p_{\Theta_0+H/\sqrt{m_n}}(X) - p_{\Theta_0}(X)}{p_{\Theta_0}(X)/\sqrt{m_n}}
-\textrm{vec}(H)^T \textrm{vec}(s_{\Theta_0}(X))\right)^2}
\\
&&
=\mathcal O\left(1/\sqrt{n}\right)o(\sqrt{n})=o(1).
\end{eqnarray*}
Thus we have
$$\xi_1^TH\xi_2 
 = \mathbb E_{\Theta_0} T_n(X) {\textrm{vec}(H)^T \textrm{vec}(s_{\Theta_0}(X))} + o(1).$$
Hence, because 
\begin{eqnarray*}
&&|\xi_1^TH \xi_2| - |\textrm{cov}(T_n, \textrm{vec}(H)^T \textrm{vec}(s_{\Theta_0}(X)))|
\\
&& \leq 
  |\xi_1^TH\xi_2-\textrm{cov}(T_n,\textrm{vec}(H)^T \textrm{vec}(s_{\Theta_0}(X)))|\\
	&&=o(1 )
	\end{eqnarray*}
	it follows using the Cauchy-Schwarz inequality that
\begin{eqnarray*}
|\xi_1^TH\xi_2|
 &=&
|\textrm{cov}_{\Theta_0}(T_n, \textrm{tr}(-n(\hat\Sigma-\Sigma_0)H/2))|
+
o(1)
\\
&\leq & \sqrt{\textrm{var}_{\Theta_0}(T_n)} \sqrt{\textrm{var}_{\Theta_0}(-n\textrm{tr}((\hat\Sigma-\Sigma_0)H/2))} + o(1).
\end{eqnarray*}
Now we have
\begin{eqnarray*}
\textrm{var}_{\Theta_0}(\textrm{tr}(n(\hat\Sigma-\Sigma_0)H/2))
&=&\textrm{var}_{\Theta_0}( n\textrm{tr}(\hat\Sigma H/2) )\\
&=&
\textrm{var}_{\Theta_0}(\xi_1^T \Theta_0 X^TX\Theta_0 \xi_2)/\sigma_{}^2 \\
&=& n\textrm{var}_{\Theta_0}(\xi_1^T \Theta_0 X^{(1)}(X^{(1)})^T\Theta_0 \xi_2)/\sigma_{}^2\\
&=&
n(\xi_1^T\Theta_0\xi_1\xi_2^T \Theta_0\xi_2 + (\xi_1^T \Theta_0\xi_2 )^2)/\sigma_{}^2\\
& =& n.
\end{eqnarray*}
Hence we conclude
\begin{eqnarray*}
|\xi_1^TH\xi_2|
\leq \sqrt{n} \sqrt{\textrm{var}_{\Theta_0}(T_n)} + o(1).
\end{eqnarray*}
Plugging in $H=\Theta_0(\xi_1\xi_2^T + \xi_2\xi_1^T)\Theta_0/\sigma_{} $ we obtain 
$$
|\xi_1^TH\xi_2|^2=
(\xi_1^T \Theta_0 \xi_1\xi_2^T \Theta_0\xi_2  + (\xi_1^T \Theta_0 \xi_2)^2)^2/\sigma_{}^2 = \sigma_{}^2.
$$
By the strong asymptotic unbiasedness assumption on $T_n$, we have $\textrm{var}_{\Theta_0}(T_n)$ $=\mathcal O(1/n)$ and thus taking squares of both sides of the last inequality we obtain
\begin{eqnarray*}
&& \xi_1^T \Theta_0 \xi_1\xi_2^T \Theta_0\xi_2  + (\xi_1^T \Theta_0 \xi_2)^2\\
&\leq &  {n} {\textrm{var}_{\Theta_0}(T_n)} + 2\sqrt{n} \sqrt{\textrm{var}_{\Theta_0}(T_n)}o(1) + o(1)\\
& \leq & {n} {\textrm{var}_{\Theta_0}(T_n)} + o(1).
\end{eqnarray*}

\end{proof}


\begin{proof}[Proof of Theorem \ref{ggm.main}]
To obtain the lower bound on the variance, we apply Theorem \ref{ggm.lc} with $H:=(\Theta^0_i(\Theta^0_j)^T + \Theta^0_j(\Theta^0_i)^T)/\sigma_{}$.  
Then by direct calculation, see that the condition $\textrm{tr}(H^T \Sigma_0H)=\mathcal O(1)$ is satisfied for this choice of $H$. 
Moreover, $\Theta_0+H/\sqrt{n}\in B_{}(\Theta_0,c/\sqrt{n})$ by assumption, but then also
$\Theta_0+H/\sqrt{m_n}\in B_{}(\Theta_0,c/\sqrt{m_n})$.
Hence by Theorem \ref{ggm.lc} follows the lower bound
$\textrm{var}_{\Theta_0}(T) \geq \frac{\xi^T \Theta_0 \xi +o(1)}{n}.$
\\
Next $\hat T_{ij}$ is strongly asymptotically unbiased at $\Theta_0$ in every direction $H$ such that 
$\Theta_0+H/\sqrt{n}\in\mathcal G(d_1,\dots,d_p)$, which follows by Lemma \ref{SAU}.
\\
It remains to calculate the variance of $\hat T_{ij}.$
First we have that
$$\textrm{var}_{\Theta_0}((\Theta^0_i)^T (\hat\Sigma -\Sigma_0)\Theta^0_j)= 
\frac{1}{n} \textrm{var}_{\Theta_0}((\Theta^0_i)^T X_1X_1^T \Theta^0_j)=(\Theta^0_{ii}\Theta^0_{jj}+(\Theta_{ij}^0)^2 )/ n.
$$
By basic calculations, it follows that
\begin{eqnarray*}
\textrm{var}_{\Theta_0}(\hat T_{ij}) 
&=&
(\Theta^0_{ii}\Theta^0_{jj}+(\Theta_{ij}^0)^2 )/ n + \mathcal O(\mathbb E_{\Theta_0}((\hat\Sigma\Theta^0_i -e_i)^T(\hat\Theta_j-\Theta^0_j))^2)\\
&&+\;\;
\mathcal O(\mathbb E_{\Theta_0}((\hat\Theta_i-\Theta^0_i)^T (\hat\Sigma \hat\Theta_j- e_j))^2)
\\
&= &
(\Theta^0_{ii}\Theta^0_{jj}+(\Theta_{ij}^0)^2 )/ n + \mathcal O(\mathbb E_{\Theta_0}\|\hat\Sigma\Theta^0_i -e_i\|_\infty ^2
\|\hat\Theta_j-\Theta^0_j\|_1^2)\\
&&+\;\;
\mathcal O(\mathbb E_{\Theta_0}\|\hat\Theta_i-\Theta^0_i\|_1^2 \|\hat\Sigma \hat\Theta_j- e_j\|_1^2).
\\
&= &
(\Theta^0_{ii}\Theta^0_{jj}+(\Theta_{ij}^0)^2 )/ n 
\\
&&+ \mathcal O( (\mathbb E_{\Theta_0}\|\hat\Sigma\Theta^0_i -e_i\|_\infty ^4)^{1/2}
(\mathbb E_{\Theta_0}\|\hat\Theta_j-\Theta^0_j\|_1^4)^{1/2})\\
&&+\;\;
\mathcal O((\mathbb E_{\Theta_0}\|\hat\Theta_i-\Theta^0_i\|_1^4)^{1/2}(\mathbb E_{\Theta_0} \|\hat\Sigma \hat\Theta_j- e_j\|_1^4)^{1/2})
\\
&=&(\Theta^0_{ii}\Theta^0_{jj}+(\Theta_{ij}^0)^2 )/ n+ \mathcal O(\lambda_i^2 s_j^2\lambda_j^2(\mathbb E_{\Theta_0}1/(\hat\tau_i^2)^2)^{1/2}) 
\\&& +\;\;
\mathcal O(s_i^2\lambda_i^2 \lambda_j^2(\mathbb E_{\Theta_0}1/(\hat\tau_j^2)^2)^{1/2} \\
&=&
(\Theta^0_{ii}\Theta^0_{jj}+(\Theta_{ij}^0)^2 )/ n+ o(1/n)
.
\end{eqnarray*}

\end{proof}

\section{Proofs for Section \ref{sec:lecam}: Le Cam's bounds for general models}
\label{sec:lcp}
In this section we give the proof of Theorem \ref{lecam}, for which we need Lemma \ref{portm} below.
Some technical results (contained in Lemmas \ref{dens}  and \ref{remainder}) are stated and proved in Appendix \ref{sec:B}.

\begin{lemma}\label{portm}
Assume the conditions of Theorem \ref{lecam}. 
Suppose that  $Z_n\rightsquigarrow Z$, where $Z$ is a random vector with values in $\mathbb R^2$. 
Let $X_n=\psi(Z_n)$ and $U_n = \psi(Z)$, where 
\[
x \mapsto \psi(x_1,x_2)=
\left(
\begin{array}{cc}
1/\sqrt{v_{11}} & 0\\
0 & 1
\end{array}
\right) \left[
V^{1/2}
\left(
\begin{array}{c}
x_1\\
x_2
\end{array}
\right)
 + 
\left(
\begin{array}{c}
-v_{12}\\
{-v_{22}/2}
\end{array}
\right)
\right].
\]
Then the following statements hold.
\begin{enumerate}
\item
For any function $f:\mathbb R^2 \rightarrow \mathbb R$ which is bounded and continuous it holds that 
$$\lim_{n\rightarrow \infty} \mathbb E f(X_n) - \mathbb E f(U_n)=0.$$
\item
Let $f$ be any bounded and continuous function $f:\mathbb R \rightarrow \mathbb R$.
Suppose that 
$$\lim_{M\rightarrow \infty}\lim_{n\rightarrow \infty} \mathbb E \min(0,M-e^{U_{n,2}})=0.$$
Then
it holds that
$$\lim_{n\rightarrow \infty}\mathbb E f(X_{n,1})e^{X_{n,2}} - \mathbb Ef(U_{n,1})e^{U_{n,2}} = 0.$$
\end{enumerate}
\end{lemma}

\begin{proof}[Proof of Lemma \ref{portm}]
We first prove the first statement. Let $\epsilon>0$ and let $f:\mathbb R^2 \rightarrow \mathbb R$ be continuous and bounded.\\
The map $\psi$ is linear, i.e. $\psi(x)=Ax+b$ for some $A\in\mathbb R^{2\times 2}$ and $b\in\mathbb R^2$ ($A,b$ depending on $n$).
Denote 
\[
D:=
\left(
\begin{array}{cc}
1/\sqrt{v_{11}} & 0\\
0 & 1
\end{array}
\right)
\] 
Observe that for any $x\in\mathbb R^2$
$$\|Ax\|_2^2 = x^T A^T Ax = x^T DVDx \leq \Lambda_{\max}(DVD) 
x^Tx
.$$
By Lemma \ref{eigen} we have that $\Lambda_{\max}(DVD) = \mathcal O(1)$
and $\|b\|_{2} =\mathcal O(1).$
Therefore, when $\|x\|_2 = \mathcal O(1)$, then 
\begin{equation}\label{bounded}
\|Ax+b\|_2=\mathcal O(1).
\end{equation}
Take a compact rectangle $R\subset \mathbb R^2$ not depending  on $n$ and such that $P(Z\not \in R) <\epsilon.$
\\
Divide the rectangle $R$ into a finite number of non-overlapping rectangles of diameter at most $\delta/L^{1/2},$ where 
$L$ is a universal constant such that $L \geq\Lambda_{\max}(DVD)$. 
By construction, the number of these rectangles, denote it $N$, does not depend on $n$.
So we have $R = \cup_{j=1}^N R_j$, where each $R_j$ is a rectangle of diameter at most $\delta/L^{1/2}.$ 
\\
For all $x,y\in R_j$ it holds that $\|x-y\|_2 \leq \delta/L^{1/2}$ and thus
\begin{equation}\label{e1}
\|\psi(x)-\psi(y)\|_2=\|A(x-y)\|_2 \leq L^{1/2} \|x-y\|_2 \leq \delta.
\end{equation}
Note that by \eqref{bounded}, there exists a compact set $S$ not depending on $n$ such that $\psi(R)\subset S$ for all $n.$
The continuous function $f$ is uniformly continuous on the compact set $S$.
Hence for the $\epsilon$ there exists a $\delta>0$ such that for all $z,v\in S $ it holds that if $\|z-v\|_2<\delta$ then $|f(z)-f(v)|<\epsilon.$
But then since for all $x,y\in R_j$ we have that $\psi(x),\psi(y)\in S$, we obtain by \eqref{e1} and the absolute continuity of $f$ that
$$|f(\psi (x))-f( \psi (y))|<\epsilon$$
for all $n.$
Take a point $x_j$  from each set $R_j$ and define $f_{\epsilon}  =\sum_{j=1}^N f(\psi(x_j)) \mathbf 1_{R_j}.$
Then 
$|f( \psi(x))-f_{\epsilon}(x)|<\epsilon$ for all $x\in R$ (and all $n$) and hence if $f$ takes values in $[-K,K],$ we have the following upper bounds
\begin{equation}\label{1}
|\mathbb Ef( \psi(Z))-\mathbb Ef_{\epsilon}(Z)| \leq \epsilon + 2KP(Z\not\in R),
\end{equation}
\begin{equation}\label{2}
|\mathbb Ef( \psi(Z_n))-\mathbb Ef_{\epsilon}(Z_n)| \leq \epsilon + 2KP(Z_n\not\in R),
\end{equation}
\begin{equation}\label{3}
|\mathbb Ef_{\epsilon}(Z_n)-\mathbb Ef_{\epsilon}(Z)| \leq \sum_{j=1}^N |P(Z_n\in R_j) - P(Z\in R_j)| |f(\psi(x_j))|.
\end{equation}
Since $Z_n\rightsquigarrow Z$, for all $j=1,\dots,N$ it holds 
$$|P(Z_n\in R_j) - P(Z\in R_j) |  \rightarrow 0.$$
Similarly, 
$$|P(Z\not \in R) - P(Z_n\not\in R)| = |P(Z \in R) - P(Z_n\in R)| \rightarrow 0.$$
Finally, by construction we have $P(Z\not \in R) <\epsilon$.
We thus conclude that the upper bounds \eqref{1}, \eqref{2} and \eqref{3} can be made smaller than $C\epsilon$ for $n$ sufficiently large. 
The claim follows by combining the three upper bounds. 
\\\\
Next we prove the second statement.
Denote $g(x_1,x_2)=f(x_1)e^{x_2}.$
We write  $g=g^+ - g^-$, where $g^+ = \max\{g,0\}$ is the positive part and $g^{-}:=\max\{-g,0\}$ is the negative part.
We first prove for the positive part $g^+$ that
\begin{equation}\label{fpos}
\lim_{n\rightarrow \infty}\mathbb E g^+(X_{n}) - \mathbb Eg^+(U_{n}) = 0.
\end{equation}
This will be achieved by first showing that 
$$\liminf_{n\rightarrow\infty}\mathbb E f^+(X_n)e^{X_{n,2}} - \mathbb E f^+(U_n)e^{U_{n,2}} 
\geq 
0$$
and secondly showing that 
$$\limsup_{n\rightarrow\infty}\mathbb E f^+(X_n)e^{X_{n,2}} - \mathbb E f^+(U_n)e^{U_{n,2}} 
\leq 
0.$$
Then combining the two gives \eqref{fpos}.
\\
First we prove that $\liminf_{n\rightarrow\infty}\mathbb E f^+(X_n)e^{X_{n,2}} - \mathbb E f^+(U_n)e^{U_{n,2}} 
\geq 
0$.
For every $M$,  since $g^+$ is non-negative, it holds that
$f^+(x)e^{x_2} \geq
f^+(x) (e^{x_2}\wedge M)$, where the symbol  $\wedge$ denotes the minimum. Hence
\begin{eqnarray*}
&&
\mathbb E f^+(X_{n})e^{X_{n,2}} - \mathbb Ef^+(U_{n})e^{U_{n,2}} 
\\
&& \geq 
\mathbb E f^+(X_{n})(e^{X_{n,2}} \wedge M) - \mathbb Ef^+(U_{n})e^{U_{n,2}} 
\\
&&= 
[\mathbb E f^+(X_{n})(e^{X_{n,2}}\wedge M) - \mathbb Ef^+(U_{n})(e^{U_{n,2}}\wedge M)]
\\&&
+\;
[\mathbb E f^+(U_{n})(e^{U_{n,2}}\wedge M )- \mathbb Ef^+(U_{n})e^{U_{n,2}}]
\end{eqnarray*}
We have $f^+(x)(e^{x_2} \wedge M) -f^+(x)e^{x_2} = f^+(x)\min (0,M-e^{x_2})$.
Taking limes inferior of both sides, it follows that
\begin{eqnarray*}
&&
\liminf_{n\rightarrow\infty}\mathbb E f^+(X_n)e^{X_{n,2}} -\mathbb Ef^+(U_n)e^{U_{n,2}}
\\
&&\geq \liminf_{n\rightarrow\infty}
[\mathbb E f^+(X_n)(e^{X_{n,2}}\wedge M) - \mathbb Ef^+(U_n)(e^{U_{n,2}}\wedge M)]
\\
&& +\;
\liminf_{n\rightarrow\infty} -
\mathbb Ef^+(U_n) \min (0,M-e^{U_{n,2}}).
\end{eqnarray*}
For every fixed $M$, the function $x\mapsto f^+(x)(e^{x_2}\wedge M)$ is bounded and continuous.
We may thus apply the first result of the lemma to conclude
$$\liminf_{n\rightarrow \infty}\mathbb E f^+(X_n)(e^{X_{n,2}}\wedge M) - \mathbb Ef ^+(U_n)(e^{U_{n,2}}\wedge M) = 0.$$ 
Therefore, we have
\begin{eqnarray}\label{sum3}
&&
\liminf_{n\rightarrow\infty}\mathbb E f^+(X_n)e^{X_{n,2}} - \mathbb E f^+(U_n)e^{U_{n,2}} 
\\\nonumber
&\geq & 
\liminf_{n\rightarrow\infty} -\mathbb E f^+(U_n)\min (0,M-e^{U_{n,2}}).
\end{eqnarray}
Next since $|f^+|\leq K$ we have 
$$|-\mathbb E f^+(U_n)\min (0,M-e^{U_{n,2}})| \leq K \mathbb E \min (0,M-e^{U_{n,2}}).$$
Then the assumption 
$$\lim_{m\rightarrow \infty}\lim_{n\rightarrow \infty} \mathbb E \min (0,M-e^{U_{n,2}})=0$$ 
implies that also
\begin{eqnarray*}
&&
\lim_{m\rightarrow \infty}\liminf_n -\mathbb E f^+(U_{n})\min (0,M-e^{U_{n,2}})
\\&&=
-\lim_{m\rightarrow \infty}\limsup_n \mathbb E f^+(U_{n})\min (0,M-e^{U_{n,2}}) =
0,
\end{eqnarray*}
so we conclude that
 \begin{eqnarray}
\liminf_{n\rightarrow\infty}\mathbb E f^+(X_n)e^{X_{n,2}} - \mathbb E f^+(U_n)e^{U_{n,2}} 
\geq 
0.
\end{eqnarray}
\\
Now we prove 
$$\limsup_{n\rightarrow\infty}\mathbb E f^+(X_n)e^{X_{n,2}} - \mathbb E f^+(U_n)e^{U_{n,2}} 
\leq 
0.$$
Similarly as before,  since $K - f^+ \geq 0 $ ($K$ is an upper bound on $f$), we have that
\begin{eqnarray}
\nonumber
&& \liminf_{n\rightarrow\infty} \mathbb E (K -f^+(X_{n,1}))e^{X_{n,2}} - \mathbb E(K -  f^+(U_{n,1}))e^{U_{n,2}} 
\\\nonumber
&&\geq \;\; \liminf_{n\rightarrow\infty}
\mathbb E (K -f^+(X_{n,1}))(e^{X_{n,2}}\wedge M)-\mathbb E (K -f^+(U_{n,1}))(e^{U_{n,2}}\wedge M)
\\\label{secon}
&& \;\;\;+\;
\liminf_{n\rightarrow\infty} 
\mathbb E (K -f^+(U_{n,1}))(e^{U_{n,2}}\wedge M) -\mathbb E (K -f^+(U_{n,1}))e^{U_{n,2}}.
\end{eqnarray}
By the the first part of the lemma, we have that for every fixed $M$ it holds 
$$\liminf_n \mathbb E (K -f^+(X_{n,1}))(e^{X_{n,2}} \wedge M) -\mathbb E (K -f^+(U_{n,1}))(e^{U_{n,2}}\wedge M)=0,
$$
since the function $(x_1,x_2)\mapsto (K -f^+(x_{1}))(e^{x_{2}} \wedge M)$ is bounded and continuous.\\
For the term in \eqref{secon}, we have since $|K-f^+|\leq 2K$ 
$$
|-\mathbb E (K -f^+(U_{n,1}))\min(0,M-e^{U_{n,2}}) |\leq
2K\mathbb E e^{U_{n,2}}\min(0,M-e^{U_{n,2}}).
$$
Hence by the assumption
$\lim_{m\rightarrow \infty}\lim_{n\rightarrow\infty} 
\mathbb E \min(0,M-e^{U_{n,2}}) =0,$
we have that
 $$\liminf_{n\rightarrow\infty} -
\mathbb E (K -f^+(U_{n,1}))\min(0,M-e^{U_{n,2}})
=0.$$
Thus we conclude that
\begin{eqnarray*}
\liminf_{n\rightarrow\infty} \mathbb E (K -f^+(X_{n,1}))e^{X_{n,2}} - \mathbb E(K -  f^+(U_{n,1}))e^{U_{n,2}} 
\geq 0.
\end{eqnarray*}
Now note that
\begin{eqnarray*}
&&\liminf_{n\rightarrow\infty} \mathbb E (K -f^+(X_{n,1}))e^{X_{n,2}} - \mathbb E(K -  f^+(U_{n,1}))e^{U_{n,2}} 
\\
&& =\;\;
\liminf_n  - \mathbb E f^+(X_{n,1})e^{X_{n,2}} + \mathbb Ef^+(U_{n,1})e^{U_{n,2}} \\
&& =\;\;
-\limsup_n \mathbb E f^+(X_{n,1})e^{X_{n,2}} - \mathbb Ef^+(U_{n,1})e^{U_{n,2}}.
\end{eqnarray*}
So in conclusion we have shown that
\begin{eqnarray*}
&&\limsup_n \mathbb E f^+(X_{n,1})e^{X_{n,2}} - \mathbb Ef^+(U_{n,1})e^{U_{n,2}}
 \\
&& \quad\quad\leq 0 \leq \liminf_n \mathbb E f^+(X_{n,1})e^{X_{n,2}} - \mathbb Ef^+(U_{n,1})e^{U_{n,2}}.
\end{eqnarray*}
This proves \eqref{fpos}.
\\
The same procedure can be used for the negative part $f^-$ (since $f^{-}$ is also bounded and positive)
to show that
\begin{equation*}
\lim_{n\rightarrow \infty}\mathbb E f^-(X_{n,1})e^{X_{n,2}} - \mathbb Ef^-(U_{n,1})e^{U_{n,2}} = 0.
\end{equation*}
We then conclude that
\begin{eqnarray*}
&& \lim_{n\rightarrow \infty}\mathbb E f(X_{n,1})e^{X_{n,2}} - \mathbb Ef(U_{n,1})e^{U_{n,2}} \\
&&\;\;\leq 
\lim_{n\rightarrow \infty} |\mathbb E f^+(X_{n,1})e^{X_{n,2}} - \mathbb Ef^+(U_{n,1})e^{U_{n,2}}|\\
&&\;\;
\;\;\;\;+\lim_{n\rightarrow \infty} |\mathbb E f^-(X_{n,1})e^{X_{n,2}} - \mathbb Ef^-(U_{n,1})e^{U_{n,2}}|\\
&&\;\;=
 0.
\end{eqnarray*}

\end{proof}

\begin{proof}[Proof of Theorem \ref{lecam}]
We denote $\ell_\beta(x):= \log p_{\beta}(x)$ and 
$$\Lambda_n:= \sum_{i=1}^n \ell_{{\beta_{n,0}}+h/\sqrt{n}}(X^{(i)}) - \ell_{{\beta_{n,0}}}(X^{(i)}).$$
Further for  $0 \leq t\leq 1$ denote $g(t):=\ell((1-t)\beta + t\beta_0).$
Then by a two-term Taylor expansion of $g$ we have
$$g(1) - g(0) = \dot g(0) + \frac{1}{2}\ddot g(\tilde t),$$
where $0\leq \tilde t\leq 1.$
Rewriting this from the definition of $g$ gives
$$\ell_{\beta} - \ell_{\beta_0} = \ell_{\beta_0}(\beta-\beta_0) + \frac{1}{2} (\beta-\beta_0)^T \ddot \ell_{\tilde\beta}(\beta-\beta_0),$$
where $\tilde\beta := (1-\tilde t )\beta + \tilde t\beta_0.$
Applying the above with $\tilde \beta:={\beta_{n,0}}+h/\sqrt{n}$ and $\beta_0:={\beta_{n,0}}$ we obtain
$$\Lambda_n:= \sum_{i=1}^n \frac{1}{\sqrt{n}}\dot\ell_{{\beta_{n,0}}}(X^{(i)})^Th+
 \frac{1}{2} h^T \frac{1}{n}\sum_{i=1}^n\ddot \ell_{\tilde\beta}(X^{(i)})h,$$
where $\tilde\beta := (1-\tilde t )({\beta_{n,0}}+h/\sqrt{n}) + \tilde t{\beta_{n,0}} = {\beta_{n,0}} + (1-\tilde t)h/\sqrt{n}.$
We can then write the decomposition of $\Lambda_n$ as follows
\begin{eqnarray}\nonumber
\Lambda_n 
&=& \frac{1}{\sqrt{n}}\sum_{i=1}^n h^T\dot\ell_{{\beta_{n,0}}}(X^{(i)})
+
\frac{1}{2} h^T \frac{1}{n}\sum_{i=1}^n \ddot\ell_{\tilde\beta}(X^{(i)}) h
\\\nonumber
&=& \frac{1}{\sqrt{n}}\sum_{i=1}^n h^T\dot\ell_{{\beta_{n,0}}}(X^{(i)}) 
+
\frac{1}{2} \underbrace{h^T \frac{1}{n}\sum_{i=1}^n (\ddot\ell_{\tilde\beta}(X^{(i)})-\ddot\ell_{{\beta_{n,0}}}(X^{(i)})) h}_{rem_1}
\\
&& +\;\;
\frac{1}{2} \underbrace{h^T \left(\frac{1}{n}\sum_{i=1}^n \ddot \ell_{{\beta_{n,0}}}(X^{(i)}) + I({{\beta_{n,0}}})\right) h}_{rem_2}
\\\nonumber
&&\;\;-\frac{1}{2} h^T I({\beta_{n,0}}) h
.
\end{eqnarray}
We will now show that the remainders $rem_1,rem_2$ converge in probability to zero. First observe that since ${\beta_{n,0}}+h/\sqrt{n}\in B_{}({\beta_{n,0}}, \frac{c}{\sqrt{n}}),$
this implies that $\|h\|_2 \leq c$ and $\|h\|_0 \leq 2s.$ Combining these two properties yields $\|h\|_1 \leq \sqrt{2s} c = \mathcal O(\sqrt{s}).$\\
Regarding the first remainder we have for each $j,k=1,\dots,p$
$$\frac{1}{n}\sum_{i=1}^n(\ddot\ell_{\tilde\beta}(X^{(i)})-\ddot\ell_{{\beta_{n,0}}}(X^{(i)}) )_{j,k}= 
\frac{1}{n}\sum_{i=1}^n(\dddot \ell_{\bar\beta}(X^{(i)})^T(\tilde\beta - {\beta_{n,0}}))_{j,k},$$
where $\bar\beta = (1-t)\tilde\beta + t{\beta_{n,0}}.$ By assumption $\|\dddot \ell_{\bar\beta}\|_\infty \leq L$ and hence using H\"older's inequality
\begin{eqnarray*}
\|\frac{1}{n}\sum_{i=1}^n\ddot\ell_{\tilde\beta}(X^{(i)})-\ddot\ell_{{\beta_{n,0}}}(X^{(i)}) \|_\infty &\leq & 
\| \frac{1}{n}\sum_{i=1}^n\dddot \ell_{\bar\beta}(X^{(i)})\|_\infty\|\tilde\beta - {\beta_{n,0}}\|_1
\\
&\leq & L\|(1-\tilde t) h/\sqrt{n}\|_1 =\mathcal O(\sqrt{s/n}).
\end{eqnarray*}
Thus using H\"older's inequality, for the first remainder we get
$$|h^T \frac{1}{n}\sum_{i=1}^n (\ddot\ell_{\tilde\beta}(X^{(i)})-\ddot\ell_{{\beta_{n,0}}}(X^{(i)})) h|\leq
\mathcal O(s^{3/2}/\sqrt{n})=o(1).
$$
For the second remainder we have using assumption $\|\frac{1}{n}\sum_{i=1}^n \ddot \ell_{{\beta_{n,0}}}(X^{(i)}) + I({{\beta_{n,0}}})\|_\infty = \mathcal O_P(\lambda)$ that
\begin{eqnarray*}
\|h^T (\frac{1}{n}\sum_{i=1}^n \dot s_{{\beta_{n,0}}}(X^{(i)}) + I({\beta_{n,0}}))h\|_\infty 
&\leq &
 \|h\|_1^2 \|\frac{1}{n}\sum_{i=1}^n \dot s_{{\beta_{n,0}}}(X^{(i)}) + I_{{\beta_{n,0}}}\|_\infty
\\
& =&\mathcal O_P( s\lambda ) = o_P(1).
\end{eqnarray*}
Hence collecting the results, from the decomposition of $\Lambda_n$ we have
\begin{eqnarray}
\label{lik.exp}
\Lambda_n 
&=& \frac{1}{\sqrt{n}}\sum_{i=1}^n h^T\dot \ell_{{\beta_{n,0}}}(X^{(i)}) 
-
\frac{1}{2} h^T I({\beta_{n,0}}) h + o_P(1)
.
\end{eqnarray}

\noindent
We introduce the following notation. Let
\[
V:=
\left(
\begin{array}{cc}
V_{{\beta_{n,0}}} & P_{{\beta_{n,0}}}(l_{{\beta_{n,0}}} h^T s_{{\beta_{n,0}}})\\
P_{{\beta_{n,0}}}(l_{{\beta_{n,0}}} h^T s_{{\beta_{n,0}}}) & h^TI({{\beta_{n,0}}})h
\end{array}
\right)
\]
Furthermore, we denote the entries of the matrix $V$ by $v_{ij},i,j=1,2.$\\
The condition \ref{lindeberg} implies that for any fixed $a\in\mathbb R^2$ it holds for all $\epsilon>0$ 
\begin{eqnarray*}\label{alindeberg}
&&
\lim_{n\rightarrow \infty}\mathbb E(a^T V^{-1/2}(l_{{\beta_{n,0}}}(X^{(i)}), h^Ts_{{\beta_{n,0}}}(X^{(i)}))^T)^2
\times
\\
&&
\times \mathbf  1_{|a^T V^{-1/2}(l_{{\beta_{n,0}}}(X^{(i)}), h^Ts_{{\beta_{n,0}}}(X^{(i)}))^T|> \epsilon \sqrt{na^T a}} =0,
\end{eqnarray*}
where we used that $\|V\|_\infty =\mathcal O(1)$, which follows by Lemma \ref{eigen}.
By the Lindeberg's central limit theorem we thus have by condition \eqref{alindeberg} for any $a\in\mathbb R^2$ that
$$\frac{1}{\sqrt{na^T a}}\sum_{i=1}^n a^T V^{-1/2}(l_{{\beta_{n,0}}}(X^{(i)}), h^Ts_{{\beta_{n,0}}}(X^{(i)}))^T\rightsquigarrow \mathcal N(0,1).$$
Hence, using the likelihood expansion \eqref{lik.exp} and by the asymptotic linearity \eqref{al}, we conclude that
$$ a^T V^{-1/2}(\sqrt{n}(T_n-g({{\beta_{n,0}}})), \Lambda_n +  \frac{1}{2}h^TI({{\beta_{n,0}}})h) \rightsquigarrow a^T Z,\textrm{ where }Z\sim\mathcal N(0,I_2).$$
Then by the  Wold device we have
\[
Z_n
:=V^{-1/2} 
\left(
\begin{array}{c}
\sqrt{n}(T_n-g({{\beta_{n,0}}})) \\
 \Lambda_n + \frac{1}{2}h^TI({{\beta_{n,0}}})h
\end{array}
\right)\stackrel{{{\beta_{n,0}}}}{\rightsquigarrow} \mathcal N_2(0,I)\sim Z.
\]
Now let $f:\mathbb R\rightarrow \mathbb R$ be bounded and continuous.
We may rewrite
\begin{eqnarray*}
&&
\mathbb E_{{{\beta_{n,0}}}+h/\sqrt{n}} f\left(\frac{\sqrt{n}(T_n-g({{\beta_{n,0}}})) - v_{12}}{\sqrt{v_{11}}}\right) 
\\
&&= 
\mathbb E_{{\beta_{n,0}}} f\left(\frac{\sqrt{n}(T_n-g({{\beta_{n,0}}})) - v_{12}}{\sqrt{v_{11}}}\right)  e^{\Lambda_n}
\\
&&=
\mathbb E_{{\beta_{n,0}}} f(X_{n,1}) e^{X_{n,2}},
\end{eqnarray*}
where 
$X_n:= (X_{n,1},X_{n,2})= \psi(Z_n) 
$ for the function $\psi$ given by
\[
\psi(x_1,x_2)=\left(
\begin{array}{cc}
1/\sqrt{v_{11}} & 0\\
0 & 1
\end{array}
\right) \left[
V^{1/2}(x_1,x_2)^T + 
\left(
\begin{array}{c}
-v_{12}\\
{-v_{22}/2}
\end{array}
\right)
\right].
\]
Similarly, define $U_n$ as follows
\begin{eqnarray*}
U_n&:=&(U_{n,1},U_{n,2}) =\psi(Z)
\\
&\sim&
\left(
\begin{array}{cc}
1/\sqrt{v_{11}} & 0\\
0 & 1
\end{array}
\right) \left[
V^{1/2} \mathcal N(0,I_2) + 
\left(
\begin{array}{c}
-v_{12}\\
{-v_{22}/2}
\end{array}
\right)
\right]
\\
&=&
\mathcal N\left(
\left(
\begin{array}{c} 
-\frac{v_{12}}{\sqrt{v_{11}}}\\
-\frac{v_{22}}{2}
\end{array}
\right),
\left(
\begin{array}{cc}
1 & \frac{v_{12}}{\sqrt{v_{11}}}\\
\frac{v_{12}}{\sqrt{v_{11}}} & v_{22}
\end{array}
\right) \right).
\end{eqnarray*}
Since we know that $Z_n\rightsquigarrow Z$, we hope that in some sense $X_n=\psi(Z_n)$ is close to $U_n=\psi(Z).$ Note that the function $\psi$ depends on $n$, so we cannot directly apply the Portmanteau Lemma.\\
We aim to apply Lemma \ref{portm} with $X_n=\psi(Z_n)$ and $U_n=\psi(Z)$ defined above and with the function $g(x_1,x_2) = f(x_1) e^{x_2}$. 
By Lemma \ref{remainder}, we have that
$$\lim_{M\rightarrow \infty}\lim_{n\rightarrow \infty} \mathbb E |\min(0,M-e^{U_{n,2}})|=0.$$
Hence we get  by the second part of Lemma \ref{portm}
$$ \lim_{n\rightarrow \infty} \mathbb E g(X_n)- \mathbb E g(U_n) = 0.$$
Next we calculate $\mathbb E g(U_n).$
We have
\begin{eqnarray*}
\mathbb E_{{\beta_{n,0}}} g(U_n)=\mathbb Ef(U_{n,1}) e^{U_{n,2}}
=
\int_{\mathbb R^2} 
f\left(  {u_1} \right) e^{u_2} f_{U_n}(u_1,u_2) du,
\end{eqnarray*}
where $f_Y$ denotes the density of a random variable $Y.$
We use Lemma \ref{dens} to obtain that
$f_{U_n}(u)e^{u_2} = f_{Y}(u)$, where
\[
Y\sim\mathcal N\left(
\left(
\begin{array}{c}
0\\
{v_{22}/2}
\end{array}
\right),
\left(
\begin{array}{cc}
1 & \frac{v_{12}}{\sqrt{v_{11}}}\\
\frac{v_{12}}{\sqrt{v_{11}}} & v_{22}
\end{array}
\right) \right).
\]
Hence
\begin{eqnarray*}
\mathbb E_{{\beta_{n,0}}} g(U_n)
&=&
\int_{\mathbb R^2} 
f\left(  {u_1} \right)  f_{Y}(u_1,u_2) du =\mathbb Ef(Y_1),
\end{eqnarray*}
where $Y\sim\mathcal N(0,1).$
Hence
for any bounded continuous function $f$ we have shown 
$$\lim_{n\rightarrow \infty}| \mathbb E_{{{\beta_{n,0}}}+h/\sqrt{n}} f\left(\frac{\sqrt{n}(T_n-g({{\beta_{n,0}}})) - v_{12}}{\sqrt{v_{11}}}\right) - \mathbb Ef(Y)| = 0.$$
By the Portmanteau Lemma (note that $Y$ in the above display does not depend on $n$), 
 we thus have
$$\frac{\sqrt{n}(T_n-g({{\beta_{n,0}}})) - v_{12}}{\sqrt{v_{11}}} \stackrel{{{\beta_{n,0}}}+h/\sqrt{n}}{\rightsquigarrow} \mathcal N(0,1).$$
Therefore, by the differentiability assumption \ref{difg} on $g$, we get
$$\frac{\sqrt{n}(T_n-g({{\beta_{n,0}}}+h/\sqrt{n})) + h^T\dot g({{\beta_{n,0}}}) - P_{{\beta_{n,0}}} l_{{\beta_{n,0}}} h^Ts_{{\beta_{n,0}}}}{{V_{{\beta_{n,0}}}}^{1/2}} \stackrel{{{\beta_{n,0}}}+h/\sqrt{n}}{\rightsquigarrow} \mathcal N(0,1).$$

\end{proof}

\appendix

\section{Proofs for Section 14
}
\label{sec:proof.ois}

\noindent
Before proving Lemmas \ref{aux.gm1}, \ref{tau}, \ref{the} and \ref{sparse} we first recall a version of Theorem 2.4 from \cite{vdgeer13}. 
We denote $\eta_j := X_j - X_{-j}\gamma^0_j,j=1,\dots,p.$

\begin{theorem}[a version of Theorem 2.4 in \cite{vdgeer13}]\label{nodew.lr}
Suppose that 
conditions \ref{design} are satisfied
and assume that ${s_j\log p/n} =o(1).$
Consider the nodewise regression estimator $\hat\Theta_j$ and the corresponding $\hat\tau_j^2$ with $\lambda_j=\lambda \geq \tau\sqrt{\log p/n}$ for $j=1,\dots,p.$
Then for $\tau>1$, on the set 
\begin{eqnarray*}
\mathcal T_j&:=&
\{\|X_{-j}^T\eta_j\|_\infty/n \leq c\tau\sqrt{\log p /n},
\\
&&\;\; \|\hat\Sigma_{-j,-j}-\Sigma^0_{-j,-j}\|_\infty \leq c\tau\sqrt{\log p /n},
\\
&& \;\; |\eta_j^T\eta_j/n - \tau_j^2| \leq c\tau\sqrt{\log p/n}\},
\end{eqnarray*}
 (where $c$ is some sufficiently large constant),
we have the following claims for $j=1,\dots,p,$ 
$$\|\hat\gamma_j-\gamma_j^0\|_1\leq C_\tau s_j\sqrt{\log p/n},\;\;\;|\hat\tau_j^2-\tau_j^2|\leq C_\tau\sqrt{ s_j\log p/n},$$
$$\|\hat\Theta_j-\Theta_j^0\|_1\leq C_\tau s_j\sqrt{\log p/n},$$
for some constant $C_\tau>0.$
Moreover, for some constant $c_1>0$ we have
$$P(\mathcal T_j^c) \leq c_1 (2p)^{-\tau^2}.$$
\end{theorem}

\begin{proof}[Proof of Lemma \ref{aux.gm1}]
The proof of this lemma is essentially the same as the proof of Theorem \ref{strong2}, 
but we need to make adjustments to obtain an oracle bound for the expectation of the \textit{maximum} of $\ell_1$-errors
of $p$ Lasso estimators, i.e. we show an oracle bound for
$$(\mathbb E\max_{j=1,\dots,p}\|\hat\gamma_j-\gamma_j\|_1^k)^{1/k}.$$
First we summarize the oracle inequality for the nodewise regression which holds with high probability. 
Let $\eta_j := X_j - X_{-j}\gamma^0_j$ for $j=1,\dots,p$, then $\eta_j$ is a sub-Gaussian random vector with a universal constant, since $\Lambda_{\max}(\Theta_0)=1/\Lambda_{\min}(\Sigma_0)=\mathcal O(1).$
Under $1/\Lambda_{\min}(\Sigma_0)=\mathcal O(1),$ one can check that for some universal constant $L>0,$ it holds $\Lambda_{\min}(\Sigma^0_{-j,-j}) \geq L$.
Let $\mathcal T_j$ be as in Theorem \ref{nodew.lr}. Then on $\mathcal T:= \cap_{j=1}^p \mathcal T_j$ it holds
$\max_{j=1,\dots,p}\|\hat\gamma_j-\gamma_j^0\|_1\leq 16\lambda_j \max_{j=1,\dots,p}{s_j}/{L}.$
\\\\
We now proceed to show that the oracle inequality for the Lasso holds also in expectation. We follows the steps of the proof of Theorem \ref{strong2} to get
\begin{eqnarray*}
\mathbb E\max_{j=1,\dots,p} (\|\epsilon_j\|_n^2/\lambda_j + 2\|\gamma_j^0\|_1)^k
&\leq &
\mathbb E \max_{j=1,\dots,p}2^{k-1}\left( (\|\epsilon_j\|_n^2/\lambda_j)^k + (2\|\gamma_j^0\|_1)^k\right)
\end{eqnarray*}
By Lemma \ref{econ} we have 
\begin{eqnarray*}
\mathbb E \max_{j=1,\dots,p} (\|\epsilon_j\|_n^2)^k &= &
\mathcal O(1).
\end{eqnarray*}
Next observe that by assumption on the eigenvalues of $\Theta_0$, we have $\|\gamma_j^0\|_2=\mathcal O(1)$ and hence
$$\|\gamma_j^0\|_1^k \leq (\max_{j=1,\dots,p}\sqrt{s_j}\|\gamma_j^0\|_2)^k \leq \mathcal O(\max_{j=1,\dots,p}s_j^{k/2}).$$
Further steps again follow the steps of the proof of Theorem \ref{strong2}. Thus we obtain
\begin{eqnarray*}
\mathbb E_{}\max_{j=1,\dots,p} \|\hat \gamma_j-\gamma_j^0\|_1^k 
&=& \mathcal O(\max_{j=1,\dots,p}s_j^k\lambda_j^k)
,
\end{eqnarray*}
where we chose $\tau$ sufficiently large (this is possible since $k$ is fixed) so that 
 $$\max_{j=1,\dots,p}s_j^{k/2}\lambda_j^{-k} p^{-\tau^2 /2}=\mathcal O(\max_{j=1,\dots,p}s_j^k\lambda_j^k).$$
 Hence we conclude that 
\begin{eqnarray}\label{expoi}
(\mathbb E_{} \max_{j=1,\dots,p} \|\hat \gamma_j -\gamma_j^0\|_1^k)^{1/k}
=\mathcal O(\max_{j=1,\dots,p}s_j\lambda_j)
.
\end{eqnarray}

\end{proof}

\begin{proof}[Proof of Lemma \ref{tau}]\vskip 0.1cm
Before proving statements \ref{taudiff} and \ref{tau.d} of the lemma, we first prove that
$\mathbb E \max_{j=1,\dots,p} 1/(\hat\tau_j^2)^k = \mathcal O(1).$
Throughout the proof, we use the notation
$$\hat \Gamma_j:= (-\hat\gamma_{j,1},\dots,-\hat\gamma_{j,j-1},1,-\hat\gamma_{j,j+1},\dots,-\hat\gamma_{j,p}).$$ 
\textbf{Proof of $\mathbb E \max_{j=1,\dots,p} 1/(\hat\tau_j^2)^k = \mathcal O(1)$:}\\\\
We first show the rough bound $\mathbb E\max_{j=1,\dots,p} \frac{1}{(\hat\tau_j^2)^k} = \mathcal O(pn^{k/2}).$ First observe that
for each $j=1,\dots,p$  it holds for $t\geq 0$
\begin{eqnarray*}
P(\hat\tau_j^2 \leq t ) &=&
P(\hat\Gamma_j^T \hat\Sigma \hat\Gamma_j + \lambda_j \|\hat\gamma_j\|_1 \leq t )\\
& \leq &
P(\hat\Gamma_j^T \hat\Sigma \hat\Gamma_j \leq t \wedge \lambda_j \|\hat\gamma_j\|_1 \leq t )
\end{eqnarray*}
Using  the following lower bound 
\begin{eqnarray*}
\hat\Gamma_j^T \hat\Sigma \hat\Gamma_j 
&=&
 \hat\Sigma_{jj} - 2\hat\Sigma_{j,-j}^T \hat\gamma_j + \hat\gamma_j^T \hat\Sigma_{-j,-j}\hat\gamma_j
\\
&\geq &
 \hat\Sigma_{jj} - 2|\hat\Sigma_{j,-j}^T \hat\gamma_j| + \hat\gamma_j^T \hat\Sigma_{-j,-j}\hat\gamma_j
\\
&\geq &
 \hat\Sigma_{jj} - 2\|\hat\Sigma_{j,-j}\|_\infty \|\hat\gamma_j \|_1 + \hat\gamma_j^T \hat\Sigma_{-j,-j}\hat\gamma_j
\\
&\geq &
 \hat\Sigma_{jj} - 2\hat\Sigma_{jj}t/\lambda_j 
\\
&=& 
 \hat\Sigma_{jj}(1 - 2t/\lambda_j),
\end{eqnarray*}
 we obtain that
\begin{eqnarray*}
P(\hat\Gamma_j^T \hat\Sigma \hat\Gamma_j \leq t \wedge \lambda_j \|\hat\gamma_j\|_1 \leq t )
& \leq &
P(\hat\Sigma_{jj}(1 - 2t/\lambda_j) \leq t \wedge \lambda_j \|\hat\gamma_j\|_1 \leq t )
\\
& \leq &
P(\hat\Sigma_{jj}(1 - 2t/\lambda_j) \leq t).
\end{eqnarray*}
Next we use concentration results for $\hat \Sigma_{jj}$ around its mean under the sub-Gaussianity assumption on $X$.
For $t \leq \lambda_j/4 $ it holds that $1 - 2t/\lambda_j\geq 1/2$ and thus 
\begin{eqnarray*}
P(\hat\Sigma_{jj}(1 - 2t/\lambda_j) \leq t)
&=&
P\left(\hat\Sigma_{jj} - \Sigma_{jj}^0  \leq 2t- \Sigma_{jj}^0\right)
.
\end{eqnarray*}
For $0<t< \lambda_j/4$ and $n$ sufficiently large it holds that $2t- \Sigma_{jj}^0<0$ (by the minimal eigenvalue
 condition on $\Sigma_0$) and thus
\begin{eqnarray*}
P(\hat\Sigma_{jj}(1 - 2t/\lambda_j) \leq t)
&=&
P\left(\hat\Sigma_{jj} - \Sigma_{jj}^0  \leq 2t- \Sigma_{jj}^0\right)\\
&\leq &
P\left(|\hat\Sigma_{jj} - \Sigma_{jj}^0|  \geq |2t- \Sigma_{jj}^0|\right)\\
&\leq &
e^{-c\left(\Sigma_{jj}^0 -2t  \right) \sqrt{n/\log p}n}
,
\end{eqnarray*}
for some constant $c>0.$
Hence collecting the above inequalities, we have so far shown that for any $0<t< \lambda_j/4$ and $n$ sufficiently large it holds 
\begin{equation}\label{bound.}
P(\hat\tau_j^2 \leq t ) \leq 
e^{-c\left(\Sigma_{jj}^0 -2t\right) \sqrt{n/\log p}}
.
\end{equation}
Then by rewriting the expectation as an integral 
\begin{eqnarray*}
\mathbb E \max_{j=1,\dots,p} \frac{1}{(\hat\tau_j^2)^k} 
&=& 
\int_0^\infty P( \max_{j=1,\dots,p}1/(\hat\tau_j^2)^k > x) dx 
\\
&=& 
\int_0^\infty \max_{j=1,\dots,p} pP( 1/(\hat\tau_j^2)^k > x) dx 
\\
&=& 
p\int_0^1\max_{j=1,\dots,p} P(1/(\hat\tau_j^2)^k > x) dx 
\\
&&
+\;\;p\int_1^{\left(\lambda_j/4\right)^{-k}} \max_{j=1,\dots,p} P(1/(\hat\tau_j^2)^k > x) dx
\\
&&+\;\;
p\int_{\left(\lambda_j/4\right)^{-k}}^\infty \max_{j=1,\dots,p}P(1/(\hat\tau_j^2)^k > x) dx
\\
&\leq &
p+
p{\left(\lambda_j/4\right)^{-k}}
+
p\underbrace{\int_{\left(\lambda_j/4\right)^{-k}}^\infty P(1/(\hat\tau_j^2)^k > x) dx}_{ii}
\end{eqnarray*}
Next we calculate an upper bound on $ii.$ 
\begin{eqnarray*}
ii=\int_{\left(\lambda_j/4\right)^{-k}}^\infty P(1/(\hat\tau_j^2)^k  > x) dx
 &=&
\int_{\left(\lambda_j/4\right)^{-k}}^\infty P(1/\hat\tau_j^2 > x^{1/k}) dx \\
 &=&
\int_{\left(\lambda_j/4\right)^{-1}}^\infty P(\hat\tau_j^2 < x^{-1}) dx 
\end{eqnarray*}
Now we can use the bound \eqref{bound.} since $x^{-1}\leq \lambda_j/4.$ Using the bound and by standard calculations, we obtain
\begin{eqnarray*}
\int_{\left(\lambda_j/4\right)^{-1}}^\infty P(\hat\tau_j^2 < x^{-1}) dx 
&\leq & 
\int_{\left(\lambda_j/4\right)^{-1}}^\infty 
e^{-c\left(\Sigma_{11}^0 -2/x\right) \sqrt{n/\log p}n}
dx
= o(1).
\end{eqnarray*}
Hence we obtain the rough bound
\begin{eqnarray*}
{\mathbb E \max_{j=1,\dots,p} \frac{1}{(\hat\tau_j^2)^k}} 
= \mathcal O\left(p\left(\lambda_j/4\right)^{-k}\right) = \mathcal O(pn^{k/2}/(\log p)^{k/2}).
\end{eqnarray*}
Define, 
 for $\tau>1$ and $j=1,\dots,p$, the sets 
\begin{eqnarray*}
\mathcal T_j&:=&
\{\|X_{-j}^T\eta_j\|_\infty/n \leq c\tau\sqrt{\log p /n},\\
&&\;\; \|\hat\Sigma_{-j,-j}-\Sigma^0_{-j,-j}\|_\infty \leq c\tau\sqrt{\log p /n},
\\
&&\;\; \|\eta_j^T\eta_j/n - \tau_j^2\|_\infty \leq c\tau\sqrt{\log p/n}\},
\end{eqnarray*}
 (where $c$ is some sufficiently large constant).
Then by Theorem \ref{nodew.lr}, on $\mathcal T=\cap_{j=1}^p\mathcal  T_j$ we have $\max_{j=1,\dots,p} 1/(\hat\tau_j^2)^k \leq C_\tau $ for some constant $C_\tau>0.$ But then and by the Cauchy-Schwarz inequality
\begin{eqnarray*}
{\mathbb E \max_{j=1,\dots,p}\frac{1}{(\hat\tau_j^2)^k}} 
& = &
{\mathbb E \max_{j=1,\dots,p}\frac{1}{(\hat\tau_j^2)^k}}1_{\mathcal T} 
+
{\mathbb E \max_{j=1,\dots,p}\frac{1}{(\hat\tau_j^2)^k}}1_{\mathcal T^c}
\\
& \leq &
\mathcal O(1) + \sqrt{\mathcal O(p{n}^{k/2})} \sqrt{2p}(2p)^{-\tau^2 /2} = \mathcal O(1),
\end{eqnarray*}
where we chose $\tau$ sufficiently large.\\

\noindent
\textbf{Proof of part 
\ref{taudiff} }
First we show that $\mathbb E(\hat\tau_j^2)^k =\mathcal O(1).$ We have
\begin{eqnarray*}
{\hat\tau_j^2} & = & 
\|X_j - X_{-j}\hat\gamma_j\|_n^2 + \lambda_j \|\hat\gamma_j\|_1
\\
&=&
\hat\Gamma_j^T \hat\Sigma \hat \Gamma_j /n +\lambda_j \|\hat\gamma_j\|_1
\\
&\leq &
\|\hat\Gamma_j\|_1^2 \|\hat\Sigma\|_\infty + \lambda_j \|\hat\gamma_j\|_1
.
\end{eqnarray*}
Hence by basic calculations
\begin{eqnarray*}
\mathbb E\max_{j=1,\dots,p}  ({\hat\tau^2_j})^k &\leq &
\mathbb E\max_{j=1,\dots,p}\left[ \|\hat\Gamma_j\|_1^2 \|\hat\Sigma\|_\infty + \lambda_j \|\hat\gamma_j\|_1\right]^k
 \\
&\leq & 
 2^{k-1}\mathbb E\max_{j=1,\dots,p}\left[ (\|\hat\Gamma_j\|_1^2 \|\hat\Sigma\|_\infty)^k + (\lambda_j \|\hat\gamma_j\|_1)^k\right]
.
\end{eqnarray*}
We have 
\begin{eqnarray*}
\mathbb E\max_{j=1,\dots,p}\|\hat\gamma_j\|_1^k & \leq & 
\mathbb E\max_{j=1,\dots,p}(\|\hat\gamma_j-\gamma^0_j\|_1 +\|\gamma_j^0\|_1)^k
\\
&\leq &
\mathbb E\max_{j=1,\dots,p} 2^{k-1}(\|\hat\gamma_j-\gamma^0_j\|_1^k +\|\gamma_j^0\|_1^k)
\\
& =&
	\mathcal O(s_j^{k/2}).
\end{eqnarray*}
Hence

\begin{eqnarray*}
\mathbb E\max_{j=1,\dots,p}  ({\hat\tau^2_j})^k 
&\leq &
\mathbb E\max_{j=1,\dots,p}\left[ \|\hat\Gamma_j\|_1^2 \|\hat\Sigma\|_\infty + \lambda_j \|\hat\gamma_j\|_1\right]^k
 \\
&=& \mathcal O(s_j^{2k}).
\end{eqnarray*}

\noindent
Hence, by Theorem \ref{nodew.lr}, on $\mathcal T$ we have that $\max_{j=1,\dots,p}\hat\tau_j^2=\mathcal O(1)$, hence it follows 
\begin{eqnarray*}
{\mathbb E \max_{j=1,\dots,p}{(\hat\tau_j^2)^k}} 
= 
{\mathbb E \max_{j=1,\dots,p}{(\hat\tau_j^2)^k} }1_{\mathcal T} 
+
{\mathbb E \max_{j=1,\dots,p}{(\hat\tau_j^2)^k} }1_{\mathcal T^c}=\mathcal O(1).
\end{eqnarray*}
We have under $1/\Lambda_{\min}(\Theta_0 )=\mathcal O(1)$ that $\tau_j^2=1/\Theta_{jj}^0=\mathcal O(1)$ and hence
$$\mathbb E\max_{j=1,\dots,p}|\hat\tau_j^2 - \tau_j^2|^k = \mathcal O(\mathbb E\max_{j=1,\dots,p}(\hat\tau_j^2)^k + (\tau_j^2)^k) = \mathcal O(1).$$
We can then apply the same procedure as before to get
\begin{eqnarray*}
{\mathbb E \max_{j=1,\dots,p}{|\hat\tau_j^2 - \tau_j^2|^k}} 
 &= &
{\mathbb E \max_{j=1,\dots,p}{|\hat\tau_j^2 - \tau_j^2|^k} }1_{\mathcal T} 
+
{\mathbb E \max_{j=1,\dots,p}{|\hat\tau_j^2 - \tau_j^2|^k} }1_{\mathcal T^c}
\\
&=&
\mathcal O(\max_{j=1,\dots,p}\sqrt{s_j}\lambda_j).
\end{eqnarray*}

\noindent
\textbf{Proof of part 
\ref{tau.d} }
First we have
\begin{eqnarray*}
&& 
\mathbb E \max_{j=1,\dots,p}\left\lvert\frac{1}{\hat\tau_j^2} -\frac{1}{\tau_j^2}\right\rvert^k 
\\
&&\leq  
\mathbb E \max_{j=1,\dots,p} \frac{|\hat\tau_j^2 - \tau_j^2|^k}{(\hat\tau_j^2)^k(\tau_j^2)^k}
\\
&&\leq 
\max_{j=1,\dots,p} 1/(\tau_j^2)^k \sqrt{\mathbb E \max_{j=1,\dots,p}{|\hat\tau_j^2 - \tau_j^2|^{2k}} }
 \sqrt{\mathbb E \max_{j=1,\dots,p}1/(\hat\tau_j^2)^{2k}}.
\end{eqnarray*}
Using that $\mathbb E \max_{j=1,\dots,p}1/(\hat\tau_j^2)^{2k} =\mathcal O(1)$ and Part \ref{taudiff}, we obtain the claim.

\end{proof}

\begin{proof}[Proof of Lemma \ref{the}]
For some $\tau>0$ and each $j=1,\dots,p$ define the sets
\begin{eqnarray*}
\mathcal T_j&:=&
\{\|X_{-j}^T\eta_j\|_\infty/n \leq c\tau\sqrt{\log p /n},\\
&&\;\; \|\hat\Sigma_{-j,-j}-\Sigma^0_{-j,-j}\|_\infty \leq c\tau\sqrt{\log p /n},
\\
&&\;\; \|\eta_j^T\eta_j/n - \tau_j^2\|_\infty \leq c\tau\sqrt{\log p/n}\},
\end{eqnarray*}
By Theorem \ref{nodew.lr}, when $\lambda_j \geq c\tau\sqrt{\log p/n}$ uniformly in $j$, we have on the set 
$\cap_{i=1}^p \mathcal T_j$ that 
$$\max_{j=1,\dots,p}\|\hat\Theta_j - \Theta^0_j\|_1\leq C_\tau\max_{j=1,\dots,p}s_j\lambda_j,\;\;\; 
\max_{j=1,\dots,p}|\hat\tau_j^2 - \tau_j^2| \leq C_\tau\max_{j=1,\dots,p}\sqrt{s_j}\lambda_j,$$
for some $C_\tau>0.$
Next we rewrite
$$\mathbb E\max_{j=1,\dots,p}\|\hat \Theta_j-\Theta^0_j\|_1^k = \mathbb E\max_{j=1,\dots,p}\|\hat \Theta_j-\Theta^0_j\|_1^k 1_{\mathcal T } 
+ \mathbb E\max_{j=1,\dots,p}\|\hat \Theta_j-\Theta^0_j\|_1^k 1_{\mathcal T^c}.
$$
Then
\begin{eqnarray*}
&&\mathbb E_{}\max_{j=1,\dots,p}\|\hat \Theta_j-\Theta^0_j\|_1^k
\\
&& \leq  
\mathbb E_{}\max_{j=1,\dots,p} \left[ \|\hat\gamma_j-\gamma^0_j\|_1/\hat\tau_j^2 
+
 \|\gamma^0_j\|_1|1/\hat \tau_j^2- 1/\tau_j^2 |\right]^k
\\
&&\leq 
\mathbb E_{}\max_{j=1,\dots,p} 2^{k-1} \left[ (\|\hat\gamma_j-\gamma^0_j\|_1/\hat\tau_j^2)^k 
+
 (\|\gamma^0_j\|_1|1/\hat \tau_j^2- 1/\tau_j^2 |)^k\right]
\\
&&= \mathcal O((s_j\lambda_j)^k) 
,
\end{eqnarray*}
where in the last display we used Lemmas \ref{aux.gm1} and \ref{tau}.

\end{proof}

\begin{proof}[Proof of Lemma \ref{sparse}]
Let $\eta_j:= X_{j}-X_{-j}\gamma_j^0.$
The Karush-Kuhn-Tucker corresponding to the optimization problem \eqref{nc} give
$$\hat\Sigma_{-j,-j} \hat \gamma_j + \lambda_j \hat Z_j = X_{-j}^T \eta_j/n,$$
where $\hat Z_{j,i}=\textrm{sign}(\hat\gamma_{j,i})$ if $\hat\gamma_{j,i}\not = 0$ and $\hat Z_{j,i}\in [-1,1]$ otherwise, for $i=1,\dots,p$.
Rearranging them, we obtain
\begin{equation}\label{spar}
\Sigma_{-j,-j}^0 (\hat \gamma_j-\gamma_j^0) + \lambda_j \hat Z_j = X_{-j}^T \eta_j/n + 
(\Sigma_{-j,-j}^0- \hat\Sigma_{-j,-j}) (\hat \gamma_j-\gamma_j^0).
\end{equation}
Firstly, 
\begin{eqnarray*}
\|\Sigma_{-j,-j}^0 (\hat \gamma_j-\gamma_j^0)\|_2^2
&\leq & 
\Lambda_{\max}(\Sigma_{-j,-j}^0)  \|X_{-j}(\hat\gamma_j-\gamma_j^0)\|_n^2
\\
&& \; + \; \Lambda_{\max}(\Sigma_{-j,-j}^0)  \|\hat\gamma_j-\gamma_j^0\|_1^2 \|\hat\Sigma-\Sigma_0\|_\infty.
\end{eqnarray*}
Secondly,
$$\|(\Sigma_{-j,-j}^0 -\hat\Sigma_{-j,-j})(\hat \gamma_j-\gamma_j^0)\|_\infty
\leq \|\Sigma_{-j,-j}^0 -\hat\Sigma_{-j,-j}\|_\infty \|\hat\gamma_j-\gamma_j^0\|_1.$$ 
Denote $\hat s_j := \|\hat\gamma_j\|_0.$ On the set
$$\mathcal T_j:=\{\|X_{-j}^T\eta_j/n\|_\infty \leq\lambda_0, \|(\Sigma_{-j,-j}^0 -\hat\Sigma_{-j,-j})(\hat\gamma_j-\gamma_j^0)\|_\infty\leq \lambda_0\},$$ 
we have
\begin{eqnarray*}
&&\|\lambda_j \hat Z_j + X_{-j}^T\eta_j/n +(\Sigma_{-j,-j}^0 -\hat\Sigma_{-j,-j})(\hat \gamma_j-\gamma_j^0)\|_2^2 \\
&=&
\sum_{i=1}^p |\lambda_j \hat Z_{j,i} + X_i^T\eta_j +e_i^T(\Sigma_{-j,-j}^0 -\hat\Sigma_{-j,-j})(\hat \gamma_j-\gamma_j^0) |^2
\\
&\geq & (\lambda_j - 2\lambda_0)^2\hat s.
\end{eqnarray*}
Combining the above observations, we obtain
\begin{eqnarray*}
&& \Lambda_{\max}(\Sigma_{-j,-j}^0) \|X_{-j}(\hat\gamma_j-\gamma_j^0)\|_n^2 + \; \Lambda_{\max}(\Sigma_{-j,-j}^0)  \|\hat\gamma_j-\gamma_j^0\|_1^2 \|\hat\Sigma-\Sigma_0\|_\infty
\\
&&\geq 
\|\Sigma_{-j,-j}^0 (\hat \gamma_j-\gamma_j^0) 
\|_2^2
\\
&&=
\|- \lambda \hat Z_j + X_{-j}^T \eta_j/n + (\Sigma_{-j,-j}^0- \hat\Sigma_{-j,-j}) (\hat \gamma_j-\gamma_j^0) \|_2^2
\\
&&\geq 
(\lambda_j -2\lambda_0)^2\hat s.
\end{eqnarray*}
Hence on the set $\mathcal T$,
$$\hat s_j \leq \frac{\Lambda_{\max}(\Sigma_{-j,-j}^0) \|X_{-j}(\hat\gamma_j-\gamma_j^0)\|_n^2
+ \; \Lambda_{\max}(\Sigma_{-j,-j}^0)  \|\hat\gamma_j-\gamma_j^0\|_1^2 \|\hat\Sigma-\Sigma_0\|_\infty
}{(\lambda_j - 2\lambda_0)^2}.$$
Finally, taking $\lambda_0:=c\sqrt{\log p/n}$ for some $c>0$ sufficiently large and taking $\lambda \geq 3\lambda_0$ and under the assumption $\Lambda_{\max}(\Sigma_0)=\mathcal O(1)$, we obtain
$$\hat s_j =\mathcal O_P(s_j).$$

\end{proof}

\section{Additional proofs for Section 16
}
\label{sec:A}
In this section we give Lemma \ref{ggm.error.lc} and its proof, but we need the following auxiliary Lemmas \ref{mgf}, 
\ref{ggm.error}, \ref{lim1}.

\begin{lemma}
\label{mgf}
Let $x\sim \mathcal N(0_p,\Sigma_0)$ and let $\Theta_0=\Sigma_0^{-1}.$
Then for any $t\in\mathbb R$ and $A\in\mathbb R^{p\times p}$ such that $\Theta_0-2tA$ is symmetric and positive definite it holds
$$\mathbb E_{\Theta_0}e^{tx^T A x} = \left(\frac{\emph{det}(\Theta_0)}{\emph{det}(\Theta_0-2tA)}\right)^{1/2}.$$
\end{lemma}

\begin{proof}[Proof of Lemma \ref{mgf}]
By direct calculation, we obtain
\begin{eqnarray*}
\mathbb E_{\Theta_0}e^{tx^T A x} 
& = &
\int_{\mathbb R^p} \frac{\textrm{det}(\Theta_0)^{1/2}}{(2\pi)^{p/2}} e^{-\frac{1}{2} x^T \Theta_0 x} e^{tx^T A x}  dx
\\
&=&
\int_{\mathbb R^p} \frac{\textrm{det}(\Theta_0)^{1/2}}{(2\pi)^{p/2}} e^{-\frac{1}{2} x^T (\Theta_0 - 2tA) x} dx
\\
&=&
\int_{\mathbb R^p} \frac{\textrm{det}(\Theta_0)^{1/2}\textrm{det}(\Theta_0-2tA)^{1/2}}{(2\pi)^{p/2}\textrm{det}(\Theta_0-2tA)^{1/2}} e^{-\frac{1}{2} x^T (\Theta_0 - 2tA) x} dx
\\
&=&
\frac{\textrm{det}(\Theta_0)^{1/2}}{\textrm{det}(\Theta_0-2tA)^{1/2}}.
\end{eqnarray*}

\end{proof}

\begin{lemma}\label{ggm.error} 
Suppose that $\Theta_0+H/\sqrt{m_n}$ is a symmetric positive definite matrix.
Let $p_{\Theta_0}$ be the joint density of the random sample $x_1,\dots,$ $x_n$, where each $x_i$ is an $\mathcal N(0,\Theta_0^{-1})$-distributed random vector. Then it holds
{\normalfont
\begin{eqnarray*}
&&\mathbb E_{\Theta_0} \left( \frac{p_{\Theta_0 + H/\sqrt{m_n}}(x)}{p_{\Theta_0}(x)} - 1 - n\textrm{tr}((\hat\Sigma - \Sigma_0)H/\sqrt{m_n}) \right)^2
\\
&&\;\;=
\frac{\textrm{det}(\Theta_0 + H/\sqrt{m_n})^{n}}{\textrm{det}(\Theta_0)^{n}}\left(\frac{\textrm{det}(\Theta_0)}{\textrm{det}(\Theta_0+2H)}\right)^{n/2}
- 1  + \sum_{i=1}^n \textrm{var}(x_i^T H x_i) 
  \\
&&  
\quad\quad - 2 
n 
\textrm{tr}[ ( (\Theta_0+H)^{-1} 
- \Sigma_0) H].
\end{eqnarray*}
}
\end{lemma}

\begin{proof}[Proof of Lemma \ref{ggm.error}]
The density is given by
$$p_{\Theta_0}(x_1,\dots,x_n) 
= \frac{\textrm{det}(\Theta_0)^{n/2}}{(2\pi)^{np/2}} e^{-\frac{1}{2}\sum_{i=1}^n x_i^T \Theta_0 x_i}$$
For simplicity of notation, denote $U:=H/\sqrt{m_n}.$
Then we have
$$\frac{p_{\Theta_0 + U}(x) }{p_{\Theta_0}(x)} - 1 =
\frac{\textrm{det}(\Theta_0 + U)^{n/2}e^{-\frac{1}{2}\sum_{i=1}^n x_i^T U x_i}}{\textrm{det}(\Theta_0)^{n/2}} -1$$
The score function is given by  $s_{\Theta_0}(x) = n(\hat\Sigma -  \Sigma_0)/2.$
Let
$$Z:=\textrm{vec}(U)^T \textrm{vec}(s_{\Theta_0}(x)) = \textrm{tr}(n(\hat\Sigma - \Sigma_0)U)=
\sum_{i=1}^n x_i^T U x_i- n \textrm{tr}(\Sigma_0U).$$
First observe that
$$\mathbb E_{\Theta_0}Z^2 = \textrm{var}(\sum_{i=1}^n x_i^T U x_i) = \sum_{i=1}^n \textrm{var}(x_i^T U x_i).$$
We have
\begin{eqnarray*}
&&
\mathbb E _{\Theta_0}\left( \frac{p_{\Theta_0 + U}(x)}{p_{\Theta_0}(x)} - 1 - Z \right)^2
\\
&&=
\mathbb E_{\Theta_0} \left( \frac{\textrm{det}(\Theta_0 + U)^{n/2}e^{-\frac{1}{2}\sum_{i=1}^nx_i^T U x_i}}{\textrm{det}(\Theta_0)^{n/2}} -1 - Z \right)^2
\\
&&=
\frac{\textrm{det}(\Theta_0 + U)^{n}}{\textrm{det}(\Theta_0)^{n}}\mathbb E_{\Theta_0} e^{-\sum_{i=1}^nx_i^T U x_i} + 1 + \mathbb E_{\Theta_0}Z^2+
\\
&& \;\; - \; 2 \frac{\textrm{det}(\Theta_0 + U)^{n/2}}{\textrm{det}(\Theta_0)^{n/2}} \mathbb E_{\Theta_0} e^{-\frac{1}{2}\sum_{i=1}^nx_i^T U x_i}
+
 2 \mathbb E_{\Theta_0}Z   \\
&& \;\;
-\; 2 \frac{\textrm{det}(\Theta_0 + U)^{n/2}}{\textrm{det}(\Theta_0)^{n/2}}
\mathbb E_{\Theta_0}  Z  e^{-\frac{1}{2}\sum_{i=1}^nx_i^T U x_i}.
\end{eqnarray*}
Using Lemma \ref{mgf}, since $\Theta_0+U/2$ is a symmetric positive definite matrix,  we obtain
\begin{eqnarray*}
\mathbb E_{\Theta_0} \left( \frac{p_{\Theta_0 + U}(x)}{p_{\Theta_0}(x)} - 1 - Z \right)^2
&&=
\frac{\textrm{det}(\Theta_0 + U)^{n}}{\textrm{det}(\Theta_0)^{n}}\left(\frac{\textrm{det}(\Theta_0)}{\textrm{det}(\Theta_0+2U)}\right)^{n/2}
+ 1 \\
&& \;\;+ \sum_{i=1}^n \textrm{var}_{\Theta_0}(x_i^T U x_i) 
\\\;\;
&& - \; 2 \frac{\textrm{det}(\Theta_0 + U)^{n/2}}{\textrm{det}(\Theta_0)^{n/2}} \left(\frac{\textrm{det}(\Theta_0)}{\textrm{det}(\Theta_0+U)}\right)^{n/2}
  \\
&&  \;\;
-\; 2 \frac{\textrm{det}(\Theta_0 + U)^{n/2}}{\textrm{det}(\Theta_0)^{n/2}}
\mathbb E_{\Theta_0}  Z  e^{-\frac{1}{2}\sum_{i=1}^nx_i^T U x_i}
\\
&&=
\frac{\textrm{det}(\Theta_0 + U)^{n}}{\textrm{det}(\Theta_0)^{n}}\left(\frac{\textrm{det}(\Theta_0)}{\textrm{det}(\Theta_0+2U)}\right)^{n/2}
- 1
\\
&& + \sum_{i=1}^n \textrm{var}_{\Theta_0}(x_i^T U x_i) 
  \\
&&  
 \underbrace{-\; 2\frac{\textrm{det}(\Theta_0 + U)^{n/2}}{\textrm{det}(\Theta_0)^{n/2}}
\mathbb E_{\Theta_0}  Z  e^{-\frac{1}{2}\sum_{i=1}^nx_i^T U x_i}}_{i}.
\end{eqnarray*}
Next we calculate $i$. We have again by Lemma \ref{mgf}
$$\mathbb E_{\Theta_0} e^{tZ} 
= e^{-n\textrm{tr}(\Sigma_0 U)t}\mathbb E_{\Theta_0} e^{t\sum_{i=1}^nx_i^T U x_i}
= e^{-n\textrm{tr}(\Sigma_0 U)t} \left(\frac{\textrm{det}(\Theta_0)}{\textrm{det}(\Theta_0-2tU)}\right)^{n/2}.$$
Since $e^{tZ}\leq e^{Z}$ for $t<0$ and  $ \mathbb E_{\Theta_0} e^{Z} < \infty$, we can interchange differentiation and integration below to obtain
$$\mathbb E_{\Theta_0} Ze^{tZ} = \mathbb E_{\Theta_0} (e^{tZ} )' = (\mathbb E_{\Theta_0} e^{tZ})'.$$
Hence
\begin{eqnarray}\nonumber
\mathbb E_{\Theta_0} Ze^{tZ} &=& (\mathbb E_{\Theta_0} e^{tZ})' 
\\\nonumber
&=& 
e^{-n\textrm{tr}(\Sigma_0 U)t} \left(\frac{\textrm{det}(\Theta_0)}{\textrm{det}(\Theta_0-2tU)}\right)^{n/2} 
\times 
\\\label{prod}&&
n\left[
\textrm{tr}( (\Theta_0-2tU)^{-1} U)
-\textrm{tr}(\Sigma_0 U) 
\right]
\end{eqnarray}
Finally, taking $t=-1/2$ in \eqref{prod}, we obtain
$$i=-\; 2 
n 
\textrm{tr}[ ( (\Theta_0+U)^{-1} 
- \Sigma_0) U]
.$$
Hence
\begin{eqnarray*}
\mathbb E_{\Theta_0} \left( \frac{p_{\Theta_0 + U}(x)}{p_{\Theta_0}(x)} - 1 - Z \right)^2
&=&
\frac{\textrm{det}(\Theta_0 + U)^{n}}{\textrm{det}(\Theta_0)^{n}}\left(\frac{\textrm{det}(\Theta_0)}{\textrm{det}(\Theta_0+2U)}\right)^{n/2}
- 1 \\
&&+ \sum_{i=1}^n \textrm{var}_{\Theta_0}(x_i^T U x_i) 
  \\
&&  
-\; 2 
n 
\textrm{tr}[ ( (\Theta_0+U)^{-1} 
- \Sigma_0) U]
\end{eqnarray*}
This finishes the proof.

\end{proof}

\begin{lemma}\label{lim1}
Let $0<\delta=\delta_n \rightarrow 0 $ and let $a,b=\mathcal O(1).$
Then
$$\left( 1 + \frac{a{\delta} }{\sqrt{n}( b\sqrt{\delta} + \sqrt{n})} \right)^{n} - 1 = O(\delta).$$
\end{lemma}

\begin{proof}[Proof of Lemma \ref{lim1}]
\begin{eqnarray*}
\left( 1 + \frac{a{\delta} }{\sqrt{n}( b\sqrt{\delta} + \sqrt{n})} \right)^{n} - 1
&=&
e^{ {n} \log \left(1+\frac{{\delta} a}{\sqrt{n}( b\sqrt{\delta} + \sqrt{n})}\right) } -1
\\
&=&
e^{ n \left[ \frac{a{\delta} }{\sqrt{n}( b\sqrt{\delta} + \sqrt{n})} + 
o\left(\frac{a{\delta} }{\sqrt{n}( b\sqrt{\delta} + \sqrt{n})}\right) \right] } -1
\end{eqnarray*}
Next
$$\frac{a n{\delta}}{\sqrt{n}(b\sqrt{\delta} + \sqrt{n})} =\mathcal O( \delta).$$
Hence, and using that $e^x -1 = o(x)$ for $x\rightarrow 0$, we obtain
\begin{eqnarray*}
e^{ n\left[ \frac{a{\delta}}{\sqrt{n}(b\sqrt{\delta} + \sqrt{n})} + 
o\left(\frac{a{\delta}}{\sqrt{n}(b\sqrt{\delta} + \sqrt{n})}\right) \right] } -1
&=&
\mathcal O(\delta).
\end{eqnarray*}

\end{proof}

\noindent

\begin{lemma}\label{ggm.error.lc}
Let $p_{\Theta_0}$ be the joint density of the random sample $X_1,\dots,X_n$, where each $X_i$ is an $\mathcal N(0,\Theta_0^{-1})$-distributed random vector.
Let $H := \Theta_0 (\xi_1\xi_2^T$ $+\xi_2\xi_1^T)\Theta_0/\sigma_{}.$
Then
\begin{eqnarray*}
\mathbb E_{\Theta_0} \left( \frac{p_{\Theta_0 + H/\sqrt{m_n}}(x)}{p_{\Theta_0}(x)} - 1 - \emph{tr}(s_{\Theta_0}(X)H)/\sqrt{m_n} \right)^2
=\mathcal O(\delta_n).
\end{eqnarray*}
\end{lemma}

\begin{proof}[Proof of Lemma \ref{ggm.error.lc}]
We apply Lemma \ref{ggm.error} with 
$$H := \Theta_0 (\xi_1\xi_2^T+\xi_2\xi_1^T)\Theta_0/\sigma_{},$$ where
$\sigma_{}^2:= \xi_1^T \Theta^0 \xi_1\xi_2^T \Theta^0 \xi_2 + (\xi_1^T \Theta^0 \xi_2)^2 $.
To apply Lemma \ref{ggm.error}, we need to show that  $\Theta_0 + H/\sqrt{m_n}$ is a symmetric, positive definite matrix (for $n$ sufficiently large).
 This can be seen as follows.
First note that 
$H := \Theta_0 (\xi_1\xi_2^T+\xi_2\xi_1^T)\Theta_0/\sigma_{}$ is symmetric and $\Theta_0$ is symmetric and hence the symmetry of the sum follows.
Next we look at positive definiteness.
\\
For any $u\in\mathbb R^p$ we have by the Cauchy-Schwarz inequality
\begin{eqnarray*}
u^T H u& =& u^T\Theta_0 (\xi_1\xi_2^T+\xi_2\xi_1^T)\Theta_0 u /\sigma_{} 
\\
&\leq &2 \sqrt{\xi_1^T \Theta_0 \xi_1 u^T \Theta_0 u}
\sqrt{\xi_2^T \Theta_0 \xi_2 u^T \Theta_0 u}/\sigma_{}  \leq  \Lambda_{\max}(\Theta) u^T u.
\end{eqnarray*}
Thus 
\begin{eqnarray*}
u^T \Theta_0 u + u^T H u/\sqrt{m_n}  & \geq &
 u^T \Theta_0 u - \mathcal O(\Lambda_{\max}(\Theta_0)u^T u)/\sqrt{m_n}
\\
&\geq &
\left[\Lambda_{\min}(\Theta_0)  - \mathcal O(1)/\sqrt{m_n}\right] u^T u.
\end{eqnarray*}
This shows that the matrix $\Theta_0 + H/\sqrt{m_n}$ is  positive definite for $n$ sufficiently large.
But then we can apply Lemma \ref{ggm.error} which gives
\begin{eqnarray*}
&&\mathbb E_{\Theta_0} \left( \frac{p_{\Theta_0 + U}(x)}{p_{\Theta_0}(x)} - 1 -\textrm{tr}(s_{\Theta_0}(X)H)/\sqrt{m_n} \right)^2
\\
&&=
\underbrace{
\frac{\textrm{det}(\Theta_0 + U)^{n}}{\textrm{det}(\Theta_0)^{n}}\left(\frac{\textrm{det}(\Theta_0)}{\textrm{det}(\Theta_0+2U)}\right)^{n/2}
}_{i} -1
\\&&\;\;  + \;\; \underbrace{\sum_{i=1}^n \textrm{var}_{\Theta_0}(x_i^T U x_i) }_{ii}
  \\
&&  \;\;
-\; 2 
n 
\underbrace{\textrm{tr}[ ( (\Theta_0+U)^{-1} 
- \Sigma_0) U]}_{iii}.
\end{eqnarray*}
Now we calculate the terms $i, ii, iii.$
\vskip 0.2cm
\noindent
\textit{$1.)$ Calculation of $i$}.\vskip 0.2cm
\noindent
 First we calculate $i$. Observe that since $H$ is of rank $2$ it follows
\begin{eqnarray*}
i&=&
\textrm{det}(\Theta_0 + H/\sqrt{m_n}) =\textrm{det}[\Theta_0 + \Theta_0 (\xi_1\xi_2^T+\xi_2\xi_1^T)\Theta_0/(2\sigma_{}\sqrt{m_n})]\\
&=&
(1 + \xi_1^T\Theta_0  \xi_2/(2\sigma_{}\sqrt{m_n}) )^2\textrm{det}(\Theta_0). 
\end{eqnarray*}
And hence
\begin{eqnarray*}
&&\frac{\textrm{det}(\Theta_0 + H/\sqrt{m_n})^{n}}{\textrm{det}(\Theta_0)^{n}}
\left(\frac{\textrm{det}(\Theta_0)}{\textrm{det}(\Theta_0+2H/\sqrt{m_n})}\right)^{n/2}
\\
&&=
\left[ \frac{(1 + \xi_1^T\Theta_0  \xi_2/(2\sigma_{}\sqrt{m_n}) )^2}{ (1 + \xi_1^T\Theta_0  \xi_2/(\sigma_{}\sqrt{m_n}) ) } \right]^{n}
\end{eqnarray*}
\vskip 0.2cm
\noindent
\textit{$2.)$ Calculation of $ii$}.\vskip 0.2cm
\noindent
The following property holds: if $Y\sim \mathcal N_2(0, S)$, then it holds
$$\textrm{var}(Y_1 Y_2) = S_{11}S_{22} + S_{12}^2.$$
Further observe that since  $\Theta_0 x_i \sim \mathcal N(0,\Theta_0)$ we obtain
\begin{eqnarray*}
\textrm{var}_{\Theta_0}(x_i^T H x_i/\sqrt{m_n}) &=& 
 \textrm{var}_{\Theta_0}(\xi_1^T \Theta_0 x_i x_i^T\Theta_0 \xi_2^T )/(\sigma_{}^2 m_n ) \\
&=& 
(\xi_1^T \Theta_0 \xi_1 \xi_2^T \Theta_0 \xi_2 + (\xi_1^T \Theta_0 \xi_2)^2)
/(\sigma_{}^2 m_n )=
1/m_n
.\end{eqnarray*}
Thus 
$$ii= n /m_n.$$
\vskip 0.2cm
\noindent
\textit{$3.)$ Calculation of $iii$}.\vskip 0.2cm
\noindent
Finally, by inversion of a sum of a matrix with another matrix of rank $2$, we get (we omit the calculations)
\begin{eqnarray*}
iii=\textrm{tr}[ ( (\Theta_0+H/\sqrt{m_n})^{-1} 
- \Sigma_0) H/\sqrt{m_n}]
&=&
\mathcal O\left(\frac{1}{m_n}\right). 
\end{eqnarray*}
By Lemma \ref{ggm.error} and the above calculations of $i,ii,iii$ it then follows
\begin{eqnarray*}
&&
\mathbb E _{\Theta_0}\left( \frac{p_{\Theta_0 + H/\sqrt{m_n}}(x)}{p_{\Theta_0}(x)} - 1 - \textrm{tr}(s_{\Theta_0}(X)H)/\sqrt{m_n}  \right)^2
\\
&&=
\underbrace{
\left(
1 + \frac{(\xi_1^T \Theta_0 \xi_2)^2/(4\sigma_{}^2 m_n)}{1+ 2\xi_1^T \Theta_0 \xi_2/(\sigma_{} m_n)}
 \right)^{n}
-1}_{I}\\
&& \;\;+\;  \frac{n}{ m_n} -\mathcal O\left(\frac{n}{ m_n}\right). 
\end{eqnarray*}
For the first term, by Lemma \ref{lim1} we have that $I=\mathcal O(\delta_n)$. 
Hence we conclude that
\begin{eqnarray*}
&&
\mathbb E_{\Theta_0} \left( \frac{p_{\Theta_0 + H/\sqrt{m_n}}(x)}{p_{\Theta_0}(x)} - 1 - \textrm{tr}(s_{\Theta_0}(X)H) /\sqrt{m_n}  \right)^2
\\
&& =\mathcal O(\delta_n) + \mathcal O\left(\frac{n}{m_n}\right) =\mathcal O(\delta_n).
\end{eqnarray*} 
\end{proof}

\section{Additional proofs for Section 17
}
\label{sec:B}
In this section, we give Lemmas \ref{dens}, \ref{remainder}. 

\begin{lemma}
\label{dens}
Let $Z\in\mathbb R^2$ be $\mathcal N(\mu,\Sigma)$-distributed, where 
\[
\mu = \left(
\begin{array}{c}
\mu_1
\\
\mu_2
\end{array}
\right),
\Sigma = 
\left(
\begin{array}{cc}
\sigma_{11} &\sigma_{12}\\
\sigma_{12} & \sigma_{22}
\end{array}
\right).
\]
Suppose that $\mu_2 = -\sigma_{22}/2$.
Let $Y\in \mathbb R^2$ be $\mathcal N(\mu+a,\Sigma)$-distributed, with
\[
a = \left(
\begin{array}{c}
\sigma_{12}
\\
\sigma_{22}
\end{array}
\right).
\]
Let $\phi_Z$ be the density of $Z$ and $\phi_Y$ be the density of $Y.$
Then we have the following equality for all $z=(z_1,z_2)\in\mathbb R^2$:
$$\phi_Z(z)e^{z_2} = \phi_Y(z).$$
\end{lemma}

\begin{proof}[Proof of Lemma \ref{dens}]
The density of $Z$ is 
$$\phi_Z(z) = \frac{1}{2\pi \sqrt{\textrm{det}(\Sigma)}} e^{-\frac{1}{2}(z-\mu)^T \Sigma^{-1}(z-\mu)}.$$
It holds that
$$\Sigma^{-1}a=(0,1)^T.$$
Then
$$\frac{1}{2}(z-\mu)^T \Sigma^{-1}(z-\mu) = \frac{1}{2}(z-\mu-a)^T \Sigma^{-1}(z-\mu-a) + a^T \Sigma^{-1}(z-\mu)-\frac{1}{2}a^T \Sigma^{-1}a.$$
We also have
$$a^T \Sigma^{-1}(z-\mu)  -\frac{1}{2}a^T \Sigma^{-1}a = (0,1)^T (z-\mu) - \frac{1}{2}(0,1)^T a = z_2 - \mu_2 - \frac{1}{2}\sigma_{22}=z_2.$$
\end{proof}


\begin{lemma}\label{eigen}
Let $\mu$ and $\Sigma$ be defined as follows
\begin{eqnarray*}
\mu=\left(
\begin{array}{c}
-\frac{v_{12}}{\sqrt{v_{11}}}\\
-\frac{v_{22}}{2}
\end{array}
\right)
, \;\;\;\;\; \Sigma = 
\left(
\begin{array}{cc}
1 & \frac{v_{12}}{\sqrt{v_{11}}}\\
\frac{v_{12}}{\sqrt{v_{11}}} & v_{22}
\end{array}
\right).
\end{eqnarray*}
Suppose that $V_\beta=\mathcal O(1), 1/V_\beta=\mathcal O(1)$ and $\Lambda_{\max} (I_\beta )= \mathcal O(1)$. (The relationship between these quantities and the $v_{ij}$'s is given in the proof of Theorem \ref{lecam}).
Then 
$$\|\mu\|_2^2 = \mathcal O(1) \;\;\;\;\textrm{ and }\;\;\;\;\Lambda_{\max}(\Sigma)=\mathcal O(1).$$
\end{lemma}
\begin{proof}
First observe that 
\begin{eqnarray*}
v_{12}^2 &=&(\mathbb E_\beta l_\beta h^T s_\beta)^2 \leq \mathbb E_\beta l_\beta^2 \mathbb E_\beta (h^Ts_\beta)^2 
\\
&=& V_\beta h\mathbb E_\beta s_\beta s_\beta^T h
\leq V_\beta\Lambda_{\max}(\mathbb E_\beta s_\beta s_\beta^T) h^Th.
\end{eqnarray*}
Then by assumption $\Lambda_{\max}(\mathbb E_\beta s_\beta s_\beta^T) = \mathcal O(1)$, $V_\beta=\mathcal O(1)$ and since $h^Th = \mathcal O(1),$ we have that
$(\mathbb E_\beta l_\beta h^T s_\beta)^2 = \mathcal O(1).$
Also observe that $v_{22}=h^T I_\beta h \leq \Lambda_{\max}(I_\beta) h^Th =\mathcal O(1)$ by assumption  $\Lambda_{\max}(I_\beta)=\mathcal O(1)$.
\\
Then, and by $1/V_\beta=\mathcal O(1)$, it follows that
$$\|\mu\|_2^2 = v_{12}^2/v_{11} + v_{22}^2/4 = (P_\beta l_\beta h^Ts_\beta)^2 / V_\beta + (h^T I_\beta h)^2/4 = \mathcal O(1).$$
We proceed to check that the eigenvalues of $\Sigma$ are bounded. We have
$$\lambda_{1,2} = \frac{1+v_{22}  \pm \sqrt{D} }{2},$$
where $D=(1+v_{22})^2 -4(v_{22} - v_{12}^2/v_{11})=(1-v_{22})^2 + 4 v_{12}^2/v_{11}.$
Clearly, $D\geq 0 $, and as above, one sees that $D=\mathcal O(1).$  Hence also $\Lambda_{\max}(\Sigma) = \mathcal O(1).$
\end{proof}

\begin{lemma} \label{remainder}
Suppose that 
\begin{eqnarray*}
U_n&\sim&
\mathcal N\left(
\left(
\begin{array}{c}
-\frac{v_{12}}{\sqrt{v_{11}}}\\
-\frac{v_{22}}{2}
\end{array}
\right),
\left(
\begin{array}{cc}
1 & \frac{v_{12}}{\sqrt{v_{11}}}\\
\frac{v_{12}}{\sqrt{v_{11}}} & v_{22}
\end{array}
\right) \right).
\end{eqnarray*}
Suppose that $V_\beta=\mathcal O(1), 1/V_\beta=\mathcal O(1)$, $\Lambda_{\max} (I_\beta )= \mathcal O(1)$ and $h\in\Theta$. (The relationship between these quantities and the $v_{ij}$'s is given in the proof of Theorem \ref{lecam}).\\
Then it holds  that
$$\lim_{M\rightarrow \infty}\lim_{n\rightarrow \infty} \mathbb E \min(0,M-e^{U_{n,2}})=0.$$
\end{lemma}
 
\begin{proof}[Proof of Lemma \ref{remainder}]
We have 
\begin{eqnarray*}
\mathbb E \min(0,M-e^{U_{n,2}})
&=&
\int_{-\infty }^{\infty}\int_{-\infty }^{\infty} \min(0,M-e^{x_2}) \phi_{U_n}(x)dx
\\
&=&
\int_{-\infty }^{\infty}\int_{\log M }^{\infty} M-e^{x_2} \phi_{U_n}(x)dx
\\
&=&
M P(U_{n,2} > \log M) - P(Y_{n,2} > \log K),
\end{eqnarray*}
where  we applied Lemma \ref{dens} and denoted
\[
Y\sim\mathcal N\left(
\left(
\begin{array}{c}
0\\
{v_{22}/2}
\end{array}
\right),
\left(
\begin{array}{cc}
1 & \frac{v_{12}}{\sqrt{v_{11}}}\\
\frac{v_{12}}{\sqrt{v_{11}}} & v_{22}
\end{array}
\right) 
 \right).
\]
Thus we have 
$$P(Y_{n,2} > \log M) \leq e^{-(\log M)^2/2},$$
and
\begin{eqnarray*}
&&P(U_{n,2} > \log M) =P(U_{n,2}+v_{12}/\sqrt{v_{11}} > \log M+v_{12}/\sqrt{v_{11}}) 
\\
&&\leq e^{-(\log M+v_{12}/\sqrt{v_{11}})^2/2}.
\end{eqnarray*}
Now observe that 
\begin{eqnarray*}
v_{12}^2 &=& (\mathbb E_\beta l_\beta h^T s_\beta)^2 \leq \mathbb E_\beta l_\beta^2 \mathbb E_\beta (h^Ts_\beta)^2 \\
&=& V_\beta h\mathbb E_\beta s_\beta s_\beta^T h
\leq V_\beta\Lambda_{\max}(\mathbb E_\beta s_\beta s_\beta^T) h^Th.
\end{eqnarray*}
Then by assumption $\Lambda_{\max}(\mathbb E_\beta s_\beta s_\beta^T) = \mathcal O(1)$, $V_\beta=\mathcal O(1)$ and since $h^Th = \mathcal O(1),$ we have that
$(\mathbb El_\beta h^T s_\beta)^2 = \mathcal O(1).$
Hence $v_{12}=\mathcal O(1).$
We also have by the assumption $1/V_{\beta} = \mathcal O(1)$ that $1/\sqrt{v_{11}}=\mathcal O(1).$
Hence we can conclude that $v_{12}/\sqrt{v_{11}}\leq L$ for some constant $L>0.$
Now without loss of generality, choose $M$ such that $\log M>2L.$ Then we have
$\log M+v_{12}/\sqrt{v_{11}} \geq \log M /2.$ Thus we obtain
\begin{eqnarray*}
\mathbb E \min(0,M-e^{U_{n,2}})
&=&
M P(U_{n,2} > \log M) - P(Y_{n,2} > \log K)
\\
&\leq &
M e^{-(\log M+v_{12}/\sqrt{v_{11}})^2/2} +e^{-(\log M)^2/2}  \\
&\leq &
Me^{-(\log M)^2/8} + e^{-(\log M)^2/2}.
\end{eqnarray*}
Finally, taking the limits we obtain
$$\lim_{M\rightarrow \infty }\lim_{n\rightarrow \infty}\mathbb E \min(0,M-e^{U_{n,2}})
 \leq \lim_{M\rightarrow \infty }\lim_{n\rightarrow \infty} Me^{-(\log M)^2/2} + e^{-(\log M)^2/8} =0.$$

\end{proof}